\documentclass[11pt,reqno]{amsart}
\usepackage{tikz}
\usepackage{graphicx}
\usepackage{subcaption}
\usepackage{epstopdf}
\usepackage{epsfig}
\usepackage{hyperref}
\usepackage{amsmath}
\usepackage{math}
\graphicspath{{./figures/}}
\usepackage{tikz}
\usetikzlibrary{shapes,arrows}
\tikzstyle{decision} = [diamond, draw, fill=blue!20, 
    text width=4.5em, text badly centered, node distance=3cm, inner sep=0pt]
\tikzstyle{block} = [rectangle, draw, fill=blue!20, 
    text width=5em, text centered, rounded corners, minimum height=4em]
\tikzstyle{line} = [draw, -latex']
\tikzstyle{cloud} = [draw, ellipse,fill=red!20, node distance=3cm,
    minimum height=2em]
\usetikzlibrary{positioning}
\tikzset{main node/.style={circle,fill=blue!20,draw,minimum size=1cm,inner sep=0pt},} 

\newcommand{\la}{\langle}
\newcommand{\ra}{\rangle}

\newcommand{\ts}{\mathsf{T}}

\newcommand{\ba}{\begin{array}}
\newcommand{\ea}{\end{array}}
\newcommand{\be}{\begin{equation}}
\newcommand{\ee}{\end{equation}}
\newcommand{\bea}{\begin{eqnarray}}
\newcommand{\eea}{\end{eqnarray}}
\newcommand{\beaa}{\begin{eqnarray*}}
\newcommand{\eeaa}{\end{eqnarray*}}

\def\div{\mathbf{div}}

\def\hR{\mathbb{R}}

\def\q{\quad}

\def\pa{\partial}

\def\qed{ \hfill \vrule width.25cm height.25cm depth0cm\smallskip}

\newcommand{\basa}{\begin{assumption}}
\newcommand{\easa}{\end{assumption}}

\newcommand{\bas}{\begin{assum}}
\newcommand{\eas}{\end{assum}}

\newcommand{\kl}{\mathrm D_{\mathrm{KL}}}

\def\1{{\bf 1}}

\def\:{\!:\!}

at 9pt

\begin{document}
\title[Time dependent information dissipation]{Fisher information dissipation for time inhomogeneous stochastic differential equations}
\author[Feng]{Qi Feng}
\address{Department of Mathematics, Florida State University, Tallahassee, FL 32306.}
\email{qfeng2@fsu.edu}
\author[Zuo]{Xinzhe Zuo}
\address{Department of Mathematics, University of California, Los Angeles, CA 90095.}
\email{zxz@math.ucla.edu}
\author[Li]{Wuchen Li}
\address{Department of Mathematics, University of South Carolina, Columbia, SC 29208.}
\email{wuchen@mailbox.sc.edu}
\thanks{Q. Feng is partially supported by the National Science Foundation under grant DMS-2306769. X. Zuo and W. Li are supported by AFOSR MURI FA9550-18-1-0502, AFOSR YIP award No. FA9550-23-1-008, and NSF RTG: 2038080.}

\begin{abstract}
We provide a Lyapunov convergence analysis for time-inhomogeneous variable coefficient stochastic differential equations (SDEs). Three typical examples include overdamped, irreversible drift, and underdamped Langevin dynamics. We first formula the probability transition equation of Langevin dynamics as a modified gradient flow of the Kullback-Leibler divergence in the probability space with respect to time-dependent optimal transport metrics. This formulation contains both gradient and non-gradient directions depending on a class of time-dependent target distribution. We then select a time-dependent relative Fisher information functional as a Lyapunov functional. We develop a time-dependent Hessian matrix condition, which guarantees the convergence of the probability density function of the SDE. We verify the proposed conditions for several time-inhomogeneous Langevin dynamics. For the overdamped Langevin dynamics, we prove the $O(t^{-1/2})$ convergence in $L^1$ distance for the simulated annealing dynamics with a strongly convex potential function. For the irreversible drift Langevin dynamics, we prove an improved convergence towards the target distribution in an asymptotic regime. We also verify the convergence condition for the underdamped Langevin dynamics. Numerical examples demonstrate the convergence results for the time-dependent Langevin dynamics.
\end{abstract}

\maketitle

\section{Introduction}
Time-inhomogeneous (time-dependent) stochastic dynamics are an essential class of equations, which are widely used in modeling engineering problems, designing Bayesian sampling algorithms of a target distribution, and approximating global optimization problems with applications in machine learning \cite{CH87,GH86,MCJ,TangZhou2021}. An important example is the stochastic dynamics from the simulated annealing method \cite{Cer85,KGV83}. It finds a global minimizer of a function with a time-dependent diffusion constant. The diffusion constant converges to zero when time approaches infinity. Eventually, the solution of stochastic dynamics will be a global minimizer of such a function. In recent years, general time-dependent stochastic dynamics have also been designed to maintain desired invariant distributions, such as the nonreversible Langevin sampler \cite{DLP,DPZ, ZMS}. The discretized stochastic dynamics are useful stochastic algorithms in practice. In these studies, a key consideration is the rate at which these stochastic dynamics converge to their stationary distributions. The convergence analysis can be leveraged to design and refine sampling algorithms that exhibit faster convergence.

This paper presents the convergence analysis for time-inhomogeneous stochastic dynamics, including three equations: overdamped, nonreversible drift, and underdamped Langevin dynamics. We use the time-dependent Fisher information as a Lyapunov functional to study convergence behaviors of the probability density functions of stochastic dynamics. Applying some convex analysis tools in generalized Gamma calculus \cite{FengLi, FengLi2021}, we derive a time-dependent Hessian matrix condition to characterize convergence behaviors of time-dependent stochastic dynamics in Theorem \ref{thm: Fisher information decay az}. 
Lastly, we present three examples for the proposed convergence analysis. We first study the Lyapunov analysis of time-dependent overdamped Langevin dynamics based on the continuous limit of simulated annealing algorithms. When the potential function is strongly convex, we show that the Fisher information converges at a rate of $O(\frac{1}{t})$ when the diffusion coefficient is $O(\frac{1}{\log t})$, where $t>0$ is a time variable. We then analyze the time-dependent Langevin dynamics with nonreversible drift and a nondegenerate diffusion matrix. We prove the speed-up of the convergence near the global minimizer of the potential function. Lastly, we study the convergence analysis for the inhomogeneous underdamped Langevin dynamics. Several numerical experiments are provided to justify our theoretical results. 

In literature, the convergence study of time-dependent stochastic dynamics is an emerging area for stochastic algorithms in machine learning \cite{Chizat2022}. In this direction, the continuous-time simulated annealing based on time-dependent overdamped Langevin dynamics was first studied in \cite{GH86}. It was shown in \cite{CH87,GH86} that the correct order of diffusion constant for the time-dependent Lanvegin dynamics to converge to the global minimum of the objective function $V$ is of order $(\log t)^{-1}$. Recent works \cite{Chizat2022, M18, MSTW18, TangZhou2021} have shown polynomial convergence in both $L^1$ distance and tail probability. The state-dependent overdamped Langevin dynamics version of simulated annealing was studied in \cite{FQG97, GXZ20}.

Compared to previous results, we focus on the convergence analysis using time-dependent Fisher information functional for general time-inhomogeneous Langevin dynamics. This allows us to derive a Hessian matrix condition in establishing the convergence rates. As a special example, in time-dependent overdamped Langevin dynamics, we obtain a $O(t^{-\frac{1}{2}})$ convergence in $L^1$ distance under the strongly convex assumption of the potential function. On the other hand, analysis on the time-dependent Fisher information dissipation in nonreversible and underdamped Langevin dynamics is still a work in progress. This paper initializes the convergence analysis of these stochastic dynamics.  

The paper is organized as follows. We formulate the main results in sections \ref{sec2} and \ref{sec3}. Using the decay of a time-dependent Fisher information functional, we state the condition for the convergence of general stochastic differential equations. We then present several examples of convergence analysis. Section \ref{sec4} provides the detailed convergence analysis for simulating annealing dynamics with a strongly convex potential function. Section \ref{sec5} presents the convergence analysis for the Langevin dynamics with an irreversible drift and nondegenerate diffusion matrices. Section \ref{sec6} shows the convergence analysis of underdamped Langevin dynamics. Several numerical examples are provided to verify the convergence analysis.  

\section{Setting}\label{sec2}
In this section, we provide the main setting of this paper. We consider the general time-dependent stochastic differential equation. We also formulate its Fokker-Planck equation, for which we develop a time-dependent decomposition of gradient and non-gradient directions in the probability density space. We then introduce the time-dependent relative Fisher information functional, which will be used in the convergence analysis of the solution of the Fokker-Planck equation.

\subsection{General setting}
Consider It\^o type stochastic differential equations (SDEs) in $\mathbb R^{n+m}$ as follows:
\bea \label{SDE setting}
dX_t = b(t,X_t)dt+\sqrt{2}a(t,X_t)dB_t.
\eea 
For $m,n\in\mathbb Z_+$, we assume that $a(t,x)\in \mathbb C^{\infty}(\hR_+\times\mathbb{R}^{n+m}; \mathbb{R}^{(n+m)\times n})$ is a time dependent diffusion matrix, $b(t,x)\in \mathbb C^{\infty}(\hR_+\times\hR^{n+m}; \hR^{n+m})$ is a time dependent vector field, and $B_t$ is a standard $\mathbb R^n$-valued Brownian motion. For the time-dependent diffusion matrix $a(t,x)$, we denote $n$ as the rank of $a(t,x)$, $a(t,x)^{\ts}$ as the transpose of matrix $a(t,x)$, and $a(t,x)a(t,x)^{\ts}$ as the standard matrix multiplication. For $i=1,\cdots, n$,  we denote $a^{\ts}_i=(a(t,x)^{\ts})_i$ as the row vectors of $a(t,x)^{\ts}$, and $a_{\cdot i}=a(t,x)_{\cdot i}$ as the column vectors of $a(t,x)$, i.e. $a^{\ts}_{i\hat i}=a_{\hat ii}$, for $\hat i=1,\cdots, n+m$. 
For each row vector  $a^{\ts}_i\in \mathbb R^{n+m}$ with $i=1,\cdots,n$, we denote $\mathbf A_i(t,x):=\sum_{\hat i=1}^{n+m}a^{\ts}_{i\hat i}\frac{\partial}{\partial x_{\hat i}}$ as the corresponding vector fields for each row vector $a^{\ts}_i$. Similarly, we denote $\mathbf A_0(t, x):=\sum_{\hat i=1}^{n+m}b_{\hat i}(t,x)\frac{\partial}{\partial x_{\hat i}}$ as the vector field associated to the drift term $b$. In this paper, we assume H\"ormander like conditions \cite{hormander} for the vector fields such that the probability density function $p(t,x)$ for the diffusion process $X_t$ exists and is smooth. In the current time inhomogeneous setting, such conditions may include the H\"ormander condition \cite{Cattiaux}, weak H\"ormander condition \cite{hopfner2017}, the UFG (uniformly finitely generated) condition \cite{Cass2021}, and the restricted H\"ormander’s hypothesis \cite{Chaleyat-Maurel}.
Denote $[\mathbf A_i(t,x),\mathbf A_j(t,x)]$, for $i,j\in\{ 0,\cdots,n\}$, as the Lie bracket of two vector fields. The H\"ormander type condition means that the Lie algebra generated by $\mathbf A_i(t,x)$, $1\le i\le n$, and $\mathbf A_0(t,x)+\frac{\partial}{\partial t}$ has full rank. For all $(t,x)$, we assume
\begin{equation*}
    \mathrm{Span}~ \mathrm{Lie}\Big\{\mathbf A_0(t,x)+\frac{\partial}{\partial t}, \mathbf A_1(t, x),\cdots, \mathbf A_n(t,x) \Big\}=\hR^{n+m}.
\end{equation*}

Under the above assumptions, $p(t,x)$, which is the probability density function for $X_t$, satisfies the following Fokker-Planck equation of the SDE \eqref{SDE setting},
\begin{equation}\label{FPE}
\partial_tp(t,x)=-\nabla\cdot(p(t, x)b(t, x))+\sum_{i=1}^{n+m}\sum_{j=1}^{n+m}\frac{\partial^2}{\partial x_i\partial x_j} \Big(\Big(a(t,x)a(t,x)^{\ts}\Big)_{ij}p(t,x)\Big),
\end{equation}
with the following initial condition 
\begin{equation*}
p_0(x)=p(0,x), \quad p_0\in\mathcal P.
\end{equation*}
Here we denote $\mathcal P$ as a probability density space supported on $\mathbb R^{n+m}$, defined as
\begin{equation*}
    \mathcal{P}=\Big\{p\in L^1(\mathbb R^{n+m})\colon \int_{\mathbb{R}^{n+m}} p(x)dx=1,\quad p\geq 0\Big\}. 
\end{equation*}

\subsection{Time dependent Gradient and Non-gradient decompositions}
To study the convergence of the probability density function $p(t,x)$ towards the invariant distribution or the reference distribution $\pi(t,x)$. We make the following decomposition of Fokker-Planck equation \eqref{FPE}. We assume that $\pi(t,x)\in \mathbb C^{2,2}(\mathbb R_+\times\mathbb R^{n+m};\mathbb R)$ has an explicit formula. If $\pi(t,x)$ indeed solves the equation, 
\begin{equation*}
    -\nabla_x\cdot(\pi(t, x)b(t, x))+\sum_{i=1}^{n+m}\sum_{j=1}^{n+m}\frac{\partial^2}{\partial x_i\partial x_j}\Big(\Big(a(t,x)a(t,x)^{\ts}\Big)_{ij} \pi(t,x)\Big)=\partial_t\pi(t,x)=0,
\end{equation*}
then $\pi(t,x)=\pi(x)$ is the invariant distribution. Otherwise, we use $\pi(t,x)$ as a reference distribution for the probability density function $p(t,x)$ at each time $t$.

The Fokker-Planck equation \eqref{FPE} can be decomposed into a gradient and a non-gradient part by introducing a non-gradient vector field $\gamma(t,x):\hR_+\times\hR^{n+m}\rightarrow \hR$. The same decomposition has been used in \cite{FengLi2021}, where the non-gradient vector field $\gamma(x)$ does not depend on the time variable. For self-consistency, we show the decomposition below for the time-dependent vector fields. We first introduce the following notation, for $t\ge 0$, 
\begin{equation}\label{div of matrix}
 \nabla\cdot \Big(a(t,x)a(t,x)^{\ts}\Big)=\Big(\sum_{j=1}^{n+m}\frac{\partial}{\partial x_j}\Big(a(t,x)a(t,x)^{\ts}\Big)_{ij}\Big)_{i=1}^{n+m}\in \mathbb R^{n+m}.
\end{equation}
We then have the following decomposition. 
\begin{proposition}[Decomposition]\label{prop: decomposition}
For the Fokker-Planck equation \eqref{FPE} and a reference distribution density function $\pi(t,x)$, we define a non-gradient vector field $\gamma(t, x): \mathbb R_+\times \mathbb R^{n+m}\rightarrow \mathbb R^{n+m}$ as
 \begin{equation*} 
\gamma(t, x):=  \Big( a(t,x)a(t,x)^{\ts}\Big)\nabla\log \pi(x)-b(x)+\nabla\cdot \Big(a(t,x)a(t,x)^{\ts}\Big).
 \end{equation*}
Then the Fokker-Planck equation \eqref{FPE} {is equivalent to the following equation}:
\begin{equation}\label{FPE1}
\begin{aligned} 
\pa_t p(t,x)=\nabla\cdot \Big(p(t,x)\Big(a(t,x)a(t,x)^{\ts}\Big)\nabla \log \frac{p(t,x)}{\pi(t,x)}\Big)+\nabla\cdot(p(t,x)\gamma(t,x)).
\end{aligned}
\end{equation}
\end{proposition}
\begin{proof}
The proof is based on a direct calculation. For simplicity of notations, we skip the variables $(t,x)$ below. We note
\begin{equation*}
\begin{aligned}
   &\sum_{i=1}^{n+m}\sum_{j=1}^{n+m}\frac{\partial^2}{\partial x_i\partial x_j} \Big(\Big(aa^{\ts}\Big)_{ij}p\Big)
      =\sum_{i=1}^{n+m}\frac{\partial}{\partial x_i}\sum_{j=1}^{n+m}\frac{\partial}{\partial x_j}\Big(\Big(aa^{\ts}\Big)_{ij}p(t,x)\Big)\\
            =&\sum_{i=1}^{n+m}\frac{\partial}{\partial x_i}\sum_{j=1}^{n+m}\Big(\frac{\partial}{\partial x_j}\Big(aa^{\ts}\Big)_{ij}p+\Big(aa^{\ts}\Big)_{ij}\frac{\partial}{\partial x_j}p\Big)\\
 =&\sum_{i=1}^{n+m}\frac{\partial}{\partial x_i}\Big(p\frac{\partial}{\partial x_j}\sum_{j=1}^{n+m}\Big(aa^{\ts}\Big)_{ij}\Big)+\sum_{i,j=1}^{n+m}\frac{\partial}{\partial x_i}\Big(\Big(aa^{\ts}\Big)_{ij}\frac{\partial}{\partial x_j}p\Big)\\
 =&\sum_{i=1}^{n+m}\frac{\partial}{\partial x_i}\Big(p\frac{\partial}{\partial x_j}\sum_{j=1}^{n+m}\Big(a(x)a(x)^{\ts}\Big)_{ij}\Big)+\sum_{i,j=1}^{n+m}\frac{\partial}{\partial x_i}\Big(\Big(aa^{\ts}\Big)_{ij}p\frac{\partial}{\partial x_j}\log p\Big),
\end{aligned}
\end{equation*}
where we used the fact $\frac{\partial}{\partial x_j}p=p\frac{\partial}{\partial x_j}\log p$. From the definition of $\gamma$ and the above observation, we show that the R.H.S. of the Fokker-Planck equation \eqref{FPE} can be written as 
\begin{equation*}
   \begin{aligned}
&-\nabla\cdot(pb)+\sum_{i=1}^{n+m}\sum_{j=1}^{n+m}\frac{\partial^2}{\partial x_i\partial x_j} \Big((aa^{\ts})_{ij}p\Big)\\
=&-\nabla\cdot(p b)+\sum_{i,j=1}^{n+m}\frac{\partial}{\partial x_i}\Big(p\frac{\partial}{\partial x_j}(aa^{\ts})_{ij}\Big)+\nabla\cdot(p(aa^{\ts})\nabla\log p)\\
=&-\nabla\cdot(p b)+\sum_{i,j=1}^{n+m}\frac{\partial}{\partial x_i}\cdot\Big(p\frac{\partial}{\partial x_j}(aa^{\ts})_{ij}\Big)+\nabla\cdot(p(aa^{\ts})\nabla\log\pi)\\
&-\nabla\cdot(p(aa^{\ts})\nabla\log\pi)+\nabla\cdot(p(aa^{\ts})\nabla\log p)\\
=&\nabla\cdot\Big(p (-b+\nabla\cdot(aa^{\ts})+aa^{\ts}\nabla\log\pi)\Big) +\nabla\cdot(paa^{\ts}\nabla\log\frac{p}{\pi})\\
=&\nabla\cdot(p \gamma) +\nabla\cdot\Big(p(aa^{\ts})\nabla\log\frac{p}{\pi}\Big),
\end{aligned}
\end{equation*}
where we used the definition of $\gamma$ and the fact that $\nabla\log\frac{p}{\pi}=\nabla\log p-\nabla\log \pi$. 
  \qed 
\end{proof}
\begin{remark}
The time-dependent hypoelliptic operator $-\nabla\cdot(p aa^{\ts}\nabla)$ is a modified gradient operator in the Wasserstein-2 type metric space \cite{Villani2009_optimal}. And the vector field $\gamma$ is not a gradient direction \cite{Villani2006_hypocoercivity}. Interested readers may look for relevant discussions in the time-homogenous case \cite{FengLi2021}.
\end{remark}

\subsection{Lyapunov functionals}
To measure the distance between $p(t,x)$ and $\pi(t,x)$, as well as the corresponding convergence rate towards $\pi(t,x)$, we define the Kullback–Leibler (KL) divergence
\bea
\kl(p(t,\cdot) \|\pi(t,\cdot)):=\int_{\mathbb R^{n+m}} {p} (t,x)\log \frac{{p} (t,x)}{\pi(t,x )}dx.
\eea
For $t\ge 0$, and a diffusion matrix $a(t,x)\in \mathbb C^{\infty}(\hR_+\times\mathbb{R}^{n+m}; \mathbb{R}^{(n+m)\times n})$ 
associated with SDE \eqref{SDE setting} with rank $n$, we introduce a complementary matrix, defined as, 
\bea\label{defn: z} 
z(t,x)\in \mathbb C^{\infty}(\mathbb R_+\times \mathbb R^{n+m}; \mathbb R^{(n+m)\times m}),
\eea  
such that, for all $t\ge 0$, 
\bea \label{full rank condition}
\text{Rank}\Big(a(t,x)a(t,x)^{\ts}+z(t,x)z(t,x)^{\ts}\Big) =n+m,\quad \text{for all} ~~x\in \hR^{n+m}.
\eea 
Adapted from the previous notation, we denote $a^{\ts}$ and $z^{\ts}$ as the transpose of  matrices $a(t,x)$ and $z(t,x)$. We denote $\{ a^{\ts}_i\}_{i=1}^n$ and $\{ z^{\ts}_j\}_{j=1}^m$ as the row vectors of $a^{\ts}$ and $z^{\ts}$. The condition \eqref{full rank condition} means that the linear span of the row vectors $\{ a^{\ts}_i\}_{i=1}^n$ and $\{ z^{\ts}_j\}_{j=1}^m$ generate the entire space $\hR^{n+m}$ for all $t\ge 0$. Furthermore, to ensure that the Bochner's formula \cite[Theorem 1]{FengLi2021} holds, we assume that, for $0\le k\le m$, $0\le i\le n$, 
\begin{equation}\label{condition: bochner}
    \mathbf Z_k(t,x)\mathbf A_i(t,x)\in \mathrm{Span}\{\mathbf A_j(t,x), 0\le j\le n \},\quad \text{for all}\quad t\ge 0,\quad \text{and}\quad x\in\mathbb R^{n+m},
\end{equation}
where we denote $\mathbf Z_k(t,x)$ as the corresponding vector field for each row vector $z^{\ts}_k$.

For a smooth function $f\in \mathbb C^{\infty}(\hR^{n+m})$, we denote the column vector $\nabla f$ as below,
\begin{equation}
\nabla f(x)= \Big(\frac{\partial f}{\partial x_1,}\cdots, \frac{\partial f}{\partial x_{n+m}} \Big)^{\ts}=  \sum_{i=1}^{n+m}\frac{\partial f}{\partial x_i}\frac{\partial}{\partial x_i}.
\end{equation} 
 We keep the following notation throughout the paper. A standard multiplication of a row vector and a column vector has the following form, 
\begin{equation} 
a^{\ts}_k\nabla f =\sum_{k'=1}^{n+m}a^{\ts}_{kk'}\frac{\partial f}{\partial x_{k'}}.
\end{equation} 
Similarly, we denote 
\bea 
a^{\ts}_k\nabla a^{\ts}_i\nabla f= a^{\ts}_k(\nabla a^{\ts}_i)\nabla f=\sum_{k',i'=1}^{n+m} a^{\ts}_{kk'}\frac{\partial a^{\ts}_{ii'}}{\partial x_{k'}}\frac{\partial f}{\partial x_{i'}},
\eea 
where the gradient is always applied to the function next to it. Given matrices $a(t,x)$, $z(t,x)$, and the reference measure $\pi(t,x)$ as above, we introduce the following relative Fisher information functionals as our Lyapunov functionals. Denote $\la u, v\ra=\sum_{i=1}^{n+m}u_iv_i$, for any vectors $u, v\in\mathbb{R}^{n+m}$. 
\begin{definition}[Fisher information functionals]
Define a functional $\mathcal{I}_a\colon \mathcal{P}\rightarrow\mathbb{R}_+$ as
\bea\label{defn: fisher}
\mathrm I_{a}({p} (t,\cdot)\|\pi(t,\cdot)):=\int_{\mathbb R^{n+m}} \Big\la \nabla\log\frac{{p} (t,x)}{\pi(t,x)},a(t,x)a(t,x)^{\ts}\nabla\log\frac{{p} (t,x)}{\pi(t,x)}\Big\ra {p} (t,x) dx,
\eea
Define an auxiliary functional $\mathcal{I}_z\colon \mathcal{P}\rightarrow\mathbb{R}_+$ as 
\bea\label{defn: fisher z}
\mathrm I_{z}({p} (t,\cdot)\|\pi(t,\cdot)):=\int_{\mathbb R^{n+m}} \Big\la \nabla\log\frac{{p} (t,x)}{\pi(t,x)},z(t,x)z(t,x)^{\ts}\nabla\log\frac{{p} (t,x)}{\pi(t,x)}\Big\ra {p} (t,x) dx.
\eea 
\end{definition}

\section{Time dependent Fisher information decay}\label{sec3}
In this section, we present the main theoretical analysis. We use the time-dependent original and auxiliary Fisher information functionals as Lyapunov functionals for the convergence of the Fokker-Planck equation in Theorem \ref{thm: Fisher information decay az}. 

We shall derive the dissipation of KL divergence and Fisher information along time-inhomogenous equations. We first show the relation between the KL divergence and the Fisher information functional in this time-dependent setting.
\begin{proposition}
For $t\ge 0$, we have
	\begin{equation}\label{KL decay}
	\begin{split}
	\pa_t \kl(p \|\pi)=&-\int_{\mathbb R^{n+m}} \la \nabla\log\frac{p(t,x) }{\pi(t,x)},aa^{\ts}\nabla\log\frac{p(t,x) }{\pi(t,x)}\ra p(t,x)  dx\\
 &-\int_{\mathbb R^{n+m}} \mathcal{R}(t,x,\pi)p(t,x)  dx,
 \end{split}
	\end{equation}
	where we define the correction term $\mathcal{R}(t,x,\pi):\mathbb R_+\times \mathbb R^{n+m}\times \mathcal P\rightarrow \mathbb R$ as below,  
\begin{equation}\label{correction R}
\mathcal{R}(t,x, \pi ):=\frac{\pa_t \pi(t,x) -\nabla\cdot(\pi(t,x)\gamma(t,x))}{\pi(t,x)}.
\end{equation}
\end{proposition}
\begin{remark}
    Note that if $\pi(t,x)=\pi(x)$ is the invariant measure, we have $\nabla\cdot (\pi \gamma)=0$, hence $\mathcal R(t,x,\pi)=0$. However, in the more general setting, $\nabla\cdot(\pi(t,x)\gamma(t,x))\neq 0$ for a general reference measure $\pi(t,x)$. 
\end{remark}
For simplicity of notation in all proofs, we shall denote $\int_{\mathbb{R}^{n+m}}$ as $\int$. We also skip the variables $(t,x)$ to simplify the notation. 

\begin{proof}
We derive the entropy dissipation as below,
\beaa
\pa_t \kl ({p} \|\pi)&=&\int \pa_t {p}  \log\frac{{p} }{\pi}dx+\int {p}  \pa_t\log{p}  dx-\int {p}  \pa_t\log\pi dx\\
&=&\int \big[\nabla\cdot  ( {p}  \gamma )+\nabla\cdot ({p}  aa^{\ts}\nabla \log \frac{{p} }{\pi} ) \big]\log\frac{{p} }{\pi}dx+\int\pa_t{p}  dx-\int {p}  \pa_t\log\pi dx\\
&=&-\int \la \nabla\log\frac{{p} }{\pi},aa^{\ts}\nabla\log\frac{{p} }{\pi}\ra{p}  dx+\int \nabla\cdot ({p} \gamma)\log\frac{{p} }{\pi}dx-\int {p}  \pa_t\log\pi dx.
\eeaa
Furthermore, we have 
\beaa
 \int \nabla\cdot ({p} \gamma)\log\frac{{p} }{\pi}dx&=&-\int \la \nabla\log\frac{{p} }{\pi},\gamma\ra {p}  dx\\
 &=&-\int \la \nabla {p} ,\gamma\ra dx +\int \la \nabla\log\pi,\gamma\ra {p}  dx\\
 &=&\int(\nabla\cdot \gamma +\la \nabla\log\pi,\gamma\ra){p}  dx\\
 &=&\int \frac{{p} }{\pi}[(\nabla\cdot \gamma)\pi +\la \nabla\pi,\gamma\ra ]dx\\
 &=&\int \frac{\nabla\cdot(\pi\gamma)}{\pi}{p}  dx.
 \eeaa
 Combining the above terms, we complete the proof.\qed 
\end{proof}

\begin{lemma}We first observe the following identity,
	\beaa 
	\pa_t\int \mathcal R {p}  dx =\int\Big[  \pa_t\mathcal R -\la \nabla \mathcal R,\gamma\ra +\nabla\cdot(aa^{\ts}\nabla\mathcal{R})+\la \nabla \mathcal R, aa^{\ts}\nabla\log\pi\ra \Big]{p}  dx.
	\eeaa 
\end{lemma}
\begin{proof}
	We observe that,
	\beaa
	\pa_t \int \mathcal R {p}  dx=\int \pa_t\mathcal R {p}  dx +\int \mathcal R \pa_t{p}  dx.
	\eeaa 
	And
\beaa
\int \mathcal{R}\pa_t{p}  dx&=& \int \mathcal{R}\Big[\nabla\cdot({p}  aa^{\ts}\nabla\log\frac{{p} }{\pi} )+\nabla\cdot({p} \gamma) \Big]dx\\
&=&-\int \la \nabla \mathcal{R}, aa^{\ts}\nabla\log\frac{{p} }{\pi}\ra{p}  dx-\int \la \nabla \mathcal{R},\gamma\ra {p}  dx\\
&=& -\int \la \nabla \mathcal R, aa^{\ts}\nabla \log {p} \ra {p}  dx+\int \la \nabla \mathcal R, aa^{\ts}\nabla \log \pi\ra {p}  dx -\int \la \nabla\mathcal R, \gamma\ra {p}  dx\\
&=& \int \nabla\cdot (aa^{\ts} \nabla \mathcal R) {p}  dx+\int \la \nabla \mathcal R, aa^{\ts}\nabla \log \pi\ra {p}  dx -\int \la \nabla\mathcal R, \gamma\ra {p}  dx.
\eeaa \qed
\end{proof}

\subsection{Fisher information decay}
In this subsection, we first present the Fisher information functional dissipation result. The proof will be postponed to Section \ref{sec fish a}, and Section \ref{sec fish z}. To simplify our notation, we define 
\begin{equation}
    \mathrm I_{a,z}(t)= \mathrm{I}_a(p\|\pi)+\mathrm{I}_z(p\|\pi).
\end{equation}
\begin{theorem}[Fisher decay]\label{thm: Fisher information decay az}
We define $\mathfrak{R}(t,x):\mathbb R_+\times  \mathbb R^{n+m}\rightarrow \mathbb{R}^{(n+m)\times (n+m)}$ as the corresponding time-dependent Hessian matrix function, which is defined in the Appendix \ref{appendix}. Assume that 
	\bea\label{assumption: a z tensor}
	\mathfrak R(t,x)-\frac{1}{2}\pa_t(aa^{\ts}+zz^{\ts})(t,x)\succeq \lambda(t) [aa^{\ts}+zz^{\ts}](t,x),
	\eea 
	for all $x\in\mathbb R^{n+m}$, and for all $t\ge 0$.
	We have 
	\begin{equation*}
    \mathrm I_{a,z}(t)\leq  e^{-2\int_{t_0}^t \lambda(r)dr}\Big(\int_{t_0}^t 2[\mathrm A(r)+\mathrm Z(r)]e^{2\int_{t_0}^r\lambda(\tau)d\tau} dr+\mathrm I_{a,z}(t_0)\Big),
\end{equation*}
where the correction term $\mathcal R(\cdot,\cdot,\cdot)$ is introduced in \eqref{correction R}. And
\beaa
\mathrm A(r) &=& \int_{\mathbb R^{n+m}}\big[
\nabla\cdot(aa^{\ts}\nabla\mathcal{R})+ \la \nabla \mathcal R, aa^{\ts}\nabla\log\pi\ra \big]{p} (r,x) dx,\\
\mathrm Z(r) &=&\int_{\mathbb R^{n+m}}\big[
\nabla\cdot(zz^{\ts}\nabla\mathcal{R})+\la \nabla \mathcal R, zz^{\ts}\nabla\log\pi\ra \big]{p} (r,x) dx.
\eeaa

\end{theorem}
\begin{proof}
From the definition, we have $\mathrm I_{a,z}(t)=\mathrm I_{a,z}({p} \|\pi)=\mathrm I_{a}({p} \|\pi)+\mathrm I_{z}({p} \|\pi)$.	According to Proposition \ref{prop: Fisher information decay} in the next section, and Assumption \eqref{assumption: a z tensor}, we have 
\begin{equation*}
    \pa_t\mathrm I_{a,z}({p} \|\pi)\leq -2\lambda(t) \mathrm{I}_{a,z}({p} \|\pi)+2[\mathrm A(t)+\mathrm Z(t)]. 
\end{equation*}
We next construct a function $Q(t)$, such that 
\beaa 
\pa_t Q(t)+2\lambda(t)Q(t)=2[\mathrm A(t)+\mathrm Z(t)].
\eeaa 
Let $F(t)=-2\int_{t_0}^t \lambda(r) dr $. We obtain $Q(t)=Q(t_0) e^{F(t)}+e^{F(t)}\int_{t_0}^t 2[\mathrm A(s)+\mathrm Z(s)]e^{-F(s)}ds$, which implies 
\beaa 
\pa_t (\mathrm I_{a,z}(t)-Q(t))\le -2\lambda(t)(\mathrm I_{a,z}(t)-Q(t)). 
\eeaa 
From Gronwall's inequality, we have
\beaa 
\mathrm I_{a,z}(t)&\le& Q(t)+(\mathrm I_{a,z}(t_0)-Q(t_0))e^{-2\int_{t_0}^t\lambda(r)dr}\\
&= & Q(t_0)e^{-2\int_{t_0}^t \lambda(r)dr}+e^{\int_{t_0}^\lambda(r)dr}\int_{t_0}^t 2[\mathrm A(r)+\mathrm Z(r)]e^{\int_{t_0}^r2\lambda(\tau)d\tau} dr\\
& &+(\mathrm I_a(t_0)-Q(t_0))e^{-2\int_{t_0}^t\lambda(r)dr}\nonumber \\
&= & e^{-2\int_{t_0}^t \lambda(r)dr}\Big(\int_{t_0}^t 2[\mathrm A(r)+\mathrm Z(r)]e^{2\int_{t_0}^r\lambda(\tau)d\tau} dr+\mathrm I_{a,z}(t_0)\Big).\nonumber
\eeaa
This finishes the proof. \qed
\end{proof}

\subsection{Information Gamma calculus} To derive the dissipation of the Fisher information functional, we first introduce the information Gamma calculus in the current setting. We refer to \cite{FengLi, FengLi2021, BFL2022} for more motivations and detailed discussions on these operators. We follow closely the notations as in \cite[Definition 2]{FengLi2021} below. 
Following the decomposition in Proposition \ref{prop: decomposition}, the diffusion operator $L$ associated with SDE \eqref{SDE setting} is defined in the following form, for smooth function $f:\mathbb R^{n+m}\rightarrow\mathbb R$,  
\begin{equation}\label{operator L}
\begin{aligned}
Lf=\widetilde Lf -\la \gamma(t,x),\nabla f\ra,
\end{aligned}
\end{equation}
where we define the reversible component of the diffusion operator $L$ as below, 
\bea
\widetilde L f&=& \nabla\cdot(a(t,x)a(t,x)^{\ts}\nabla f)+\la a(t,x)a(t,x)^{\ts}\nabla \log\pi(t,x),\nabla f\ra.
\eea
For the diffusion matrix function $a(t,x)$, we construct a matrix $z(t,x)\in\mathbb C^{\infty}(\mathbb R_+\times \mathbb R^{n+m}; \mathbb R^{(n+m)\times m })$ such that conditions \eqref{full rank condition}, \eqref{condition: bochner}, and the H\"ormander condition hold true. We then introduce the following $z$-direction differential operator as
\begin{equation*}\label{z operator}
\widetilde L_z f=\nabla\cdot(z(t,x)z(t,x)^{\ts}\nabla f)+\la z(t,x)z(t,x)^{\ts}\nabla \log\pi(t,x),\nabla f\ra.   
\end{equation*}
The Gamma one bilinear forms for the matrices $a(t,x)$ and $z(t,x)$ are defined as below, $\Gamma_{1},\Gamma_{1}^z\colon \mathbb C^{\infty}(\mathbb{R}^{n+m})\times \mathbb C^{\infty}(\mathbb{R}^{n+m})\rightarrow \mathbb C^{\infty}(\mathbb{R}^{n+m})$ as 
\bea
\Gamma_{1}(f,f)=\la a(t,x)^{\ts}\nabla f, a(t,x)^{\ts}\nabla f\ra_{\hR^n},\quad \Gamma_{1}^z(f,f)=\la z(t,x)^{\ts}\nabla f, z(t,x)^{\ts}\nabla f\ra_{\hR^m}.
\eea
\begin{definition}[Time dependent Information Gamma operators]\label{defn:tilde gamma 2 znew}
Define the following three bi-linear forms: 
\begin{equation*}
 \widetilde\Gamma_{2}, \widetilde\Gamma_{2}^{z,\pi}, \Gamma_{\mathcal{I}_{a,z}}\colon C^{\infty}(\mathbb{R}^{n+m})\times C^{\infty}(\mathbb{R}^{n+m})\rightarrow C^{\infty}(\mathbb{R}^{n+m}).
\end{equation*}
\begin{itemize}
\item[(i)] Gamma two operator: 
\begin{equation*}
\widetilde\Gamma_{2}(f,f)=\frac{1}{2}\widetilde L\Gamma_{1}(f,f)-\Gamma_{1}(\widetilde Lf, f).
\end{equation*}
\item[(ii)] Generalized Gamma $z$ operator:
\beaa
\widetilde\Gamma_2^{z,\pi}(f,f)&=&\quad\frac{1}{2}\widetilde L\Gamma_1^{z}(f,f)-\Gamma_{1}^z(\widetilde Lf,f)\\
&& \label{new term}+\div^{{\pi}}_z\Big(\Gamma_{1,\nabla(aa^{\ts})}(f,f )\Big)-\div^{{\pi} }_a\Big(\Gamma_{1,\nabla(zz^{\ts})}(f,f )\Big).
  \eeaa 
Here $\div^{{\pi} }_a$, $\div^{{\pi} }_z$ are divergence operators defined by
\begin{equation*}
\div^{{\pi}}_a(F)=\frac{1}{{\pi} }\nabla\cdot({\pi} aa^{\ts} F), \quad\div^{{\pi} }_z(F)=\frac{1}{{\pi} }\nabla\cdot({\pi} zz^{\ts}F),
\end{equation*}
for any smooth vector field $F\in \mathbb{R}^{n+m}$, and $\Gamma_{1, \nabla (aa^{\ts})}$, $\Gamma_{1, \nabla (zz^{\ts})}$ are vector Gamma one bilinear forms defined by 
\beaa
\Gamma_{1,\nabla(aa^{\ts)}}(f,f)&=&\la \nabla f,\nabla(aa^{\ts})\nabla f\ra=(\la \nabla f,\frac{\partial}{\partial x_{\hat k}}(aa^{\ts})\nabla f\ra)_{\hat k=1}^{n+m},\\
\Gamma_{1,\nabla(zz^{\ts)}}(f,f)&=&\la \nabla f,\nabla(aa^{\ts})\nabla f\ra=(\la \nabla f,\frac{\partial}{\partial x_{\hat k}}(zz^{\ts})\nabla f\ra)_{\hat k=1}^{n+m}.
\eeaa
\item[(iii)] Irreversible Gamma operator: 
\beaa
\Gamma_{\mathcal{I}_{a}}(f,f)+\Gamma_{\mathcal{I}_{z}}(f,f)&=& (\widetilde Lf+\widetilde L_zf) \la \nabla f,\gamma\ra -\frac{1}{2}\la \nabla \big(\Gamma_1(f,f)+\Gamma_1^z(f,f)\big),\gamma\ra .
\eeaa
\end{itemize}
\end{definition}
	\begin{remark}
 One key difference in the current setting compared to \cite{FengLi2021} is the fact $\nabla\cdot (\pi(t,x)\gamma(t,x))$ can be non-zero. Due to the decomposition of operator $L$ in \eqref{operator L}, the time dependent vector field $\gamma(t,x)$ does not make a difference for the second order operators $\widetilde\Gamma_2$ and $\widetilde\Gamma_2^{z,\pi}$. Thus the Bochner's formula for $\widetilde\Gamma_2$ and $\widetilde\Gamma_2^{z,\pi}$ remains the same as in \cite{FengLi2021}. The expressions for $\Gamma_{\mathcal{I}_{a}}(f,f)$ and $\Gamma_{\mathcal{I}_{z}}(f,f)$ are different, see Lemma \ref{lemma a: 2} and Lemma \ref{lemma z: 2} below. 
	\end{remark}
 
By using the time dependent Information Gamma operator defined above, we have the following estimates for the first order dissipation of the Fisher Information functional.
\begin{proposition}\label{prop: Fisher information decay}
    \begin{equation}
    \begin{split}
             &\partial_t [\mathrm I_a(p\|\pi ) + \mathrm I_z(p\|\pi) ]\\
              \le &-2\int  \mathfrak R(\nabla \log\frac{{p} }{\pi},\nabla \log\frac{{p} }{\pi})p  dx+\int \la \nabla\log\frac{{p} }{\pi},\pa_t(aa^{\ts}+zz^{\ts})\nabla\log\frac{{p} }{\pi}\ra {p}  dx\\
& +\int\big[
2\nabla\cdot((aa^{\ts}+zz^{\ts})\nabla\mathcal{R})+2\la \nabla \mathcal R,( aa^{\ts}+zz^{\ts})\nabla\log\pi\ra \big]{p}  dx,
              \end{split}
    \end{equation}
where the correction term $\mathcal R(t,x,\pi)$ is defined in \eqref{correction R}. And the Hessian matrix $\mathfrak R(t,x)$ is defined in the Appendix \ref{appendix}.
\end{proposition}
\begin{proof} Combining Proposition \ref{prop: dissipation of fisher a} and Proposition \ref{prop: dissipation of fisher z} below, we have 
    \begin{equation}
        \begin{split}
                  &\partial_t [\mathrm I_a(p\|\pi ) + \mathrm I_z(p\|\pi) ]\\
              =& -2\int\big[\widetilde \Gamma_2(\log\frac{{p} }{\pi},\log\frac{{p} }{\pi}) \big]{p}  dx
+\int \la \nabla\log\frac{{p} }{\pi},\pa_t(aa^{\ts})\nabla\log\frac{{p} }{\pi}\ra {p}  dx\nonumber \\
&-\int \la \gamma,\la \nabla \log\frac{{p} }{\pi}, \nabla(aa^{\ts})\nabla \log\frac{{p} }{\pi}\ra \ra {p}  dx+2\int \la aa^{\ts}\nabla \log\frac{{p} }{\pi}, \nabla \gamma \nabla \log\frac{{p} }{\pi}\ra {p}  dx\nonumber  \\	
& -2\int \widetilde\Gamma_2^{z,\pi}(\log\frac{{p} }{\pi},\log\frac{{p} }{\pi}){p}  dx
+\int \la \nabla\log\frac{{p} }{\pi},\pa_t(zz^{\ts})\nabla\log\frac{{p} }{\pi}\ra {p}  dx\nonumber \\
&-\int \la \gamma,\la \nabla \log\frac{{p} }{\pi}, \nabla(zz^{\ts})\nabla \log\frac{{p} }{\pi}\ra \ra {p}  dx+2\int \la zz^{\ts}\nabla \log\frac{{p} }{\pi}, \nabla \gamma \nabla \log\frac{{p} }{\pi}\ra {p}  dx \nonumber\\
& +\int\big[
2\nabla\cdot((aa^{\ts}+zz^{\ts})\nabla\mathcal{R})+2\la \nabla \mathcal R,( aa^{\ts}+zz^{\ts})\nabla\log\pi\ra \big]{p}  dx.
        \end{split}
    \end{equation}
    Applying the Information Bochner's formula in \cite[Theorem 3]{FengLi2021} or equivalently in \cite[Proposition 11]{FengLi2021}, for $\beta=0$, we have 
     \begin{equation}
    \begin{split}
             &-2\int\big[\widetilde \Gamma_2(\log\frac{{p} }{\pi},\log\frac{{p} }{\pi}) \big]{p}  dx-2\int \widetilde\Gamma_2^{z,\pi}(\log\frac{{p} }{\pi},\log\frac{{p} }{\pi}){p}  dx\\
              =&-2\int  \Big( \|\mathfrak{Hess}_{\beta} f\|_{\mathrm{F}}^2+ (\mathfrak R-\mathfrak R_{\gamma_a}-\mathfrak R_{\gamma_z} )(\nabla \log\frac{{p} }{\pi},\nabla \log\frac{{p} }{\pi})\Big){p}  dx,
              \end{split}
    \end{equation}
    where $\|\mathfrak{Hess}_{\beta} f\|_{\mathrm{F}}^2$ is the same as defined in \cite[Theorem 3]{FengLi2021} with diffusion matrix $a(t,x)$ and matrix $z(t,x)$. $\|\cdot\|_{\mathrm{F}}$ is the matrix Frobenius norm. The detailed matrices $\mathfrak R$, $\mathfrak R_{\gamma_a}$, $\mathfrak R_{\gamma_z}$ are all defined in the Appendix \ref{appendix}. Note that 
\begin{equation*}
\begin{split}
&-\int \la \gamma,\la \nabla \log\frac{{p} }{\pi}, \nabla(aa^{\ts}+zz^{\ts})\nabla \log\frac{{p} }{\pi}\ra \ra {p}  dx+2\int \la (aa^{\ts}+zz^{\ts})\nabla \log\frac{{p} }{\pi}, \nabla \gamma \nabla \log\frac{{p} }{\pi}\ra {p}  dx\\
&=-2\int  (\mathfrak R_{\gamma_a}+\mathfrak R_{\gamma_z})(\nabla \log\frac{{p} }{\pi},\nabla \log\frac{{p} }{\pi}){p}  dx.
\end{split}
\end{equation*}
Combining the above terms, we finish the proof, since $\|\mathfrak{Hess}_{\beta} f\|_{\mathrm{F}}^2\ge 0$.\qed
\end{proof}
\subsection{Dissipation of $\mathrm I_{a}({p} \|\pi)$} \label{sec fish a}
Now we are ready to present the following technical lemmas for first order dissipation of $\mathrm I_{a}({p} \|\pi)$.
\begin{proposition}\label{prop: dissipation of fisher a}
	\bea
\pa_t\mathrm I_{a}({p} \|\pi) &=& -2\int\big[\widetilde \Gamma_2(\log\frac{{p} }{\pi},\log\frac{{p} }{\pi}) \big]{p}  dx
+\int \la \nabla\log\frac{{p} }{\pi},\pa_t(aa^{\ts})\nabla\log\frac{{p} }{\pi}\ra {p}  dx\nonumber \\
&&-\int \la \gamma,\la \nabla \log\frac{{p} }{\pi}, \nabla(aa^{\ts})\nabla \log\frac{{p} }{\pi}\ra \ra {p}  dx+2\int \la aa^{\ts}\nabla \log\frac{{p} }{\pi}, \nabla \gamma \nabla \log\frac{{p} }{\pi}\ra {p}  dx\nonumber  \\
	&& +\int\big[
	2\nabla\cdot(aa^{\ts}\nabla\mathcal{R})+2\la \nabla \mathcal R, aa^{\ts}\nabla\log\pi\ra \big]{p}  dx.\nonumber
	\eea
\end{proposition}
\begin{proof}
The proof of Proposition \ref{prop: dissipation of fisher a} follows from Lemma \ref{lemma a: 1} and Lemma \ref{lemma a: 2}. According to Lemma \ref{lemma a: 2}, we have 
\beaa
	\int\widetilde \Gamma_{\mathcal I_a}(f,f){p}  dx&=&\frac{1}{2}\int \la \gamma,\la \nabla f, \nabla(aa^{\ts})\nabla f\ra \ra {p}  dx-\int \la aa^{\ts}\nabla f, \nabla \gamma \nabla f\ra {p}  dx\\&& +\int \frac{\nabla\cdot (\pi\gamma)}{\pi} \Gamma_1(f,f){p}  dx.	
	\eeaa
	Plugging into Lemma \ref{lemma a: 1} with $f=\log\frac{{p} }{\pi}$, we prove the results. 
	\qed \end{proof}	
\begin{remark}
    The extra term in Lemma \ref{lemma a: 2} involving $\nabla\cdot(\pi\gamma)$ is canceled by the extra term in Lemma \ref{lemma a: 1}. This makes it possible to define the tensor $\mathfrak R_{\gamma_a}$ the same as in the time independent setting \cite{FengLi2021}.
\end{remark}
\begin{lemma}\label{lemma a: 1}
	\bea
\pa_t\mathrm I_{a}({p} \|\pi) &=& -2\int\big[\widetilde \Gamma_2(\log\frac{{p} }{\pi},\log\frac{{p} }{\pi})+\widetilde \Gamma_{\mathcal I_a}(\log\frac{{p} }{\pi},\log\frac{{p} }{\pi}) \big]{p}  dx\\
&&+\int \la \nabla\log\frac{{p} }{\pi},[\pa_t(aa^{\ts})+2aa^{\ts}\frac{\nabla\cdot(\pi\gamma)}{\pi}]\nabla\log\frac{{p} }{\pi}\ra {p}  dx\nonumber \\
	&& +\int\big[
	2\nabla\cdot(aa^{\ts}\nabla\mathcal{R})+2\la \nabla \mathcal R, aa^{\ts}\nabla\log\pi\ra \big]{p}  dx.\nonumber
	\eea
\end{lemma}

\begin{proof}
	\beaa
	\pa_t \mathrm I_{a}({p} \|\pi)&=&2\int \la \nabla\pa_t \log\frac{{p} }{\pi},aa^{\ts}\nabla\log\frac{{p} }{\pi}\ra {p}  dx+\int \la \nabla\log\frac{{p} }{\pi},\pa_t(aa^{\ts})\nabla\log\frac{{p} }{\pi}\ra {p}  dx\\
	&&+\int \la \nabla\log\frac{{p} }{\pi},aa^{\ts}\nabla\log\frac{{p} }{\pi}\ra \pa_t {p}  dx\\
	&=&2\int \la \nabla\pa_t \log{p} ,aa^{\ts}\nabla\log\frac{{p} }{\pi}\ra {p}  dx+\int \la \nabla\log\frac{{p} }{\pi},aa^{\ts}\nabla\log\frac{{p} }{\pi}\ra \pa_t {p}  dx\\
	&&+\int \la \nabla\log\frac{{p} }{\pi},\pa_t(aa^{\ts})\nabla\log\frac{{p} }{\pi}\ra {p}  dx -2\int \la \nabla\pa_t \log\pi,aa^{\ts}\nabla\log\frac{{p} }{\pi}\ra {p}  dx.
	\eeaa
We first observe that, 
	\bea \label{term: gamma 2}
&&2\int \la \nabla\pa_t \log{p} ,aa^{\ts}\nabla\log\frac{{p} }{\pi}\ra {p}  dx+\int \la \nabla\log\frac{{p} }{\pi},aa^{\ts}\nabla\log\frac{{p} }{\pi}\ra \pa_t {p}  dx\nonumber \\
&=&\int \Gamma_1(\log\frac{{p}  }{\pi}, \log\frac{{p}  }{\pi})\partial_t {p}  -2\frac{\nabla\cdot({p}  aa^{\ts}\nabla\log\frac{{p}  }{\pi})}{{p}  }\partial_t {p}   dx \nonumber\\
&=&\int \Gamma_1(\log\frac{{p}  }{\pi}, \log\frac{{p}  }{\pi})\partial_t {p}  -2\Big(\la\nabla\log {p}  , aa^{\ts}\nabla\log\frac{{p}  }{\pi}\ra+\nabla\cdot(aa^{\ts}\nabla\log\frac{{p}  }{\pi})\Big)\partial_t {p}   dx \nonumber\\
&=&\int \Gamma_1(\log\frac{{p}  }{\pi}, \log\frac{{p}  }{\pi})\partial_t {p}  -2\Big(\la\nabla\log\frac{{p}  }{\pi}, aa^{\ts}\nabla\log\frac{{p}  }{\pi}\ra+\widetilde L\log\frac{{p}  }{\pi}\Big)\partial_t {p}   dx \nonumber\\
&=&-2\int \Big\{\frac{1}{2}\Gamma_1(\log\frac{{p}  }{\pi}, \log\frac{{p}  }{\pi})\partial_t {p}  +\widetilde L\log\frac{{p}  }{\pi}\partial_t {p}   \Big\}dx \nonumber\\
&=&-2\int \Big\{\quad\frac{1}{2}\Gamma_1(\log\frac{{p}  }{\pi}, \log\frac{{p}  }{\pi})\nabla\cdot({p}  \gamma)+\widetilde L\log\frac{{p}  }{\pi}\nabla\cdot({p}  \gamma)\nonumber\\
&&\hspace{1.5cm}+\frac{1}{2}\Gamma_1(\log\frac{{p}  }{\pi}, \log\frac{{p}  }{\pi})\widetilde L^*{p}  +\widetilde L\log\frac{{p}  }{\pi}\widetilde L^*{p}  \Big\} dx\nonumber \\
&=&-2\int \Big\{-\frac{1}{2}\la\nabla\Gamma_1(\log\frac{{p}  }{\pi}, \log\frac{{p}  }{\pi}), \gamma\ra+\widetilde L\log\frac{{p}  }{\pi}\la\nabla\log\frac{{p}  }{\pi}, \gamma\ra+\widetilde L\log\frac{{p}  }{\pi} \frac{\nabla\cdot(\pi\gamma)}{\pi} \nonumber\\
&&\hspace{1.5cm}+\frac{1}{2}\widetilde L\Gamma_1(\log\frac{{p}  }{\pi}, \log\frac{{p}  }{\pi})-\Gamma_1( \widetilde L\log\frac{{p}  }{\pi}, \log\frac{{p}  }{\pi})\Big\} {p}  dx \nonumber\\
&=& -2\int [\widetilde \Gamma_2(\log\frac{{p} }{\pi},\log\frac{{p} }{\pi})+\widetilde \Gamma_{\mathcal I_a}(\log\frac{{p} }{\pi},\log\frac{{p} }{\pi})]{p}  dx-2\int \tilde L \log\frac{{p} }{\pi}\frac{\nabla\cdot(\pi\gamma)}{\pi} {p}  dx.\nonumber
		\eea
		
In the second last equality, we apply the following fact $p\nabla\log p=\nabla p$, $\pi\nabla\log \pi=\nabla \pi$, such that 
\bea \label{trick: rho gamma}
\nabla\cdot({p}  \gamma)&=&{p}  \Big(\la\nabla\log {p}  , \gamma\ra+\nabla\cdot\gamma\Big)\nonumber \\
&=& {p} \Big(\la\nabla\log {p}  , \gamma\ra+\frac{\nabla\cdot(\pi\gamma)}{\pi}-\la\nabla\log\pi, \gamma\ra   \Big)\nonumber \\
&=&{p}  \la\nabla\log\frac{{p}  }{\pi}, \gamma\ra+ {p}  \frac{\nabla \cdot(\pi \gamma)}{\pi}.    
\eea 
We then have, 
	\beaa 
	\pa_t \mathrm I_{a}({p} \|\pi)&=& -2\int\Big[\widetilde \Gamma_2(\log\frac{{p} }{\pi},\log\frac{{p} }{\pi})+\widetilde \Gamma_{\mathcal{I}_a}(\log\frac{{p} }{\pi},\log\frac{{p} }{\pi}) \Big]{p}  dx-2\int \tilde L \log\frac{{p} }{\pi}\frac{\nabla\cdot(\pi\gamma)}{\pi} {p}  dx  \\
	&&+\int \la \nabla\log\frac{{p} }{\pi},\pa_t(aa^{\ts})\nabla\log\frac{{p} }{\pi}\ra {p}  dx -2\int \la \nabla\pa_t \log\pi,aa^{\ts}\nabla\log\frac{{p} }{\pi}\ra {p}  dx.
		\eeaa
Observing the following equality, we have 
\bea\label{term 1}
&&-2\int \la \nabla\pa_t \log\pi,aa^{\ts}\nabla\log\frac{{p} }{\pi}\ra {p}  dx\nonumber \\
&=& 2\int  \pa_t \log\pi \frac{\nabla\cdot(  {p}  aa^{\ts}\nabla\log\frac{{p} }{\pi})}{{p} } {p}  dx\nonumber\\
&=&2\int\Big[ \nabla\cdot (aa^{\ts}\nabla\log\frac{{p} }{\pi})+\la \nabla\log\frac{{p} }{\pi},aa^{\ts}\nabla\log\frac{{p} }{\pi}\ra+\la \nabla\log\pi,aa^{\ts}\nabla\log\frac{{p} }{\pi}\rangle\Big] \pa_t\log\pi {p}  dx\nonumber\\
&=&2\int \Big[ \tilde L  \log\frac{{p} }{\pi}+\la \nabla\log\frac{{p} }{\pi},aa^{\ts}\nabla\log\frac{{p} }{\pi}\ra \Big] \pa_t\log\pi {p}  dx,
\eea
which implies
\beaa
\pa_t\mathrm I_{a
}({p} \|\pi)	&=&  -2\int \Big[\widetilde{\Gamma}_2(\log\frac{{p} }{\pi},\log\frac{{p} }{\pi})+\widetilde \Gamma_{\mathcal I_a}(\log\frac{{p} }{\pi},\log\frac{{p} }{\pi})\Big] {p}  dx+\int \la \nabla\log\frac{{p} }{\pi},\pa_t(aa^{\ts})\nabla\log\frac{{p} }{\pi}\ra {p}  dx\\
	&&+2\int \tilde L \log\frac{{p} }{\pi}\frac{\pa_t\pi-\nabla\cdot(\pi\gamma)}{\pi} {p}  dx  +2\int \la \nabla\log\frac{{p} }{\pi},aa^{\ts}\nabla\log\frac{{p} }{\pi}\ra  \pa_t\log\pi {p}  dx.
\eeaa
We also have 
\beaa
&&2\int \tilde L \log\frac{{p} }{\pi}\mathcal{R} {p}  dx \\
&=&2\int \nabla\cdot(aa^{\ts}\nabla\log \frac{{p} }{\pi})\mathcal{R}{p}  dx + 2\int \la \nabla \log\pi,aa^{\ts}\nabla\log \frac{{p} }{\pi}\ra \mathcal{R} {p}  dx\\
&=& -2\int \la \nabla(\mathcal{R}{p} ),aa^{\ts}\nabla\log\frac{{p} }{\pi}\ra dx+ 2\int \la \nabla \log\pi,aa^{\ts}\nabla\log \frac{{p} }{\pi}\ra \mathcal{R} {p}  dx\\
&=& -2\int \la \nabla\mathcal{R},aa^{\ts}\nabla\log\frac{{p} }{\pi}\ra {p}  dx-2\int \la \nabla {p} ,aa^{\ts}\nabla\log\frac{{p} }{\pi}\ra\mathcal{R} dx+ 2\int \la \nabla \log\pi,aa^{\ts}\nabla\log \frac{{p} }{\pi}\mathcal \ra \mathcal{R} {p}  dx\\
&=& -2\int \la \nabla\mathcal{R},aa^{\ts}\nabla\log\frac{{p} }{\pi}\ra {p}  dx-2\int \la \nabla\log \frac{{p} }{\pi},aa^{\ts}\nabla\log\frac{{p} }{\pi}\ra\mathcal{R} {p}  dx.
\eeaa
Combining the above terms, we have 
\beaa
\pa_t\mathrm I_{a}({p} \|\pi) &=& -2\int\big[\widetilde{ \Gamma}_2(\log\frac{{p} }{\pi},\log\frac{{p} }{\pi})+\widetilde \Gamma_{\mathcal I_a}(\log\frac{{p} }{\pi},\log\frac{{p} }{\pi}) \big]{p}  dx\\
&&+\int \la \nabla\log\frac{{p} }{\pi},[\pa_t(aa^{\ts})+2aa^{\ts}\pa_t\log\pi]\nabla\log\frac{{p} }{\pi}\ra {p}  dx\\
	&&-2\int \la \nabla \mathcal{R}, aa^{\ts}\nabla\log\frac{{p} }{\pi}\ra{p}  dx	
	-2\int \la \nabla\log \frac{{p} }{\pi},aa^{\ts}\nabla\log\frac{{p} }{\pi}\ra\mathcal{R} {p}  dx.
\eeaa
Note that, 
\beaa
-\int \la \nabla \mathcal{R}, aa^{\ts}\nabla\log\frac{{p} }{\pi}\ra{p}  dx&=&-\int \la \nabla \mathcal{R}, aa^{\ts}\nabla{p} \ra dx+\int \la \nabla \mathcal{R}, aa^{\ts}\nabla\log\pi\ra{p}  dx\\
&=&\int \nabla\cdot(aa^{\ts}\nabla\mathcal{R}){p}  dx+\int \la \nabla \mathcal{R}, aa^{\ts}\nabla\log\pi\ra{p}  dx.
\eeaa
We conclude with 
\beaa
\pa_t\mathrm I_{a}({p} \|\pi) &=& -2\int\big[\widetilde{ \Gamma}_2(\log\frac{{p} }{\pi},\log\frac{{p} }{\pi})+\widetilde \Gamma_{\mathcal I_a}(\log\frac{{p} }{\pi},\log\frac{{p} }{\pi}) \big]{p}  dx\\
&&+\int \la \nabla\log\frac{{p} }{\pi},[\pa_t(aa^{\ts})+2aa^{\ts}\pa_t\log\pi-2aa^{\ts}\mathcal{R}]\nabla\log\frac{{p} }{\pi}\ra {p}  dx\\
	&& +\int\big[ 
	2\nabla\cdot(aa^{\ts}\nabla\mathcal{R})+2 \la \nabla \mathcal{R}, aa^{\ts}\nabla\log\pi\ra \big]{p}  dx.
\eeaa
And the results follow the fact $\pa_t\log \pi-\mathcal R=\frac{\nabla\cdot(\pi\gamma)}{\pi}$. \qed 
\end{proof}

Recall that the irreversible Gamma operator associated with $a$ is defined as 
\beaa
\widetilde\Gamma_{\mathcal I_a}(f,f)=\widetilde L f\la\nabla f, \gamma\ra-\frac{1}{2}\la \nabla\Gamma_1(f,f),\gamma\ra.
\eeaa
We next show the following equivalence identity in a weak form for the irreversible Gamma operator. 
\begin{lemma}\label{lemma a: 2}
Denote $f=\log\frac{{p} }{\pi}$, we have 
	\beaa
	\int\widetilde \Gamma_{\mathcal I_a}(f,f){p}  dx&=&\frac{1}{2}\int \la \gamma,\la \nabla f, \nabla(aa^{\ts})\nabla f\ra \ra {p}  dx-\int \la aa^{\ts}\nabla f, \nabla \gamma \nabla f\ra {p}  dx\\
 &&+\int \frac{\nabla\cdot (\pi\gamma)}{\pi} \Gamma_1(f,f){p}  dx.	
	\eeaa
\end{lemma}
\begin{proof}
We first observe that, 
\beaa
\frac{\nabla\cdot({p}  \gamma)}{{p} }=\la \nabla \log{p} ,\gamma\ra +\nabla\cdot\gamma =\la \nabla \log\frac{{p} }{\pi},\gamma\ra+\la \nabla \log \pi,\gamma\ra +\nabla\cdot\gamma =\la \nabla f,\gamma\ra +\frac{\nabla\cdot (\pi\gamma)}{\pi}.
\eeaa
According to the definition of the information Gamma operator, we have 
	\beaa
	&&\int \widetilde\Gamma_{\mathcal I_a}(f,f){p}  dx\\
	&=& \int [\widetilde L f\la \nabla f, \gamma\ra -\frac{1}{2}\la \nabla\Gamma_1(f,f),\gamma\ra ]{p}  dx\\
	&=&\int \Big[\nabla\cdot(aa^{\ts}\nabla f)\la \nabla f, \gamma\ra+\la aa^{\ts}\nabla\log\pi,\nabla f\ra \la \nabla f, \gamma\ra\Big]{p}  dx+\frac{1}{2}\int\nabla\cdot ({p} \gamma)\Gamma_1(f,f)dx\\
	&=&\int \Big[\nabla\cdot(aa^{\ts}\nabla f)\la \nabla f, \gamma\ra+\la aa^{\ts}\nabla\log\pi,\nabla f\ra \la \nabla f, \gamma\ra\Big]{p}  dx\\
	&&+\frac{1}{2}\int [\la \nabla f,\gamma\ra +\frac{\nabla\cdot (\pi\gamma)}{\pi} ]\Gamma_1(f,f){p}  dx
 \eeaa 
 \beaa 
	&=&\int \Big[-\la aa^{\ts}\nabla f,\nabla \log {p} \ra  \la \nabla f, \gamma\ra+\la aa^{\ts}\nabla\log\pi,\nabla f\ra \la \nabla f, \gamma\ra\Big]{p}  dx+\frac{1}{2}\int \la \nabla f,\gamma\ra\Gamma_1(f,f){p}  dx\\
&&-\int [\la aa^{\ts}\nabla f, \nabla^2 f\gamma\ra-\la aa^{\ts}\nabla f,\nabla\gamma\nabla f\ra]{p}  dx+\frac{1}{2}\int \frac{\nabla\cdot (\pi\gamma)}{\pi} \Gamma_1(f,f){p}  dx\\
&=&-\frac{1}{2}\int \la \nabla f,\gamma\ra\Gamma_1(f,f){p}  dx-\int [\la aa^{\ts}\nabla f, \nabla^2 f\gamma\ra-\la aa^{\ts}\nabla f,\nabla\gamma\nabla f\ra]{p}  dx+\frac{1}{2}\int \frac{\nabla\cdot (\pi\gamma)}{\pi} \Gamma_1(f,f){p}  dx\\
&=&-\frac{1}{2}\int \la \nabla  {p} ,\gamma\ra\Gamma_1(f,f) dx+\frac{1}{2}\int \la \nabla \log \pi,\gamma\ra\Gamma_1(f,f){p}  dx\\
&&-\int [\la aa^{\ts}\nabla f, \nabla^2 f\gamma\ra-\la aa^{\ts}\nabla f,\nabla\gamma\nabla f\ra]{p}  dx+\frac{1}{2}\int \frac{\nabla\cdot (\pi\gamma)}{\pi} \Gamma_1(f,f){p}  dx\\
&=&\frac{1}{2}\int \la \gamma,\la \nabla f, \nabla(aa^{\ts})\nabla f\ra \ra {p}  dx-\int \la aa^{\ts}\nabla f, \nabla \gamma \nabla f\ra {p}  dx +\int \frac{\nabla\cdot (\pi\gamma)}{\pi} \Gamma_1(f,f){p}  dx.
	\eeaa
	The last equality follows from the fact that 
	\beaa
&&-\frac{1}{2}\int \la \nabla  {p} ,\gamma\ra\Gamma_1(f,f) dx
=	\frac{1}{2}\int \nabla\cdot (\gamma \Gamma_1(f,f)){p}  dx\\
&=&\frac{1}{2}\int\nabla\cdot \gamma \Gamma_1(f,f){p}  dx +\int \la aa^{\ts}\nabla f, \nabla^2 f\gamma\ra {p}  dx+\frac{1}{2}\int \la \gamma,\la \nabla f, \nabla(aa^{\ts})\nabla f\ra \ra {p}  dx,
	\eeaa
	and 
	\beaa
\frac{\nabla\cdot(\pi \gamma)}{\pi}	=\la \nabla \log\pi,\gamma\ra +\nabla\cdot \gamma.
	\eeaa
	\qed
\end{proof}
\subsection{Dissipation of auxillary Fisher information}\label{sec fish z}
Similar to the first-order dissipation of the Fisher information functional, we have the following decay for the auxiliary Fisher information functional. 
\begin{proposition}
	\label{prop: dissipation of fisher z}
	\bea
\pa_t\mathrm I_{z}({p} \|\pi) &=& -2\int \widetilde\Gamma_2^{z,\pi}(\log\frac{{p} }{\pi},\log\frac{{p} }{\pi}){p}  dx
+\int \la \nabla\log\frac{{p} }{\pi},\pa_t(zz^{\ts})\nabla\log\frac{{p} }{\pi}\ra {p}  dx\nonumber \\
&&-\int \la \gamma,\la \nabla \log\frac{{p} }{\pi}, \nabla(zz^{\ts})\nabla \log\frac{{p} }{\pi}\ra \ra {p}  dx+2\int \la zz^{\ts}\nabla \log\frac{{p} }{\pi}, \nabla \gamma \nabla \log\frac{{p} }{\pi}\ra {p}  dx \nonumber\\
	&& +\int\big[
	2\nabla\cdot(zz^{\ts}\nabla\mathcal{R})+2 \la \nabla \mathcal R, zz^{\ts}\nabla\log\pi\ra \big]{p}  dx.
	\eea
\end{proposition}
The proof follows from the following Lemma \ref{lemma z: 1}
 and Lemma \ref{lemma z: 2}. 
 \begin{lemma}\label{lemma z: 1}
	\bea
\pa_t\mathrm I_{z}({p} \|\pi) &=& -2\int\big[\widetilde \Gamma^{z, \pi}_2(\log\frac{{p} }{\pi},\log\frac{{p} }{\pi})+\Gamma_{\mathcal I_z}(\log\frac{{p} }{\pi},\log\frac{{p} }{\pi}) \big]{p}  dx\\
&&+\int \la \nabla\log\frac{{p} }{\pi},[\pa_t(zz^{\ts})+2zz^{\ts}\frac{\nabla\cdot(\pi\gamma)}{\pi}]\nabla\log\frac{{p} }{\pi}\ra {p}  dx\nonumber \\
	&& +\int\big[
	2\nabla\cdot(zz^{\ts}\nabla\mathcal{R})+2 \la \nabla \mathcal R, zz^{\ts}\nabla\log\pi\ra \big]{p}  dx.\nonumber
	\eea
\end{lemma}
\begin{proof}
	\beaa
	\pa_t \mathrm I_{z}({p} \|\pi)&=&2\int \la \nabla\pa_t \log\frac{{p} }{\pi},zz^{\ts}\nabla\log\frac{{p} }{\pi}\ra {p}  dx+\int \la \nabla\log\frac{{p} }{\pi},\pa_t(zz^{\ts})\nabla\log\frac{{p} }{\pi}\ra {p}  dx\\
	&&+\int \la \nabla\log\frac{{p} }{\pi},zz^{\ts}\nabla\log\frac{{p} }{\pi}\ra \pa_t {p}  dx\\
	&=&2\int \la \nabla\pa_t \log{p} ,zz^{\ts}\nabla\log\frac{{p} }{\pi}\ra {p}  dx+\int \la \nabla\log\frac{{p} }{\pi},zz^{\ts}\nabla\log\frac{{p} }{\pi}\ra \pa_t {p}  dx\\
	&&+\int \la \nabla\log\frac{{p} }{\pi},\pa_t(zz^{\ts})\nabla\log\frac{{p} }{\pi}\ra {p}  dx -2\int \la \nabla\pa_t \log\pi,zz^{\ts}\nabla\log\frac{{p} }{\pi}\ra {p}  dx.
	\eeaa
Similar to the derivation for $\mathrm I_a({p} \|\pi)$, we first observe the following the fact,
\beaa 
&=&\int \Gamma_1^z(\log\frac{{p}  }{\pi}, \log\frac{{p}  }{\pi})\partial_t {p}  +2\Gamma_1^z(\frac{\partial_t {p}  }{{p}  }, \log\frac{{p}  }{\pi}){p}   dx \\
&=&\int \Gamma_1^z(\log\frac{{p}  }{\pi}, \log\frac{{p}  }{\pi})\partial_t {p}  -2\frac{\nabla\cdot({p}  zz^{\ts}\nabla\log\frac{{p}  }{\pi})}{{p}  }\partial_t {p}   dx \\
&=&\int \Gamma_1^z(\log\frac{{p}  }{\pi}, \log\frac{{p}  }{\pi})\partial_t {p}  -2\Big(\la\nabla\log {p}  , zz^{\ts}\nabla\log\frac{{p}  }{\pi}\ra+\nabla\cdot(zz^{\ts}\nabla\log\frac{{p}  }{\pi})\Big)\partial_t {p}   dx\\
&=&\int \Gamma_1^z(\log\frac{{p}  }{\pi}, \log\frac{{p}  }{\pi})\partial_t {p}  -2\Big(\la\nabla\log\frac{{p}  }{\pi}, zz^{\ts}\nabla\log\frac{{p}  }{\pi}\ra+\widetilde L_z\log\frac{{p}  }{\pi}\Big)\partial_t {p}   dx \\
&=&-2\int \Big\{\frac{1}{2}\Gamma_1^z(\log\frac{{p}  }{\pi}, \log\frac{{p}  }{\pi})\partial_t {p}  +\widetilde 
L_z\log\frac{{p}  }{\pi}\partial_t {p}   \Big\}dx \\
&=&-2\int \Big\{\quad\frac{1}{2}\Gamma_1^z(\log\frac{{p}  }{\pi}, \log\frac{{p}  }{\pi})\nabla\cdot({p}  \gamma)+\widetilde L_z\log\frac{{p}  }{\pi}\nabla\cdot({p}  \gamma)\\
&&\hspace{1.5cm}+\frac{1}{2}\Gamma_1^z(\log\frac{{p}  }{\pi}, \log\frac{{p}  }{\pi})\widetilde L^*{p}  +\widetilde L_z\log\frac{{p}  }{\pi}\widetilde L^*{p}  \Big\} dx 
\eeaa 
\beaa 
&=&-2\int \Big\{\widetilde \Gamma_{\mathcal{I}_z}(\log\frac{{p}  }{\pi}, \log\frac{{p}  }{\pi})+\frac{1}{2}\widetilde L_z\Gamma_1(\log\frac{{p}  }{\pi}, \log\frac{{p}  }{\pi})-\Gamma_1( \widetilde L_z\log\frac{{p}  }{\pi}, \log\frac{{p}  }{\pi})\Big\} {p}  dx\\
&&-2\int \tilde L_z \log\frac{{p} }{\pi}\frac{\nabla\cdot(\pi\gamma)}{\pi} {p}  dx, 
\eeaa 
where we apply the equality in \eqref{trick: rho gamma} for $\nabla\cdot({p}  \gamma)$. Now applying [Proposition $5.11$]\cite{FengLi} (see also [Proposition 8]\cite{FengLi2021}), we have the following equality 
\begin{equation}\label{identity}
\int\Big\{\frac{1}{2}\widetilde L_z\Gamma_1(\log\frac{{p}  }{\pi}, \log\frac{{p}  }{\pi})-\Gamma_1( \widetilde L_z\log\frac{{p}  }{\pi}, \log\frac{{p}  }{\pi})\Big\} {p}  dx=\int  \widetilde\Gamma_2^{z,\pi}(\log\frac{{p}  }{\pi},\log\frac{{p}  }{\pi}) {p}   dx.
\end{equation}
We then have, 
	\beaa 
	\pa_t \mathrm I_{z}({p} \|\pi)&=& -2\int\Big[\widetilde \Gamma_2^{z,\pi}(\log\frac{{p} }{\pi},\log\frac{{p} }{\pi})+\widetilde \Gamma_{\mathcal I_z}(\log\frac{{p} }{\pi},\log\frac{{p} }{\pi}) \Big]{p}  dx-2\int \tilde L_z\log\frac{{p} }{\pi}\frac{\nabla\cdot(\pi\gamma)}{\pi} {p}  dx  \\
	&&+\int \la \nabla\log\frac{{p} }{\pi},\pa_t(zz^{\ts})\nabla\log\frac{{p} }{\pi}\ra {p}  dx -2\int \la \nabla\pa_t \log\pi,zz^{\ts}\nabla\log\frac{{p} }{\pi}\ra {p}  dx.
	\eeaa
Observing the following equality, we have 
\bea\label{term 1}
&&-2\int \la \nabla\pa_t \log\pi,zz^{\ts}\nabla\log\frac{{p} }{\pi}\ra {p}  dx\nonumber \\
&=& 2\int  \pa_t \log\pi \frac{\nabla\cdot(  {p}  zz^{\ts}\nabla\log\frac{{p} }{\pi})}{{p} } {p}  dx\nonumber\\
&=&2\int\Big[ \nabla\cdot (zz^{\ts}\nabla\log\frac{{p} }{\pi})+\la \nabla\log\frac{{p} }{\pi},zz^{\ts}\nabla\log\frac{{p} }{\pi}\ra+\la \nabla\log\pi,zz^{\ts}\nabla\log\frac{{p} }{\pi}\rangle\Big] \pa_t\log\pi {p}  dx\nonumber\\
&=&2\int \Big[ \tilde L_z  \log\frac{{p} }{\pi}+\la \nabla\log\frac{{p} }{\pi},zz^{\ts}\nabla\log\frac{{p} }{\pi}\ra \Big] \pa_t\log\pi {p}  dx,
\eea
which implies
\beaa
\pa_t\mathrm I_{z
}({p} \|\pi)	&=&  -2\int \Big[\widetilde\Gamma_2^{z, \pi}(\log\frac{{p} }{\pi},\log\frac{{p} }{\pi})+\widetilde\Gamma_{\mathcal I_z}(\log\frac{{p} }{\pi},\log\frac{{p} }{\pi})\Big] {p}  dx+\int \la \nabla\log\frac{{p} }{\pi},\pa_t(zz^{\ts})\nabla\log\frac{{p} }{\pi}\ra {p}  dx\\
	&&+2\int \tilde L_z \log\frac{{p} }{\pi}\frac{\pa_t\pi-\nabla\cdot(\pi\gamma)}{\pi} {p}  dx  +2\int \la \nabla\log\frac{{p} }{\pi},zz^{\ts}\nabla\log\frac{{p} }{\pi}\ra  \pa_t\log\pi {p}  dx.
\eeaa
We also have 
\beaa
&&2\int \tilde L_z \log\frac{{p} }{\pi}\mathcal{R} {p}  dx \\
&=&2\int \nabla\cdot(zz^{\ts}\nabla\log \frac{{p} }{\pi})\mathcal{R}{p}  dx + 2\int \la \nabla \log\pi,zz^{\ts}\nabla\log \frac{{p} }{\pi}\ra \mathcal{R} {p}  dx\\
&=& -2\int \la \nabla\mathcal{R},zz^{\ts}\nabla\log\frac{{p} }{\pi}\ra {p}  dx-2\int \la \nabla\log \frac{{p} }{\pi},zz^{\ts}\nabla\log\frac{{p} }{\pi}\ra\mathcal{R} {p}  dx.
\eeaa
Combining the above terms, we have 
\beaa
\pa_t\mathrm I_{z}({p} \|\pi) &=& -2\int\big[\widetilde \Gamma_2^{z, \pi}(\log\frac{{p} }{\pi},\log\frac{{p} }{\pi})+\widetilde\Gamma_{\mathcal I_z}(\log\frac{{p} }{\pi},\log\frac{{p} }{\pi}) \big]{p}  dx\\
&&+\int \la \nabla\log\frac{{p} }{\pi},[\pa_t(zz^{\ts})+2zz^{\ts}\pa_t\log\pi]\nabla\log\frac{{p} }{\pi}\ra {p}  dx\\
	&& -2\int \la \nabla \mathcal{R}, zz^{\ts}\nabla\log\frac{{p} }{\pi}\ra{p}  dx
	-2\int \la \nabla\log \frac{{p} }{\pi},zz^{\ts}\nabla\log\frac{{p} }{\pi}\ra\mathcal{R} {p}  dx.
\eeaa
Note that, 
\beaa
-2\int \la \nabla \mathcal{R}, zz^{\ts}\nabla\log\frac{{p} }{\pi}\ra{p}  dx&=&-2\int \la \nabla \mathcal{R}, zz^{\ts}\nabla{p} \ra dx+2\int \la \nabla \mathcal{R}, zz^{\ts}\nabla\log\pi\ra{p}  dx\\
&=&2\int \nabla\cdot(zz^{\ts}\nabla\mathcal{R}){p}  dx+2\int \la \nabla \mathcal{R}, zz^{\ts}\nabla\log\pi\ra{p}  dx.
\eeaa
We conclude with 
\beaa
\pa_t\mathrm I_{z}({p} \|\pi) &=& -2\int\big[\widetilde \Gamma_2^{z,\pi}(\log\frac{{p} }{\pi},\log\frac{{p} }{\pi})+\widetilde\Gamma_{\mathcal I_z}(\log\frac{{p} }{\pi},\log\frac{{p} }{\pi}) \big]{p}  dx\\
&&+\int \la \nabla\log\frac{{p} }{\pi},[\pa_t(zz^{\ts})+2zz^{\ts}\pa_t\log\pi-2zz^{\ts}\mathcal{R}]\nabla\log\frac{{p} }{\pi}\ra {p}  dx\\
	&& +\int\big[ 
	2\nabla\cdot(zz^{\ts}\nabla\mathcal{R})+2 \la \nabla \mathcal{R}, zz^{\ts}\nabla\log\pi\ra \big]{p}  dx,
\eeaa
and the result follows the fact $\pa_t\log \pi-\mathcal R=\frac{\nabla\cdot(\pi\gamma)}{\pi}$. \qed 
\end{proof}
The irreversible Gamma operator associated with matrix $z$ has the following equivalent form.
\begin{lemma}\label{lemma z: 2}
Denote $f=\log\frac{{p} }{\pi}$, we have 
	\beaa
	\int\widetilde \Gamma_{\mathcal I_z}(f,f){p}  dx&=&\frac{1}{2}\int \la \gamma,\la \nabla f, \nabla(zz^{\ts})\nabla f\ra \ra {p}  dx-\int \la zz^{\ts}\nabla f, \nabla \gamma \nabla f\ra {p}  dx\\
 &&+\int \frac{\nabla\cdot (\pi\gamma)}{\pi} \Gamma_1^z(f,f){p}  dx.
 \eeaa
\end{lemma}
\begin{proof}
We will use the following fact again 
\beaa
\frac{\nabla\cdot({p}  \gamma)}{{p} }=\la \nabla \log{p} ,\gamma\ra +\nabla\cdot\gamma =\la \nabla \log\frac{{p} }{\pi},\gamma\ra+\la \nabla \log \pi,\gamma\ra +\nabla\cdot\gamma =\la \nabla f,\gamma\ra +\frac{\nabla\cdot (\pi\gamma)}{\pi}.
\eeaa
We have 
	\beaa
	&&\int \widetilde\Gamma_{\mathcal I_z}(f,f){p}  dx\\
	&=& \int [\widetilde L_z f\la \nabla f, \gamma\ra -\frac{1}{2}\la \nabla\Gamma_1^z(f,f),\gamma\ra ]{p}  dx\\
	&=&\int \Big[\nabla\cdot(zz^{\ts}\nabla f)\la \nabla f, \gamma\ra+\la zz^{\ts}\nabla\log\pi,\nabla f\ra \la \nabla f, \gamma\ra\Big]{p}  dx+\frac{1}{2}\int\nabla\cdot ({p} \gamma)\Gamma_1^z(f,f)dx\\
	&=&\int \Big[\nabla\cdot(zz^{\ts}\nabla f)\la \nabla f, \gamma\ra+\la zz^{\ts}\nabla\log\pi,\nabla f\ra \la \nabla f, \gamma\ra\Big]{p}  dx\\
	&&+\frac{1}{2}\int [\la \nabla f,\gamma\ra +\frac{\nabla\cdot (\pi\gamma)}{\pi} ]\Gamma_1^z(f,f){p}  dx\\
	&=&\int \Big[-\la zz^{\ts}\nabla f,\nabla \log {p} \ra  \la \nabla f, \gamma\ra+\la zz^{\ts}\nabla\log\pi,\nabla f\ra \la \nabla f, \gamma\ra\Big]{p}  dx+\frac{1}{2}\int \la \nabla f,\gamma\ra\Gamma_1^z(f,f){p}  dx\\
&&-\int [\la zz^{\ts}\nabla f, \nabla^2 f\gamma\ra-\la zz^{\ts}\nabla f,\nabla\gamma\nabla f\ra]{p}  dx+\frac{1}{2}\int \frac{\nabla\cdot (\pi\gamma)}{\pi} \Gamma_1^z(f,f){p}  dx\\
&=&-\frac{1}{2}\int \la \nabla f,\gamma\ra\Gamma_1^z(f,f){p}  dx-\int [\la zz^{\ts}\nabla f, \nabla^2 f\gamma\ra-\la zz^{\ts}\nabla f,\nabla\gamma\nabla f\ra]{p}  dx+\frac{1}{2}\int \frac{\nabla\cdot (\pi\gamma)}{\pi} \Gamma_1^z(f,f){p}  dx\\
&=&-\frac{1}{2}\int \la \nabla  {p} ,\gamma\ra\Gamma_1^z(f,f) dx+\frac{1}{2}\int \la \nabla \log \pi,\gamma\ra\Gamma_1^z(f,f){p}  dx\\
&&-\int [\la zz^{\ts}\nabla f, \nabla^2 f\gamma\ra-\la zz^{\ts}\nabla f,\nabla\gamma\nabla f\ra]{p}  dx+\frac{1}{2}\int \frac{\nabla\cdot (\pi\gamma)}{\pi} \Gamma_1^z(f,f){p}  dx\\
&=&\frac{1}{2}\int \la \gamma,\la \nabla f, \nabla(zz^{\ts})\nabla f\ra \ra {p}  dx-\int \la zz^{\ts}\nabla f, \nabla \gamma \nabla f\ra {p}  dx +\int \frac{\nabla\cdot (\pi\gamma)}{\pi} \Gamma_1^z(f,f){p}  dx.
	\eeaa
	The last equality follows from the fact that 
	\beaa
&&-\frac{1}{2}\int \la \nabla  {p} ,\gamma\ra\Gamma_1^z(f,f) dx=	\frac{1}{2}\int \nabla\cdot (\gamma \Gamma_1^z(f,f)){p}  dx\\
&=&\frac{1}{2}\int\nabla\cdot \gamma \Gamma_1^z(f,f){p}  dx +\int \la zz^{\ts}\nabla f, \nabla^2 f\gamma\ra {p}  dx+\frac{1}{2}\int \la \gamma,\la \nabla f, \nabla(zz^{\ts})\nabla f\ra \ra {p}  dx,
	\eeaa
	and 
	$\frac{\nabla\cdot(\pi \gamma)}{\pi}	=\la \nabla \log\pi,\gamma\ra +\nabla\cdot \gamma$.
	\qed
\end{proof}

\section{Example I: reversible SDE}\label{sec4}
This example considers an inhomogeneous stochastic differential equation (SDE).
\begin{equation}\label{VSDE-gradient}
\begin{split}
    dX_t=&\Big(-\alpha(t,X_t)\alpha(t,X_t)^{\ts}\nabla V(X_t)+\beta(t)\nabla\cdot\Big(\alpha(t,X_t)\alpha(t,X_t)^{\ts}\Big)\Big)dt\\
    &+\sqrt{2\beta(t)} \alpha(t,X_t) dB_t,
    \end{split}
\end{equation}
   where $n=d$, $m=0$, $X_t\in \mathbb R^d$, $B_t$ is a standard $d$-dimensinal Brownian motion, $\beta(t)\in\mathbb{R}^1_+$ is a positive, twice continuously differentiable, decreasing function, $V\in \mathbb C^{2}(\mathbb{R}^d;\mathbb{R})$ and $\alpha(t,x)\in \mathbb R^{d\times d}$ is a positive definite matrix function with at least twice differentiable in $x$ and differentiable in $t$. We denote $a(t,x)=\sqrt{\beta(t)}\alpha(t,x)$. And we assume that $a$ satisfies the uniform non-degenerate condition (see, e.g., \cite{KusuokaStrook}). Hence there exists a smooth density function for the solution $X_t$, denoted as ${p} (t,x)$. Furthermore, we denote $\pi(t,x)\in\mathbb{R}_+$ as a time-dependent probability density function with 
\begin{equation}\label{defn: pi}
\pi(t,x):=\frac{1}{Z(t)}e^{-\frac{V(x)}{\beta(t)}}, 
\end{equation}
where we assume that the normalization constant is finite, i.e., $Z(t)=\int_{\mathbb{R}^d} e^{-\frac{V(y)}{\beta(t)}}dy<\infty$. We note that $\pi(t,x)$ is not the stationary distribution of the SDE \eqref{VSDE-gradient}.

\subsection{Convergence analysis}
As a special case of Proposition \ref{prop: dissipation of fisher a}, with $\gamma\equiv 0$ and $z(t,x)\equiv 0$, we have the following  lemma.
\begin{lemma} \label{lemma: Fisher information decay}
For any $t\ge 0$, we have
	\bea
\pa_t\mathrm I_{a}({p} \|\pi) &=& -2\int \widetilde \Gamma_2(\log\frac{{p} }{\pi},\log\frac{{p} }{\pi}){p}  dx+\int \la \nabla\log\frac{{p} }{\pi},\pa_t(aa^{\ts})\nabla\log\frac{{p} }{\pi}\ra {p}  dx\\
	&& +2\int\big[
	\nabla\cdot(aa^{\ts}\nabla \pa_t\log\pi)+ \la \nabla \pa_t\log\pi, aa^{\ts}\nabla\log\pi\ra \big]{p}  dx.\nonumber
	\eea
\end{lemma}

As a special case of Proposition \ref{prop: Fisher information decay}, following Lemma \ref{lemma: Fisher information decay}, we have
\begin{equation*}
\widetilde \Gamma_{2}(\log\frac{{p} }{\pi},\log\frac{{p} }{\pi})\ge \mathfrak{R}(\nabla \log\frac{{p} }{\pi},\nabla \log\frac{{p} }{\pi}), 
\end{equation*}
where $\mathfrak R$ (as defined in Appendix \ref{appendix}) denotes the Ricci curvature tensor in this example with $\gamma(t,x)\equiv 0$, $z(t,x)\equiv 0$, and $a(t,x)=\sqrt{\beta(t)}\alpha(t,x)$.
We then have the following Fisher information functional decay for $\mathrm I_a(t):=\mathrm I_{a}({p} (t,\cdot)\|\pi(t,\cdot))$. 
\begin{theorem}\label{thm: Fisher information decay1}
Assume that there exists a positive function $\lambda(t)>0$, such that
\begin{equation}\label{LSI condition}
   \mathfrak R -\frac{1}{2}\partial_t(aa^{\ts}) 
    \succeq \lambda(t) aa^{\ts},
\end{equation}
for $t\ge t_0$ with some constant $t_0>0$.
Then we have
\begin{equation*}
    \mathrm I_{a}(t)\leq  e^{-2\int_{t_0}^t \lambda(r)dr}\Big(\int_{t_0}^t 2\mathrm A(r)e^{2\int_{t_0}^r\lambda(\tau)d\tau} dr+\mathrm I_a(t_0)\Big),
\end{equation*}
where $A(t)$ is a function depending on the time variable
\bea 
\mathrm A(t):=\int\big[ 
	\nabla\cdot(aa^{\ts}\nabla \pa_t\log\pi)+\la \nabla \pa_t\log\pi, aa^{\ts}\nabla\log\pi\ra \big]{p}  dx.
 \eea 
\end{theorem}
\begin{proof}
The proof follows from Theorem \ref{thm: Fisher information decay az} with $\mathcal R= \partial_t \log \pi$, and our choice of parameters for SDE \eqref{VSDE-gradient}.
\qed 
\end{proof}

\subsection{Time dependent overdamped Langevin dynamics}
In this section, we present an explicit example of the convergence result in Theorem \ref{thm: Fisher information decay1}. Consider the overdamped Langevin dynamics
\begin{equation}\label{eq:langevin}
    dX_t=-\nabla V(X_t)+\sqrt{2\beta(t)}dB_t.
\end{equation}
And the diffusion matrix $a(t,x)\in\hR^{d\times d}$ has the following form,
\bea\label{defn: beta}
a(t,x)=\sqrt{\beta(t)}\mathbb{I}, 
\eea
where $\mathbb{I}\in \mathbb{R}^{d\times d}$ is an identity matrix. 
\begin{corollary}\label{cor: Fisher information decay} Let $\beta(t)=\frac{C}{\log t}$ for some constant $C>0$, and $t\ge t_0>e$ for some constant $t_0>0$. Assume $\nabla^2_{xx}V \succeq \lambda_0\mathbb I$, for some constant $\lambda_0>0$, and $\int_{\hR^d} ( \|\nabla_x V\|^2+|\Delta_x V| ) {p} (t,x)dx \le \bar C$, for some constant $\bar C>0$. Then there exists a constant $C_0>0$, such that 
	\beaa 
	 \mathrm I_{a}({p} (t,\cdot)\|\pi(t,\cdot))\leq \frac{C_0}{t}. 
	\eeaa 
\end{corollary}
\begin{proof}
The matrix function $\mathfrak R$ defined in Appendix \eqref{appendix} is simply $\mathfrak R=\beta(t)\nabla^2_{xx}V(x)$ for equation \eqref{eq:langevin}. Applying Theorem \ref{thm: Fisher information decay1}, the Assumption \eqref{LSI condition} in Theorem \ref{thm: Fisher information decay1} is then reduced to the following condition,  
\begin{equation}\label{LSI condition a}
   {\beta(t)} \nabla^2_{xx}V(x)-\frac{1}{2}\pa_t \beta(t)\mathbb I
    \succeq \lambda(t) \beta(t) \mathbb I.
\end{equation}
For $\beta(t)=\frac{C}{\log t}$, the above condition is equivalent to
\beaa 
\frac{C}{\log t}\nabla^2_{xx}V+\frac{1}{2}\frac{C/t}{(\log t)^2}\mathbb I \succeq \lambda(t) \frac{C}{\log t}\mathbb I.
\eeaa 
Based on assumption $\nabla^2_{xx}V \succeq \lambda_0\mathbb I$ with $\lambda_0>0$, for $t\ge t_0$, and we let $\lambda(t)\equiv \lambda_0$, then
\bea \label{lambda upper bound}
\nabla^2_{xx}V+\frac{1}{2t\log t}\mathbb I\succeq  (\lambda_0+\frac{1}{2t\log t})\mathbb I\succeq \lambda_0\mathbb I=\lambda(t)\mathbb I.
\eea 
Now we turn to the estimate for $\mathrm A(t)$ in Theorem \ref{thm: Fisher information decay1}. Plugging in $\beta(t)=\frac{C}{\log t}$, we obtain 
\beaa
\nabla\cdot(aa^{\ts}\nabla \pa_t\log\pi )+\la \nabla \pa_t\log\pi, aa^{\ts}\nabla\log\pi\ra 
&=&\beta\Delta (\pa_t\log\pi)+\beta \la \nabla \pa_t\log\pi,\nabla\log e^{-V/\beta}\ra \\
&=&\frac{\pa_t\beta}{\beta} \Delta_{x}V-\frac{\pa_t\beta}{\beta^2}\| \nabla_{x}V\|^2.
\eeaa 
Applying the assumption that $\int_{\hR^d} (\|\nabla_xV\|^2+|\Delta_xV|){p} (t,x)dx\le \overline{C}$, we get 
\beaa  
\mathrm A(t)={2\int [\frac{\pa_t\beta}{\beta}\Delta_x V-\frac{\pa_t\beta}{\beta^2}\|\nabla_{x}V\|^2]{p}  dx}&\le& 2 \overline{C}( |\frac{\pa_t\beta}{\beta}| +|\frac{\pa_t\beta}{\beta^2}|)\\
&\le &2 \overline{C}( |\frac{1}{t\log t}| +|\frac{1}{Ct}|) \le  \frac{C_A}{t}, 
\eeaa 
where we denote $C_A$ as the upper bound of $\mathrm A(t)$ for $t>t_0>e$. Following the proof of Theorem \ref{thm: Fisher information decay1}, we have  
\begin{equation*}
     \frac{d}{dt}\mathrm I_a(t)\le -2\lambda_0 I_a(t)+\frac{C_A}{t}.
\end{equation*}
Hence  
\begin{equation*}
\begin{split}
    \mathrm I_a(t)\leq &e^{-2\lambda_0(t-t_0)}\Big(\int_{t_0}^t 2\frac{C_A}{r}e^{2\lambda_0(r-t_0)} dr+\mathrm I_a(t_0)\Big)\\
=&e^{-2\lambda_0(t-t_0)}\mathrm I_a(t_0)+2C_Ae^{-2\lambda_0 t}\int_{t_0}^t \frac{e^{2\lambda_0 r}}{r} dr. 
 \end{split}   
\end{equation*}
We notice that 
\begin{equation*}
   \lim_{t\rightarrow +\infty}\frac{e^{-2\lambda_0 t}\int_{t_0}^t \frac{e^{2\lambda_0 r}}{r} dr}{\frac{1}{t}}=    \lim_{t\rightarrow +\infty}\frac{t\int_{t_0}^t \frac{e^{2\lambda_0 r}}{r} dr}{e^{2\lambda_0 t}}=\lim_{t\rightarrow +\infty}\frac{\int_{t_0}^t \frac{e^{2\lambda_0 r}}{r} dr+e^{2\lambda_0 t}}{2\lambda_0 e^{2\lambda_0 t}}=\frac{1}{2\lambda_0}. 
\end{equation*}
For a sufficient small $\epsilon>0$, there exists a constant $T>0$, such that when $t>T$, 
\begin{equation*}
    e^{-2\lambda_0 t}\int_{t_0}^t \frac{e^{2\lambda_0 r}}{r} dr\leq (\frac{1}{2\lambda_0}+\epsilon)\frac{1}{t}. 
\end{equation*}
Denote $M=
\sup_{t\in [0, T]} e^{-2\lambda_0 t}\int_{t_0}^t \frac{e^{2\lambda_0 r}}{r} dr$. 
Thus, when $t_0\leq t\leq T$, we have 
\begin{equation*}
     e^{-2\lambda_0 t}\int_{t_0}^t \frac{e^{2\lambda_0 r}}{r} dr\leq M=\frac{M}{T}T\leq \frac{MT}{t}. 
\end{equation*}
Thus, there exists a constant $C_0>0$, such that 
\begin{equation*}
  \mathrm I_a(t)\leq e^{-2\lambda_0(t-t_0)}\mathrm I_a(t_0)+2\frac{C_A}{t} \max\{\frac{1}{2\lambda_0}+\epsilon, MT\} \leq \frac{C_0}{t}.  
\end{equation*}
This finishes the proof. 
\qed 
\end{proof}
Following the Fisher information decay in Corollary \ref{cor: Fisher information decay}, we get the decay of the KL divergence of the density for the dynamics \eqref{eq:langevin} as below.
\begin{corollary} \label{corolllary:KL_convergence_rate}
Under the assumptions in Theorem \ref{thm: Fisher information decay1}, for any $\tau\ge t_0$, we have 
    \begin{equation*}
    \mathrm{D}_{\mathrm{KL}}({p} (\tau)\|\pi(\tau))\le \frac{1}{2\lambda_{0}} \mathrm{I}_a({p} (\tau)\|\pi(\tau))\leq \frac{C_0}{2\lambda_0 \tau}, 
\end{equation*}
and 
\begin{equation*}
    \int_{\mathbb R^d} |{p} (\tau,x)-\pi(\tau,x)|dx\leq \sqrt{\frac{C_0}{\lambda_0\tau}}.
\end{equation*}

\end{corollary}
\begin{proof}
For any fixed $\tau \ge t_0$, we consider the standard overdamped Langevin dynamics: 
\beaa 
dX_t^\tau=-\nabla V(X_t^\tau)dt+\sqrt{2\beta(\tau)}dB_t,
\eeaa
which is equipped with the invariant measure $\pi(x;\tau)=\frac{1}{Z} e^{-\frac{V(x)}{\beta(\tau)}}$. Denote ${p} (t;\tau)$ as the density for $X_t^{\tau}$. Since $V$ is strongly convex, i.e. $\nabla^2_{xx}V\succeq \lambda_0 \mathbb I$, we have the classical log-Sobolev inequality, such that 
\begin{equation*}
    \mathrm{D}_{\mathrm{KL}}({p} (\tau)\|\pi(\tau))\le \frac{1}{2\lambda_{0}} \mathrm{I}_a({p} (\tau)\|\pi(\tau))\leq \frac{C_0}{2\lambda_0 \tau}, 
\end{equation*}
where the last inequality follows from Corollary \ref{cor: Fisher information decay}. From Pinsker's inequality, we have 
\begin{equation*}
  \int_{\mathbb{R}^d} |{p} (\tau,x)-\pi(\tau,x)|dx\leq\sqrt{2 \kl({p} (\tau)\|\pi(\tau))}\leq \sqrt{\frac{C_0}{\lambda_0\tau}}. 
\end{equation*}
\qed 
\end{proof}
\begin{remark}
When we use the upper bound of $\lambda(t)$ in \eqref{lambda upper bound} as $\frac{1}{2t\log t}$, $A(t)$ can be infinity. Thus, the current convergence analysis does not work. In other words, the choice of $\beta=\frac{C}{\log t}$ is essential for the current convergence proof, as discussed in \cite{TangZhou2021}. 
\end{remark}

\subsection{Numerics}
In this section, we perform numerical experiments to demonstrate the convergence rate in Corollary \ref{corolllary:KL_convergence_rate}. 
\begin{figure}[H]
     \centering
     \begin{subfigure}{0.32\textwidth}
         \centering
         \includegraphics[width=\textwidth]{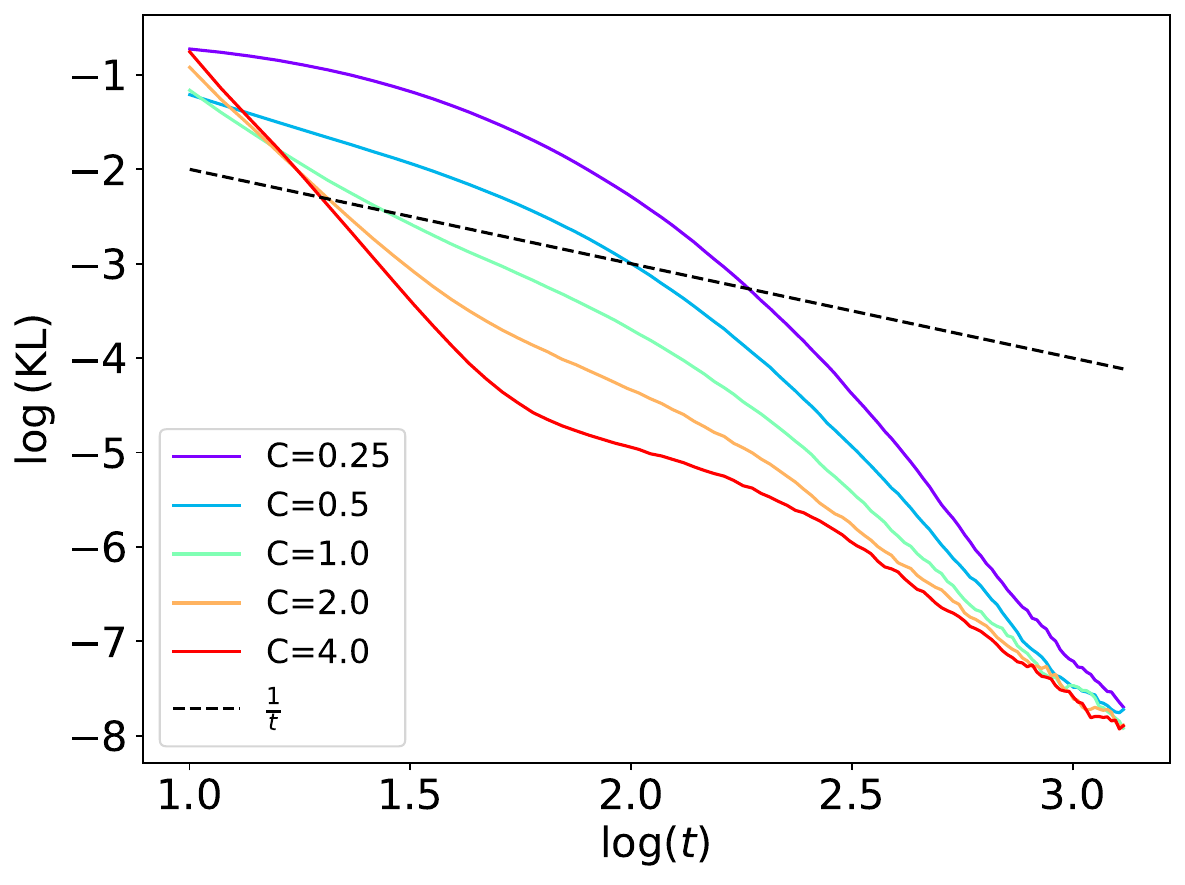}
         \caption{$V(x)=\frac{(x-1)^2}{8}$}
         \label{fig:x_square}
     \end{subfigure}
     \begin{subfigure}{0.32\textwidth}
         \centering
         \includegraphics[width=\textwidth]{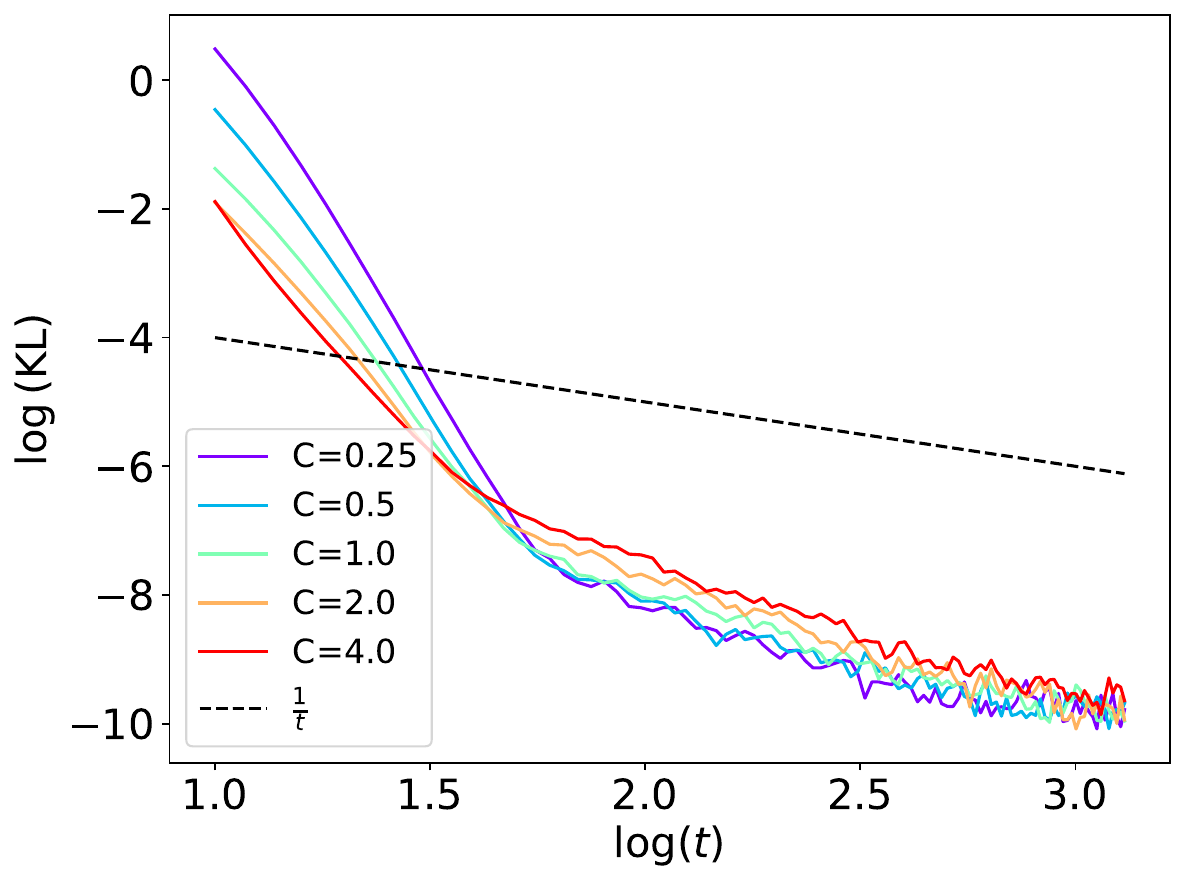}
         \caption{$V(x)=\frac{(x+1)^2}{2}-\frac{\cos(x)}{2}$}
         \label{fig:x_square_cos}
     \end{subfigure}

        \caption{Convergence rate of two strongly convex functions in one-dimension. Here y-axis represents the logarithm of the KL divergence between the empirical distribution and the invariant measure $\pi(t,x)$ given by \eqref{defn: pi}. And the x-axis is $\log(t)$. We have also added a dotted line representing $t^{-1}$ on a logarithmic scale for comparison.  }
        \label{fig:strongly_convex}
\end{figure} 

We consider $V:\RR^d \to \RR$, $d=1,2$, and $\beta(t) = \frac{C}{\log(t)}$ for some choices of constant $C$. We would like to compare the KL divergence between the invariant measure $\pi(t,x)$ given by \eqref{defn: pi} and the sample distribution of $X_t$ that follows \eqref{eq:langevin} for different choices of $V(x)$ and $\beta(t)$. We first sample $M=10^6$ particles from $\mathcal{N}(0,1)$. Then we evolve \eqref{eq:langevin} using the Euler-Maruyama scheme shown below for $N=10000$ steps with a step size of $h = 0.002$:
\begin{equation}\label{EM}
    X_{n+1} = X_n - h \nabla V(X_n) + \sqrt{\frac{2C}{\log(nh+t_0)}}B_n\,,
\end{equation}
where $B_n \sim \mathcal{N}(0,\sqrt{h})$, and $t_0 = e$. 
During each iteration, we compute the discrete KL divergence between the empirical distribution of the $M$ particles and the invariant measure $\pi(t,x)$ given by \eqref{defn: pi}. The KL divergence between two discrete distributions is given by $\mathrm{D}_{\mathrm{KL}}(p\|q)=\sum p_i \log(p_i/q_i)$. At each iteration, we can use the histogram of the empirical distribution to get $p_i$ for $i=1,...,K$. Here $K$ is the number of bins of the histogram and we choose $K=50$ in our numerical experiment. Let $x_i$ denote the location (midpoint between the left and right bin edge) of each of the bins. Then at the $n$-th iteration, we can compute $q_i = \frac{1}{Z}\exp(-\frac{V(x_i)}{\beta(nh+t_0)})$, where $Z = \int_{\RR}\exp(-\frac{V(x)}{\beta(nh+t_0)}) dx $ is the normalization constant and can be estimated numerically. The results are plotted (on a logarithmic scale) in Fig.~\ref{fig:strongly_convex} for strongly convex $V(x)$ and Fig.~\ref{fig:nonconvex} for non-convex function $V(x)$ with different constant $C$ in the expression of $\beta(t)$.   

In the strongly-convex setting (Fig.~\ref{fig:strongly_convex}), we see that the KL divergence between empirical distribution and $\pi$ decreases at a rate faster than $O(1/t)$ for all choices of $C$. In the non-convex setting (Fig.~\ref{fig:nonconvex}), we observe a convergence rate faster than $O(1/t)$ at the beginning which then drops to $O(1/t)$ as $t$ becomes larger. In two-dimension (Fig.~\ref{fig:2d}), we observe the $O(1/t)$ convergence in both the strongly convex examples (\ref{fig:x_square_2d} and \ref{fig:x_cos_2d}) and the non-convex example (\ref{fig:x_sin_2d}). In our two-dimensional examples, we used $M=10^6$ particles, $N=10000$ steps with a stepsize of $h=0.001$. We have 50 bins in both $x$ and $y$ direction which gives a total of $2500$ bins.

\begin{figure}[H]
     \centering
     \begin{subfigure}{0.32\textwidth}
         \centering
         \includegraphics[width=\textwidth]{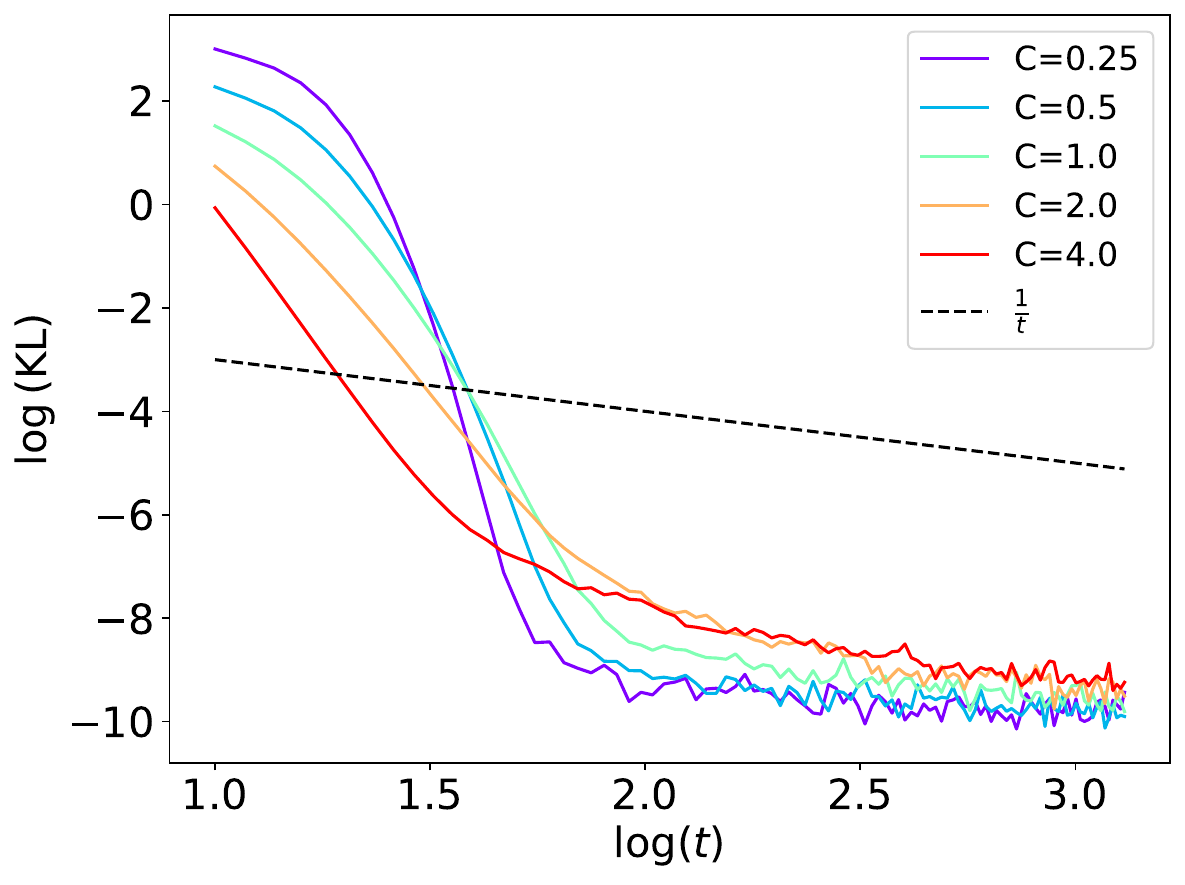}
         \caption{$V(x)=\frac{(x+2)^4}{4}-\frac{x^2}{2}+\frac{x}{8}$}
         \label{fig:x_421}
     \end{subfigure}
     \begin{subfigure}{0.32\textwidth}
         \centering
         \includegraphics[width=\textwidth]{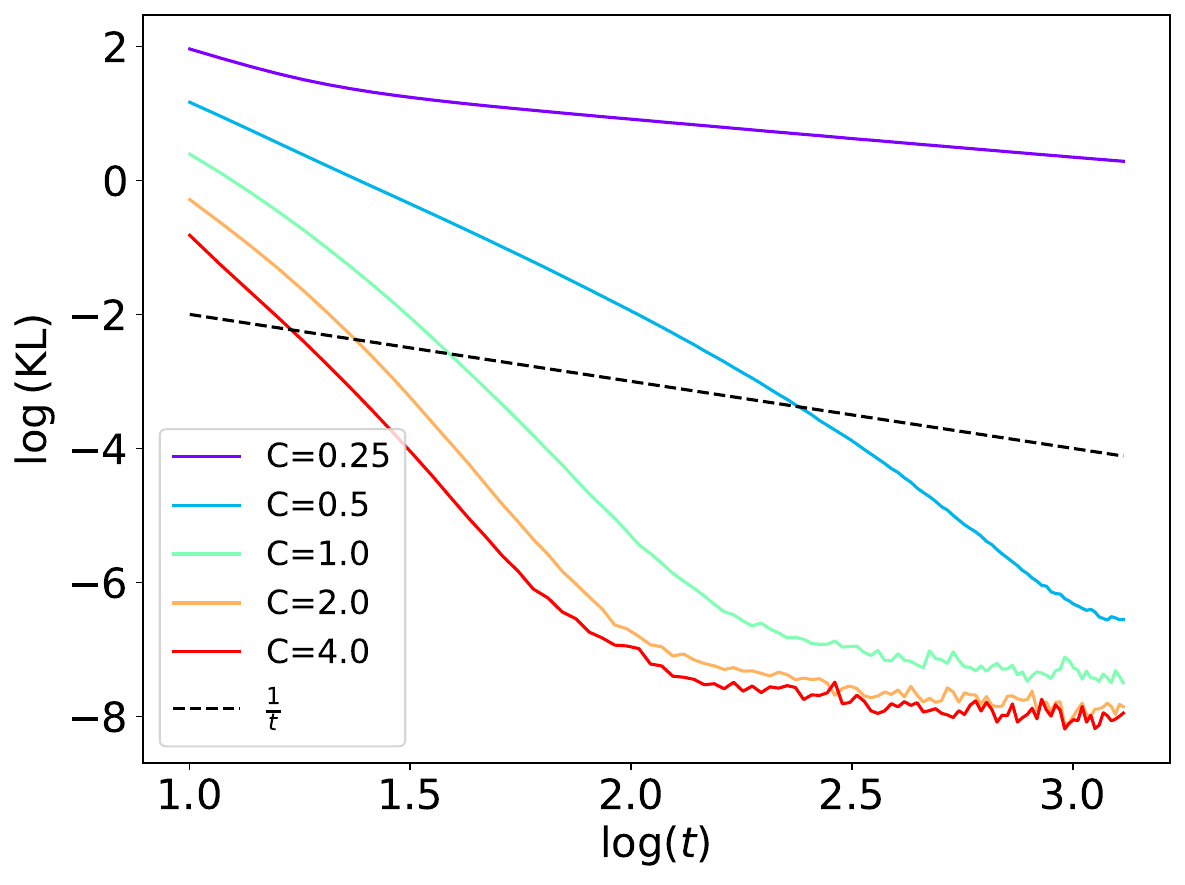}
         \caption{$V(x)=\frac{(x-2)^2}{2}-\frac{\sin(5x)}{2}$}
         \label{fig:x_square_sin}
     \end{subfigure}

        \caption{Convergence rate of two non-convex functions in one-dimension. }
        \label{fig:nonconvex}
\end{figure} 
\begin{figure}[H]
     \centering
     \begin{subfigure}{0.32\textwidth}
         \centering
         \includegraphics[width=\textwidth]{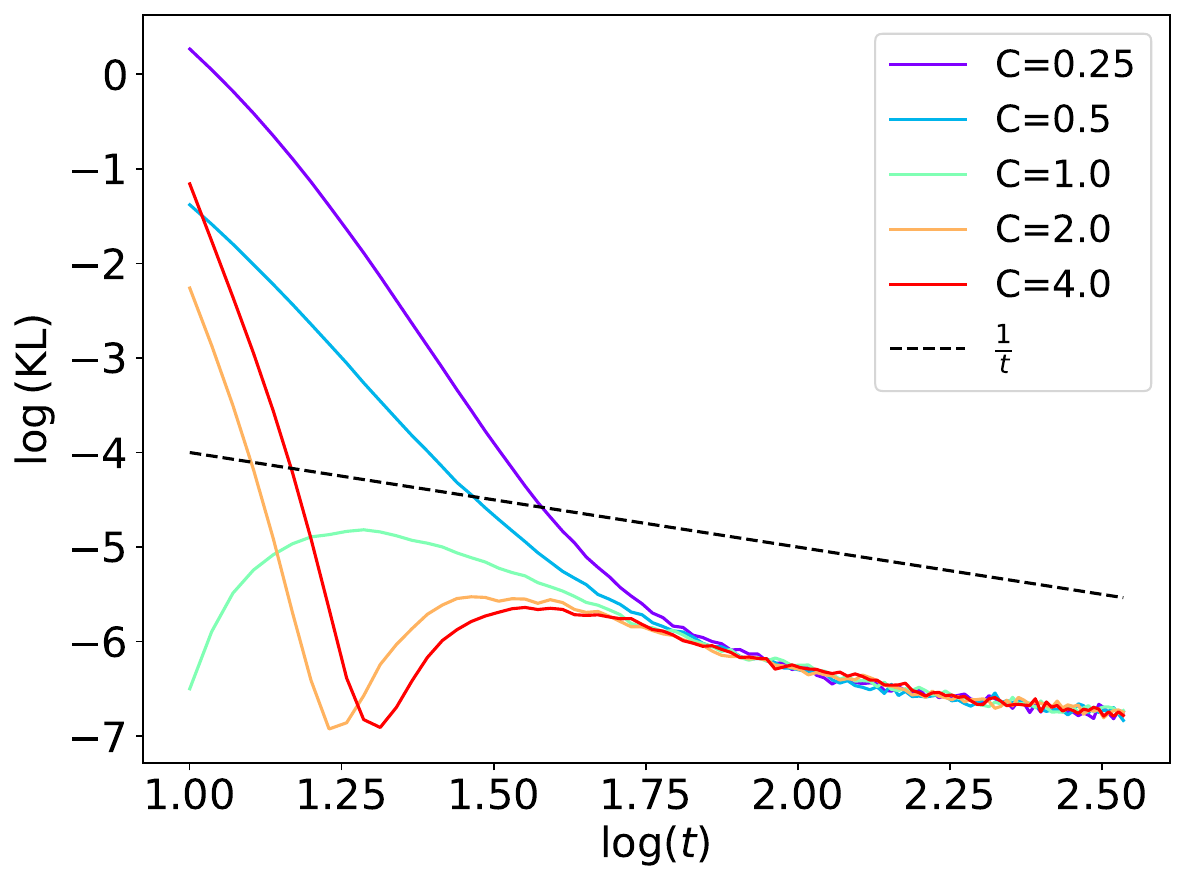}
         \caption{$\frac{x^2+y^2}{2}$}
         \label{fig:x_square_2d}
     \end{subfigure}
     \hfill
     \begin{subfigure}{0.32\textwidth}
         \centering
         \includegraphics[width=\textwidth]{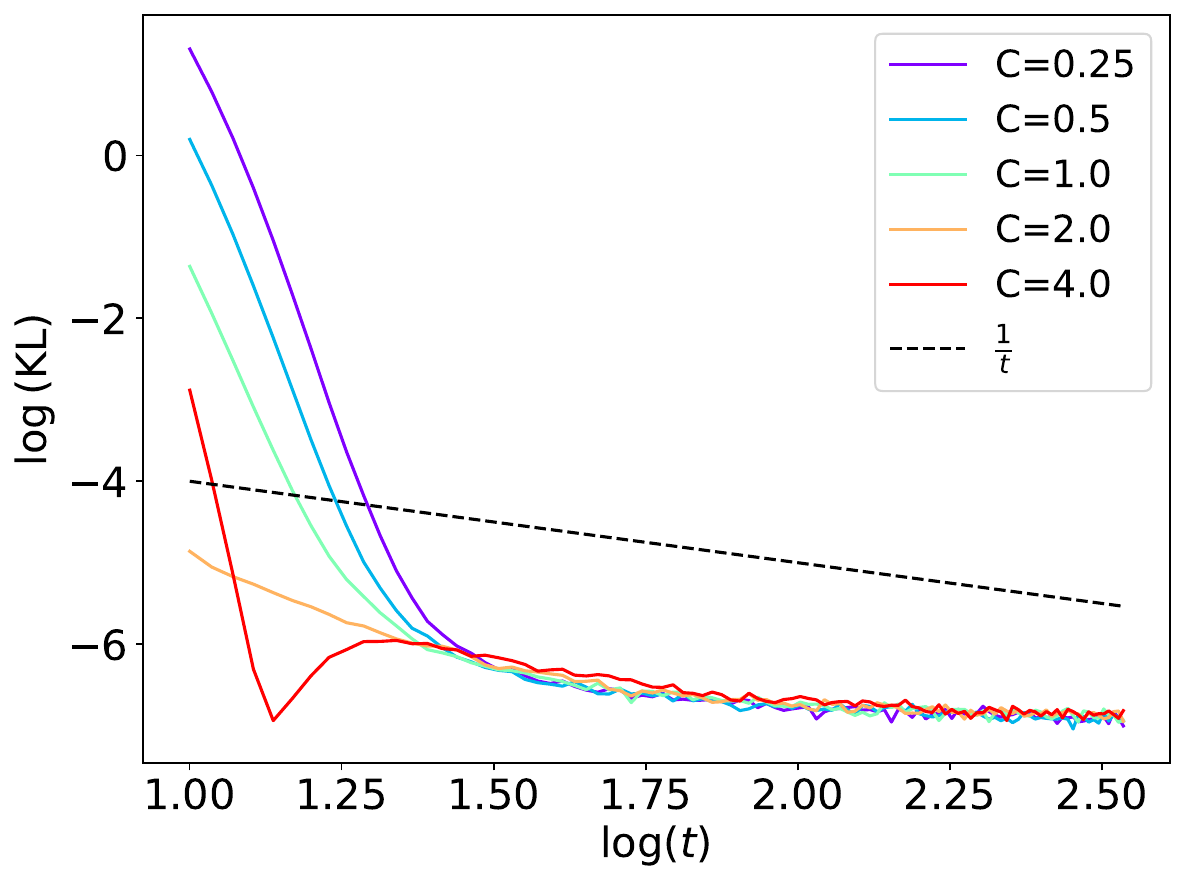}
         \caption{$x^2+y^2-\frac{\cos(x+y)}{2}$}
         \label{fig:x_cos_2d}
     \end{subfigure}
     \hfill
     \begin{subfigure}{0.32\textwidth}
         \centering
         \includegraphics[width=\textwidth]{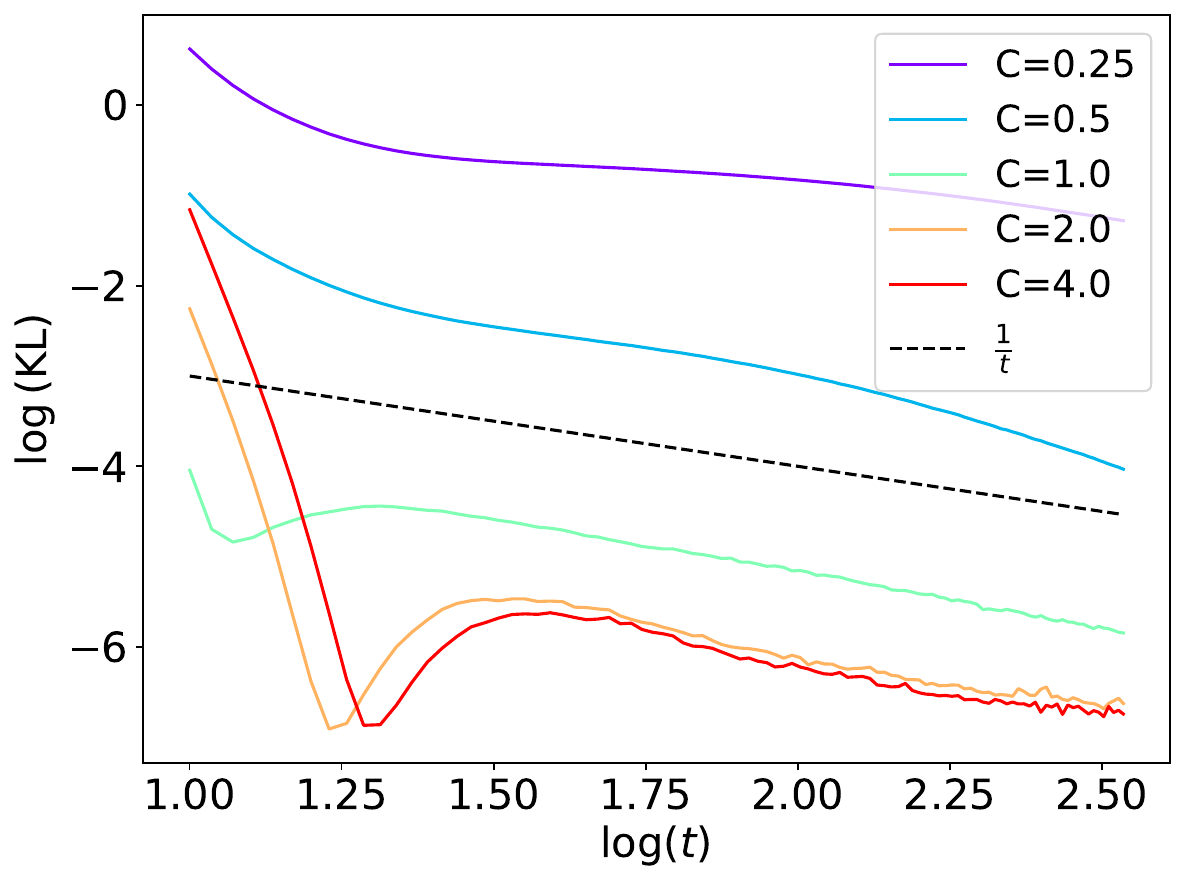}
         \caption{$ x^2+y^2-\sin(2x+2y)$}
         \label{fig:x_sin_2d}
     \end{subfigure}

        \caption{Convergence rate of two strongly convex functions (A) and (B), and a non-convex function (C) in two dimensions.} 
        \label{fig:2d}
\end{figure}

\section{Example II: non-degenerate, non-reversible SDEs}\label{sec5}
In this section, we apply Theorem \ref{thm: Fisher information decay az} to study the following non-degenerate and non-reversible SDE,
\begin{equation}\label{VSDE}
\begin{split}
    dX_t=&\Big(-\alpha(t,X_t)\alpha(t,X_t)^{\ts}\nabla V(X_t)+\beta(t)\nabla\cdot\Big(\alpha(t,X_t)\alpha(t,X_t)^{\ts}\Big)-\gamma(t,X_t)\Big)dt\\
    &+\sqrt{2\beta(t)} \alpha(t,X_t) dB_t. 
    \end{split}
\end{equation}
Again, we have $d=n$, $m=0$. The above SDE is a variant of SDE \eqref{VSDE-gradient} by adding a smooth irreversible vector field $\gamma(t,x)\in{\mathbb R}^d$, which is assumed to satisfy  
\begin{equation*}
 \nabla\cdot(e^{-V(x)}\gamma(t,x))=0.    
\end{equation*}

In the current setting, we focus on a special case with $a(t,x)=\sqrt{\beta(t)}\alpha(t,x)$. In particular, we consider the diffusion matrix $a$ in a special form, which satisfies $a_{ii}(t,x)=a_{ii}(t, x_i)=\sqrt{\beta(t)}\alpha_{ii}(x_i)>0$, for all $x_i\in \mathbb{R}$, $i=1,\cdots,n$, with   
\bea \label{non degenerate a}
\alpha(x)=\begin{pmatrix}\alpha_{11}(x_1)&0&\cdots&0\\
0&\alpha_{22}(x_2)&\cdots&0\\
0 &0&\cdots&0\\
0&\cdots&0&\alpha_{nn}(x_n)
\end{pmatrix}.
\eea

\begin{proposition}\label{prop: non deg non rev a}
    The Hessian matrix $\mathfrak R$ for the above time-dependent non-reversible SDE  has the following form,
    \beaa 
    \mathfrak R-\frac{1}{2}\pa_t(aa^{\ts})= \mathfrak R_a + \mathfrak R_{\gamma_a}-\frac{1}{2}\pa_t(aa^{\ts}),
    \eeaa 
    where 
    \bea \label{ricci curvature}
\begin{cases}
\mathfrak R_{a,ii}&=\beta(t )\alpha_{ii}^3\partial_{x_i}\alpha_{ii}\partial_{x_i}V(x) 
 +\beta(t)\alpha_{ii}^4\partial^2_{x_ix_i}V(x)-\beta^2(t)\alpha_{ii}^3\partial_{x_ix_i}^2\alpha_{ii}
, \quad \hfill  i=1,\cdots,n;\\
\mathfrak R_{a,ij}&=\beta(t)\alpha_{ii}^2\alpha_{jj}^2\partial^2_{x_ix_j}V(x), \hfill i,j=1,\cdots,n,  i\neq j;\\
\mathfrak R_{\gamma_a,ii}&=\beta(t)\gamma_i\alpha_{ii}\pa_{x_i}\alpha_{ii}-\beta(t)\pa_{x_i}\gamma_i (\alpha_{ii} )^2,  \hfill i=1,\cdots,n;\\
\mathfrak R_{\gamma_a,ij}&=-\frac{1}{2}\beta(t)[\pa_{x_i}\gamma_j (\alpha^{}_{jj})^2+ \pa_{x_j}\gamma_i (\alpha^{}_{ii})^2 ], \hfill i,j=1,\cdots,n, i\neq j.  
\end{cases}
\eea 
\end{proposition}
\begin{proof}
    Following \cite{FengLi2021}[Proposition 2],  we have 
        \bea 
\begin{cases}
\mathfrak R_{a,ii}&=  -a_{ii}^3\partial_{x_i}a_{ii}\partial_{x_i}\log \pi-a_{ii}^4\partial^2_{x_ix_i}\log \pi-a_{ii}^3\partial_{x_ix_i}^2a_{ii} , \hfill  i=1,\cdots,n;\\
\mathfrak R_{a,ij}&=-a_{ii}^2a_{jj}^2\partial^2_{x_ix_j}\log \pi, \hfill i,j=1,\cdots,n,  i\neq j.\\
\mathfrak R_{\gamma_a,ii}&=\gamma_ia^{}_{ii}\pa_{x_i}a_{ii}-\pa_{x_i}\gamma_i (a^{}_{ii} )^2,  \hfill i=1,\cdots,n;\\
\mathfrak R_{\gamma_a,ij}&=-\frac{1}{2}[\pa_{x_i}\gamma_j (a^{}_{jj})^2+ \pa_{x_j}\gamma_i (a^{}_{ii})^2 ], \hfill i,j=1,\cdots,n, i\neq j.  
\end{cases}\nonumber 
\eea 
Plugging in the matrix $a(t,x)=\sqrt{\beta(t)}\alpha(x)$, we derive the desired matrix $\mathfrak R$. \qed
\end{proof}
As in the previous section, if there exists a constant $\lambda>0$, such that 
\begin{equation*}
\mathfrak{R}-\frac{1}{2}\pa_t(aa^{\ts})\succeq \lambda aa^{\ts}, 
\end{equation*}
then the Fisher information decay in Theorem \ref{thm: Fisher information decay az} holds.
\subsection{Time dependent non-reversible Langevin dynamics 
}\label{sec:nonreversible_example}
In this section, we consider a special case with $n=2$, $\alpha\equiv \mathbb I$, and $\gamma= \frac{1}{\beta(t)}\mathsf J \nabla V$, where the matrix $\mathsf J$ has the following form, for some smooth function $c(t): \mathbb R^+\rightarrow \mathbb R$
\beaa 
\mathsf J=\begin{pmatrix} 
0&\beta(t)c(t)\\
-\beta(t) c(t)&0
\end{pmatrix}, \quad i.e.\quad \gamma(t,x)=\begin{pmatrix}c(t)\partial_{x_2}V(x)\\
-c(t)\partial_{x_1}V(x)
\end{pmatrix}.
\eeaa 
It is easy to check that $\nabla\cdot(\pi(t,x)\gamma(t,x))=0$ (e.g.: see \eqref{pi gamma =0} below). Applying Proposition \ref{prop: non deg non rev a}, we have 
\bea
\mathfrak R 
&=& \beta(t)\begin{pmatrix}
     \partial_{x_1x_1}V-c(t) \partial_{x_1x_2}V & \partial_{x_1x_2}V-c(t)\frac{1}{2} ( - \partial_{x_1x_1}V+ \partial_{x_2x_2}V ) \\
    \partial_{x_2x_1}V-c(t)\frac{1}{2} ( - \partial_{x_1x_1}V+ \partial_{x_2x_2}V )  & \partial_{x_2x_2}V +c(t)\partial_{x_2x_1}V
\end{pmatrix}\nonumber\\
&=:&\beta(t) \mathsf B(t,x).
\eea
Comparing with the Corollary \ref{cor: Fisher information decay} and Corollary \ref{corolllary:KL_convergence_rate}, the irreversible vector field $\gamma(t,x)$ only changes the matrix $\mathfrak R$, but does not change the estimate of $\mathrm A(t)$. If the smallest eigenvalue of $\mathsf B(t,x)$ is bigger than the smallest eigenvalue of $\nabla^2_{xx}V$ for a proper choice of the function $c(t)$, the convergence of stochastic dynamics \eqref{VSDE} can be faster than the underdamped Langevin dynamics \eqref{VSDE-gradient}.

\noindent
\textbf{Variable matrices $\mathsf J$.}
We also study a case with the variable coefficient anti-symmetric vector field. Consider a two-dimensional stochastic differential equation:
\begin{equation}
    \label{eq:new_J}
dX_t = (-\nabla V(X_t)- \mathsf J(t,X_t) \nabla V(X_t)-\beta(t)\nabla\cdot \mathsf J(t,X_t)) dt +\sqrt{2\beta(t)}dB_t,
\end{equation}
where we define 
\beaa 
\mathsf J=\begin{pmatrix} 
0&c(t,x)\\
- c(t,x)&0
\end{pmatrix},
\eeaa 
and 
\beaa
\gamma(t,x)&=&aa^{\ts}\nabla\log \pi -b +\beta(t)(\frac{\partial}{\partial x_j} (\alpha \alpha^{\ts})_{ij})_{i=1}^n\\
&=& \begin{pmatrix}c(t, x)\partial_{x_2}V(x)\\
-c(t,x)\partial_{x_1}V(x)
\end{pmatrix}+\beta(t)\begin{pmatrix}
    -\partial_{x_2}c(t,x)\\
    \partial_{x_1}c(t,x)
\end{pmatrix}.
\eeaa 
Here $\pi(t,x)=\frac{1}{Z(t)}e^{-\frac{V(x)}{\beta(t)}}$. We also have the fact that $\nabla\cdot(\pi \gamma)=\frac{1}{Z}\nabla\cdot(e^{-V}\gamma)=0$, since
\bea \label{pi gamma =0}
\nabla\cdot (\pi \gamma)&=& \nabla\cdot \Big(\pi \begin{pmatrix}c(t, x)\partial_{x_2}V(x)\\
-c(t,x)\partial_{x_1}V(x)
\end{pmatrix}+\pi \beta \begin{pmatrix}
    -\partial_{x_2}c(t,x)\\
    \partial_{x_1}c(t,x)
\end{pmatrix}\Big) \nonumber \\
&=& \langle \nabla\pi, \begin{pmatrix}c(t, x)\partial_{x_2}V(x)\\
-c(t,x)\partial_{x_1}V(x)
\end{pmatrix}+\beta\begin{pmatrix}
    -\partial_{x_2}c(t,x)\\
    \partial_{x_1}c(t,x)
\end{pmatrix}\rangle \nonumber\\
&&+\pi \nabla\cdot \Big( \begin{pmatrix}c(t, x)\partial_{x_2}V(x)\\
-c(t,x)\partial_{x_1}V(x)
\end{pmatrix}+\beta\begin{pmatrix}
    -\partial_{x_2}c(t,x)\\
    \partial_{x_1}c(t,x)
\end{pmatrix}\Big) \nonumber\\
&=& -\frac{\pi}{\beta}  \Big\langle \begin{pmatrix}
    \partial_{x_1} V\\
    \partial_{x_2} V
\end{pmatrix} , \begin{pmatrix}c(t, x)\partial_{x_2}V(x)\\
-c(t,x)\partial_{x_1}V(x)
\end{pmatrix}+\beta\begin{pmatrix}
    -\partial_{x_2}c(t,x) \\
    \partial_{x_1}c(t,x)
\end{pmatrix} \Big\rangle \nonumber\\
&&+\pi  \partial_{x_1}[c(t,x)\partial_{x_2}V-\beta\partial_{x_2}c(t,x)]+\pi \partial_{x_2}[-c(t,x)\partial_{x_1}V+\beta\partial_{x_1}c(t,x) ] \nonumber\\
&=& -\frac{\pi}{\beta} (c \partial_{x_1}V\partial_{x_2}V-c \partial_{x_1}V\partial_{x_2}V)+\pi (\partial_{x_1}V \partial_{x_2}c -\partial_{x_2}V\partial_{x_1}c) \nonumber\\
&&+ \pi[ \partial_{x_1}c\partial_{x_2}V+c\partial_{x_1x_2}V -\beta\partial_{x_1x_2}c-\partial_{x_2}c\partial_{x_1}V-c\partial_{x_1x_2}V+\beta\partial_{x_1x_2}c]\nonumber \\
&=0. 
\eea 
For the matrix $\mathfrak R$, with $\gamma_1= c(t,x)\partial_{x_2}V-\beta(t)\partial_{x_2}c(t,x)$, and $\gamma_2=-c(t,x)\partial_{x_1}V+\beta(t)\partial_{x_1}c(t,x)$, we have
\bea
\mathfrak R 
=&&\beta\begin{pmatrix}
    \partial_{x_1x_1}V & \partial_{x_1x_2}V\\
    \partial_{x_2x_1}V & \partial_{x_2x_2}V 
\end{pmatrix}-\beta\begin{pmatrix}
    \partial_{x_1}\gamma_1 &\frac{1}{2} ( \partial_{x_1}\gamma_2+ \partial_{x_2}\gamma_1 )  \\
\frac{1}{2} (  \partial_{x_1}\gamma_2+ \partial_{x_2}\gamma_1 )      &  \partial_{x_2}\gamma_2
\end{pmatrix}\nonumber \\
=&&\quad\beta\begin{pmatrix}
    \partial_{x_1x_1}V & \partial_{x_1x_2}V\\
    \partial_{x_2x_1}V & \partial_{x_2x_2}V 
\end{pmatrix}\nonumber\\
&&-\beta\begin{pmatrix}
    \partial_{x_1}c\partial_{x_2}V+c\partial_{x_1x_2}V-\beta\partial_{x_1x_2}c &\mathfrak R_{\gamma,12} \\
 \mathfrak R_{\gamma,12}    & -\partial_{x_2}c\partial_{x_1}V-c\partial_{x_2x_1}V+\beta\partial_{x_2x_1}c
\end{pmatrix} \nonumber,
\eea
where $\mathfrak R_{\gamma,12}=\frac{1}{2}[c(\partial_{x_2x_2}V-\partial_{x_1x_1}V)+\beta(-\partial_{x_2x_2}c+\partial_{x_1x_1}c) +\partial_{x_2}c\partial_{x_2}V-\partial_{x_1}c\partial_{x_1}V] $. 

\begin{figure}[H]
    \centering
\includegraphics[width=0.7\textwidth]{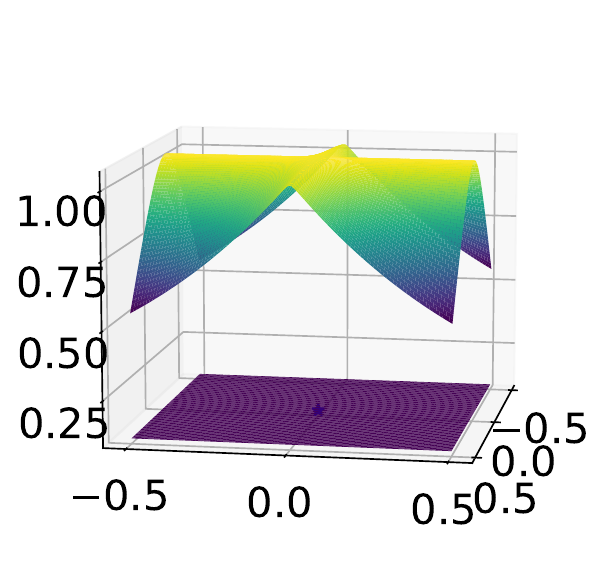}
    \caption{Eigenvalue comparison between $V(x)$ and $\mathfrak R(x)$ for $x\in [-0.5,0.5]\times[-0.5,0.5] $. The parameters are chosen as in Example \ref{ex:new J}. The yellow surface represents the smallest eigenvalue of $\mathfrak R(x)$ and the purple surface represents the smallest eigenvalue of $V(x)$. The global minimum of $V$ is marked with a blue asterisk. As shown in the figure, the smallest eigenvalue of $\mathfrak R$ is larger than that of $V$ near the global minimum.  }
    \label{fig:eig_comp}
\end{figure}
\begin{example}\label{ex:new J}
Let us consider an example where $V$ is a two dimensional quadratic form with 
$$
\nabla_{xx}^2 V = \begin{pmatrix}
    a&0\\
    0& b
\end{pmatrix},
$$
such that $a>b>0$. This implies that $\nabla_{xx}^2 V$ is positive definite. We assume that the minimum of V is at $(0,0)$. Let $c(t,x)=c(x)$ be another quadratic form such that it has the same global minimum as $V$. Denote by $c''_{ij} = \partial_{x_ix_j} c$. We now consider the neighbourhood near the global minimum, so that all first-order partial derivatives of $V$ can be neglected, and $c \approx 0$.  Then the matrix $\mathfrak R$ is approximated by 
$$
\mathfrak R \approx \beta\begin{pmatrix}
    a+\beta c''_{12} & -\frac{1}{2}\beta(c''_{11}-c''_{22}) \\
    -\frac{1}{2}\beta(c''_{11}-c''_{22}) & b-\beta c''_{12} 
\end{pmatrix}\,.
$$
There are many choices of $c$ to make the smallest eigenvalues of $\mathfrak R$ larger than that of $V$. For instance, take $c''_{11}=c''_{22}=0$, $c''_{12} =(b-a)/2\beta(t) $. In this case, the smallest eigenvalue is $(a+b)/2$ whereas the smallest eigenvalue of $V$ is $b$. A visualization is shown Fig.~\ref{fig:eig_comp}  when $a=2$, $b=0.1$, $\beta=1$. Now let us consider $\beta(t) = \frac{1}{t_0 + t}$ for $t_0 = 1$. We use the Euler-Maruyama scheme to run  \eqref{eq:langevin} and \eqref{eq:new_J} with a step size of $dt = 5\times 10^{-5}$ for $3*10^5$ iterations. We use $10^4$ particles initially sampled from a standard Gaussian distribution for our comparison. The result is demonstrated in Fig.~\ref{fig:new_J}. We observe that equation \eqref{eq:new_J} yields a faster convergence towards the global minimum than equation \eqref{eq:langevin}.
\end{example}

\begin{figure}[H]
    \centering
    \includegraphics[width=0.7\textwidth]{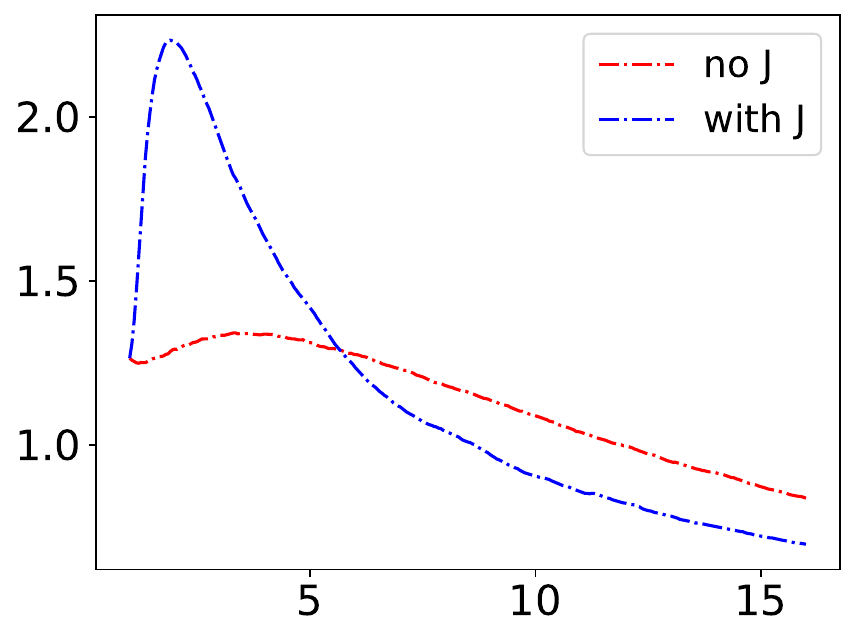}
    \caption{Convergence comparison between \eqref{eq:langevin} and \eqref{eq:new_J} in Example \ref{ex:new J}. $x$-axis represents time. $y$-axis represents the average distance of the particles to the global minimum using the two SDEs.  }
    \label{fig:new_J}
\end{figure}

\section{Example III: underdamped Langevin dynamics}\label{sec6}
In this section, we consider an underdamped Langevin dynamics with variable diffusion coefficients:
\begin{equation}\label{SDE}
\left\{\begin{aligned}
       dx_t=&v_t dt \\
   dv_t=&(- r(t,x_t) v_t-\nabla_xV(x_t))dt+\sqrt{2 r(t,x_t)}dB_t,
\end{aligned}\right.
\end{equation}
where $n=m=1$, $d=n+m=2$, $X_t=(x_t, v_t)\in\mathbb{R}^2$ is a two dimensional stochastic process, $V\in\mathbb C^2(\mathbb{R}^1)$ is 
a Lipschitz potential function with assumption $\int_{\mathbb{R}^1} e^{-V(x)}dx<+\infty$, $B_t$ is a standard Brownian motion in $\mathbb{R}$, and $r:\mathbb{R}_+\times\mathbb R^2\rightarrow \mathbb{R}_+$ is a positive smooth Lipschitz function. Indeed, the reference measure $\pi(t,x,v)=\pi(x,v)$ is the invariant measure, defined as 
\beaa 
\pi(x,v)=\frac{1}{Z}e^{-H(x,v)},\q H(x,v)=\frac{v^2}{2}+V(x),
\eeaa  
where $Z=\int_{\mathbb{R}^2} e^{-H(x,v)}dxdv<+\infty$ is a normalization constant. Following the definition of diffusion matrix $a$, the vector field $\gamma$ and the correction term $\mathcal{R}$, we have 
\bea 
a=\begin{pmatrix}
0 \\ \sqrt{r(t,x)}
\end{pmatrix},\quad
\gamma = \begin{pmatrix}
    -v\\
    \nabla V(x)
\end{pmatrix}, \quad\text{and}\quad \mathcal R(t,x,\pi)=0,
\eea 
since $\pa_t\pi(x,v)=0$, and $\nabla\cdot(\pi\gamma)=0$. 

Consider a time-dependent vector field  $z=\begin{pmatrix} z_1(t,x) \\ z_2(t,x)\end{pmatrix}$. We have the following proposition. The derivation follows similar studies in the time-independent case as shown in \cite{FengLi2021}. We skip the details here.
\begin{proposition}\label{thm: constant z}
For the time-dependent underdamped Langevin dynamics, the time-dependent Hessian matrix function $\mathfrak{R}(t,x): \mathbb R_+\times \mathbb R^2\rightarrow \mathbb{R}^{2\times 2}$ has the following form, 
\beaa 
\mathfrak{R}&=&  \mathfrak R_a+\mathfrak R_z+\mathfrak R_{\pi}-\mathfrak M_{\Lambda}+\mathfrak R_{\gamma_a}+
\mathfrak R_{\gamma_z},
\eeaa
where $a_{21}=\sqrt{r}$, and 
\beaa 
\mathfrak R_a&=&\begin{pmatrix}
0&0\\
0& -\frac{\partial^2  \log\pi }{\partial v^2} |a_{21}|^4
\end{pmatrix},\quad \mathfrak{R}_{\pi}
	=\begin{pmatrix}
0&0\\
0& C_{\pi}
\end{pmatrix}, \\
\mathfrak R_z&=&\frac{1}{2}\Big[  \begin{pmatrix}
0\\
-z^{\ts}_1\nabla((a_{21})^2 \frac{\partial  \log\pi }{\partial v})
\end{pmatrix} z^{\ts}_1+z_1\begin{pmatrix}
0& -z^{\ts}_1\nabla((a_{21})^2 \frac{\partial  \log\pi }{\partial v})
\end{pmatrix} \Big],\\ 
\mathfrak R_{\gamma_a}&=&{\frac{1}{2}}\gamma_1\nabla_1(aa^{\ts})-\frac{1}{2}\Big[ (\nabla \gamma)^{\ts}aa^{\ts}+aa^{\ts}\nabla\gamma \Big],\\ 
\mathfrak R_{\gamma_z}&=&{\frac{1}{2}}\gamma_1\nabla_1(zz^{\ts})-\frac{1}{2}\Big[(\nabla \gamma)^{\ts}zz^{\ts}+zz^{\ts}\nabla\gamma \Big],
\quad \mathfrak M_{\Lambda}=\frac{1}{(a_{21})^2} \mathsf K^{\ts}(aa^{\ts}+zz^{\ts})^{-1}\mathsf K,
\eeaa
with
\beaa\label{C pi}
C_\pi&=&2\Big[z^{\ts}_{1} z^{\ts}_{1} \nabla^2 a_{21}  a_{21}+(z^{\ts}_{1}\nabla a_{21} )^2+ (z^{\ts}_1\nabla\log\pi) [ z^{\ts}_{ 1}\nabla a_{21} a_{21}]\Big],\\
\mathsf K&=&\begin{pmatrix} \mathsf 0& 2z_1^2\pa_x[a_{21}]a_{21}-\frac{1}{2}\beta\gamma_1(a_{21})^2\\
-z_1^2\pa_x[a_{21}]a_{21}+\frac{1}{2}\beta \gamma_1(a_{21})^2&z_1z_2\pa_x[a_{21}]a_{21}\end{pmatrix}.
\eeaa 
If there exists a constant $\lambda>0$, such that 
\begin{equation*}
\mathfrak{R}-\frac{1}{2}\pa_t(aa^{\ts}+zz^{\ts})\succeq \lambda (aa^{\ts}+zz^{\ts}), 
\end{equation*}
then the Fisher information decay in Theorem \ref{thm: Fisher information decay az} holds.
\end{proposition}

In the following, we consider a special case where we choose $r(t,x)=r(t)$, and $z=\begin{pmatrix}
	z_1\\
	z_2
\end{pmatrix}$, for some constants $z_1,z_2\in\hR$. In this case, the matrix $\mathfrak R(t,x)$ is simplified into the following form, 
\beaa 
\mathfrak R= \mathfrak R_a+\mathfrak R_z+\mathfrak R_{\gamma_a}+\mathfrak R_{\gamma_z},
\eeaa 
where we have 
\beaa 
\mathfrak R_a&=&\begin{pmatrix}
	0& 0\\
	0& (r(t))^2
\end{pmatrix},\quad \mathfrak R_z=r(t)\begin{pmatrix}
	0 & \frac{z_1z_2}{2}\\
	\frac{z_1z_2}{2}& z_2^2
\end{pmatrix},\\
\mathfrak R_{\gamma_a}&=&r(t) \begin{pmatrix}
	0&\frac{1}{2}\\
	\frac{1}{2}& 0
\end{pmatrix},\quad \mathfrak R_{\gamma_z}=\begin{pmatrix}
	z_1z_2&\frac{1}{2}( z_2^2-z_1^2\nabla_{xx}^2V)\\
	\frac{1}{2}( z_2^2-z_1^2\nabla_{xx}^2V) & -z_1z_2\nabla^2_{xx}V
\end{pmatrix}.
\eeaa

\begin{proposition}[Sufficient conditions]\label{prop: underdamped} In the above example,
\beaa 
&&\mathfrak R-\frac{1}{2}\partial_t (aa^{\ts})\\
&=& \begin{pmatrix}
	 z_1z_2 & \frac{1}{2}[r(t)+r(t)z_1z_2+z_2^2-z_1^2\nabla^2_{xx}V(x) ]\\
	 \frac{1}{2}[r(t)+r(t)z_1z_2+z_2^2-z_1^2\nabla^2_{xx}V(x) ] & (r(t))^2+r(t)z_2^2-z_1z_2 \nabla^2_{xx}V(x)-\frac{1}{2}\partial_t r(t)
\end{pmatrix}.
\eeaa 
Assume that $0<\underline{\lambda} \le \pa_{xx}^2V\le \overline{\lambda}$, and there exist constants $z_2\in(0,\frac{r(t)+\sqrt{r(t)^2+4r(t)}}{2})$, for all $t\ge t_0$, such that $\underline{\lambda}$, $\overline{\lambda}$ satisfy the following conditions:
\bea \label{1d condition simple}
-\overline{\lambda}^2+[2(r(1+z_2)-z_2^2)] \underline{\lambda}-[(1-z_2)r+z_2^2]^2-2z_2\partial_tr>0,\q r^2+rz_2^2-\overline{\lambda} z_2-\frac{1}{2}\partial_t r>0.
\eea 
Then there exists a function $\lambda(t)>0$, for $t>t_0$, such that 
\bea\label{eq:lambda}
\mathfrak R-\frac{1}{2}\partial_t(aa^{\ts})\succeq \lambda (aa^{\ts}+zz^{\ts}).
\eea
\end{proposition}
\begin{proof}
 For notation convenience, we take $r=r(t)$.  It is sufficient to prove $\mathsf{det}(\mathfrak R)>0$ for $z_1z_2>0$ and  $r^2+rz_2^2-\partial^2_{xx} V z_1z_2-\frac{1}{2}\partial_t r>0$, which is equivalent to 
\beaa 
z_1z_2(r^2+rz_2^2-\partial^2_{xx} V z_1z_2-\frac{1}{2}\partial_t r)-\frac{1}{4}(rz_1z_2+z_2^2-\partial^2_{xx}V z_1^2+r)^2>0.
\eeaa 
It is equivalent to the following inequality:
\bea\label{1d condition general} 
-z_1^4(\pa_{xx}^2V)^2+[2(r(1+z_1z_2)-z_2^2)z_1^2] \pa_{xx}^2V-[(1-z_1z_2)r+z_2^2]^2-2z_1z_2\partial_t r>0.
\eea 
According to the assumption of $\pa_{xx}^2V$, it is sufficient to prove the following conditions: 
\bea\label{1d condition}
\begin{cases}
& z_1z_2>0,\q r^2+rz_2^2-\overline{\lambda} z_1z_2-\frac{1}{2}\partial_t r>0,\q   (r(1+z_1z_2)-z_2^2)>0;\\
&-z_1^4\overline{\lambda}^2+[2(r(1+z_1z_2)-z_2^2)z_1^2] \underline{\lambda}-[(1-z_1z_2)r+z_2^2]^2-2z_1z_2\partial_t r>0.
\end{cases}
\eea

Let $z_1=1$, then \eqref{1d condition simple} is equivalent to \eqref{1d condition}. We complete the proof.\qed 
\end{proof}
The next corollary estimates $\lambda$ in \eqref{eq:lambda} under some specific choices of parameters. 
\begin{corollary}
{If $z_2 = z_1 = 1$,  $\overline{\lambda}>\frac{1}{2\underline{\lambda}}+\frac{\underline{\lambda}}{2}+1$, $\beta \geq \overline{\lambda}/2$, we have $\mathfrak R-\frac{1}{2}\partial_t(aa^{\ts})\succeq 0$ as $t\to \infty$. Suppose further that $\beta = \overline{\lambda}/2$, and $\overline{\lambda}\geq\underline\lambda+2$, then we have $\lambda = \mathcal{O}(\frac{2\underline{\lambda}}{\overline{\lambda}}-\frac{1}{\overline{\lambda}^2})$. }
\end{corollary}
{\begin{proof}
    Since $r(t)=\beta+C/\log t$, we have that $r(t) \to \beta$ and $\partial_t r \to 0$ as $t\to \infty$. Denote by $u(x) =\pa_{xx}^2V(x)$ and let $\beta = \overline{\lambda}/2$. We directly compute 
    \begin{align}
        \mathsf{det}(\mathfrak R-\frac{1}{2}\partial_t(aa^{\ts})) &= \frac{\overline{\lambda}^2}{4} + \frac{\overline{\lambda}}{2}-u(x)-\frac{1}{4}(\overline{\lambda}+1-u(x))^2\nonumber \\
        &= \frac{\overline{\lambda}u(x)}{2}-\frac{u(x)}{2}-\frac{u(x)^2}{4}-\frac{1}{4}\,,
    \end{align}
    which is a quadratic function in $u(x)$. One can check that since $0<\underline{\lambda}\leq u(x) \leq \overline{\lambda}$ for all $x$, we have $ \mathsf{det}(\mathfrak R-\frac{1}{2}\partial_t(aa^{\ts}))>0$ as long as 
    $$
    \overline{\lambda}>\max\{\frac{1}{2\underline{\lambda}}+\frac{\underline{\lambda}}{2}+1,1+\sqrt{2} \} = \frac{1}{2\underline{\lambda}}+\frac{\underline{\lambda}}{2}+1\,. 
    $$
    Now let $\beta = \frac{\overline{\lambda}}{2}$. We want to find the largest $\lambda$, such that 
    $$
\mathfrak R-\frac{1}{2}\partial_t(aa^{\ts})-\lambda (aa^{\ts}+zz^{\ts}) = \begin{pmatrix}
    1-\lambda & \frac{1}{2}(\overline{\lambda}+1-u(x))-\lambda \\
    \frac{1}{2}(\overline{\lambda}+1-u(x))-\lambda & \frac{\overline{\lambda}^2}{4}+ \frac{\overline{\lambda}}{2}-u(x)-(1+\frac{\overline{\lambda}}{2})\lambda 
\end{pmatrix}\succeq 0,
    $$
    as $t\to \infty$. This translates to 
    \begin{align}
        1-\lambda &\geq 0\label{condition1} \\
        \frac{\overline{\lambda}^2}{4}+ \frac{\overline{\lambda}}{2}-u(x)-(1+\frac{\overline{\lambda}}{2})\lambda &\geq0\label{condition2} \\
        -1+2(\overline{\lambda}-1)u(x)-u(x)^2-\overline{\lambda}(\overline \lambda - 2\lambda)\lambda &\geq0 \label{condition3}
    \end{align}
    Define $f(u,\lambda) = -1+2(\overline{\lambda}-1)u(x)-u(x)^2-\overline{\lambda}(\overline \lambda - 2\lambda)\lambda$. It is clear that when $\lambda$ is fixed, $f$ is quadratic in $u$ and peaks at $u=\overline{\lambda}-1$. We also have that by definition of $u(x)$, we have $\underline \lambda \leq u(x) \leq \overline \lambda$ for all $x$.  Therefore,
    $$
    \min_u f(u,\lambda) = \min \{f(\overline{\lambda},\lambda),f(\underline\lambda,\lambda),f(\overline{\lambda}-1,\lambda)\}\,.
    $$
    When $\overline{\lambda}\geq\underline\lambda+2$, the above implies $\min_u f(u,\lambda) = f(\underline\lambda,\lambda)$. Thus \eqref{condition3} is satisfied as long as $f(\underline\lambda,\lambda)\geq 0$. We would like to maximize $\lambda$ subject to the constraints $f(\underline\lambda,\lambda)\geq 0$ together with \eqref{condition1} and \eqref{condition2}. From $\eqref{condition2}$ we have that 
    $$\lambda \leq \frac{\overline{ \lambda}}{2} - \frac{\underline \lambda}{1+\overline{ \lambda}/2}\,.$$
    Using our assumption $\overline{\lambda}\geq\underline\lambda+2$ we get that 
    $$
    \frac{\overline{ \lambda}}{2} - \frac{\underline \lambda}{1+\overline{ \lambda}/2} > 1\,.
    $$
    Therefore, \eqref{condition1} and \eqref{condition2} together imply $\lambda \leq 1$. Observe that $f(\underline \lambda,\lambda)$ is a quadratic function of $\lambda$, which produces two roots. It is straightforward to check that the larger root of $f$ is greater than 1. Hence, we conclude that $\lambda$ cannot be larger than the smaller root of $f$. We have 
    \begin{align}
        \lambda_{\rm max}  &= \frac{\overline{\lambda}}{4}- \frac{1}{4}\sqrt{\frac{8}{\overline{\lambda}}+\overline{\lambda}^2 + 16\frac{\underline\lambda}{\overline{\lambda}}-16\underline\lambda+8\frac{\underline\lambda^2}{\overline{\lambda}}}\\
        &\approx \frac{\overline{\lambda}}{4}-\frac{1}{4}\sqrt{\frac{8}{\overline{\lambda}}+\overline{\lambda}^2 -16\underline\lambda}\nonumber\\
        &\approx  \frac{2\underline{\lambda}}{\overline{\lambda}}-\frac{1}{\overline{\lambda}^2}\,,\nonumber
    \end{align}
    where our approximation holds when  $\underline\lambda/\overline{\lambda}\ll 1$. \qed
\end{proof}}

\subsection{Numerics}
We plot the convergence of \eqref{SDE} in Fig.~\ref{fig:UL_case1_strongly_convex} for strongly convex functions and in Fig.~\ref{fig:UL_case1_nonconvex} for non-convex functions. We have also plotted the KL divergence for $x$ variable only in Fig.~\ref{fig:UL_case1_strongly_convex_x} and Fig.~\ref{fig:UL_case1_nonconvex_x}. We used the same experiment setting as described in Section \ref{sec4}. In all of our numerical experiments, we observe that the KL divergence converges to 0. Comparing Fig.~\ref{fig:strongly_convex} with Fig.~\ref{fig:UL_case1_strongly_convex_x}, we observe that the convergence speed of the underdamped Langevin dynamics \eqref{SDE} has a greater dependence on the constant than overdamped Langevin dynamics \eqref{eq:langevin} does (recall that there is a constant $C$ in $\beta(t)$ in \eqref{eq:langevin} and a constant $\beta$ in $r(t)$ in \eqref{SDE}). If the constant is chosen appropriately, the underdamped Langevin dynamics could converge much faster to the invariant measure than the overdamped Langevin dynamics. In both Fig.~\ref{fig:UL_case1_strongly_convex_x} and Fig.~\ref{fig:UL_case1_nonconvex_x}, we observe oscillations of the error, which is a typical phenomenon in accelerated convex optimization methods \cite{Attouch1, Attouch2,Xinzhe}. Designing the optimal constant $\beta$ in $r(t)$ with fast convergence speed is a delicate issue that is left for future studies.
\begin{figure}[H]
     \centering
     \begin{subfigure}{0.32\textwidth}
         \centering
         \includegraphics[width=\textwidth]{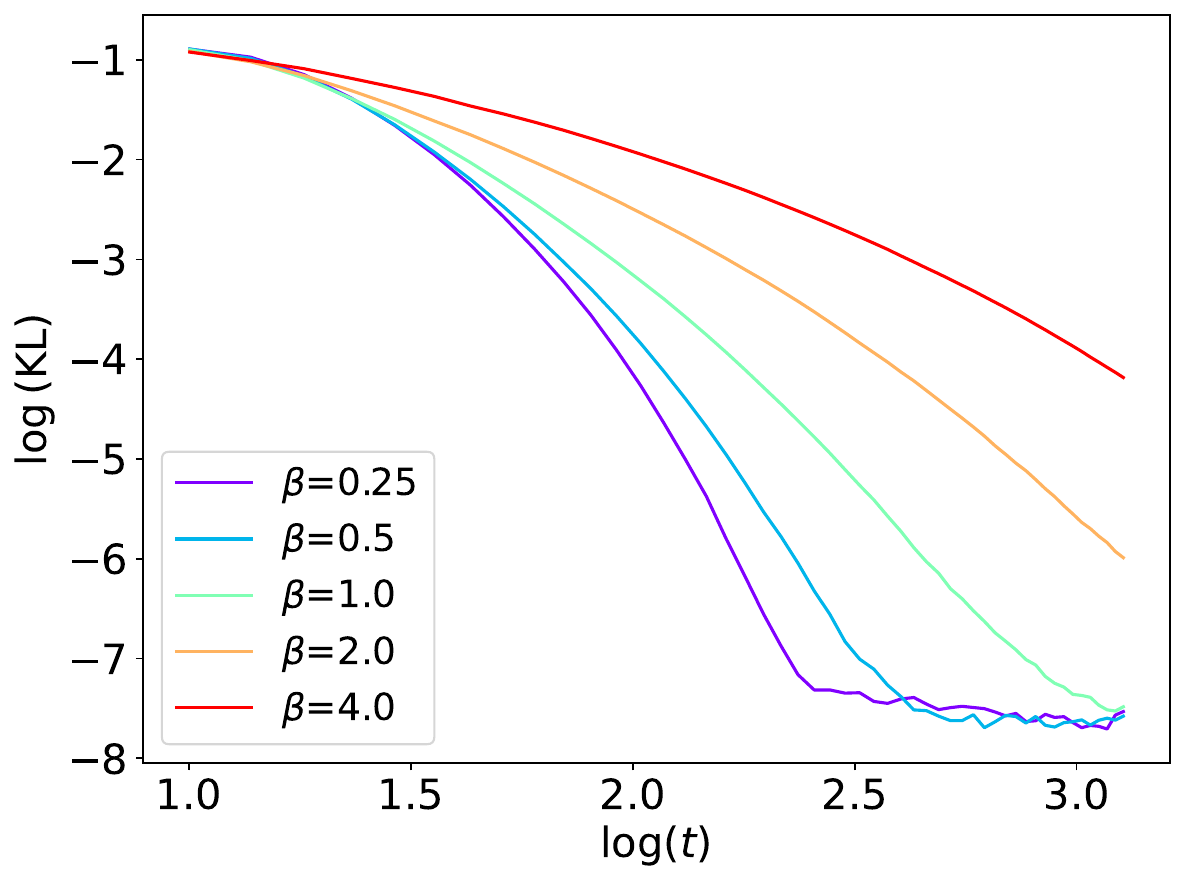}
         \caption{$V(x)=\frac{(x-1)^2}{8}$}
         \label{fig:UL_case1_x_square}
     \end{subfigure}
     \begin{subfigure}{0.32\textwidth}
         \centering
         \includegraphics[width=\textwidth]{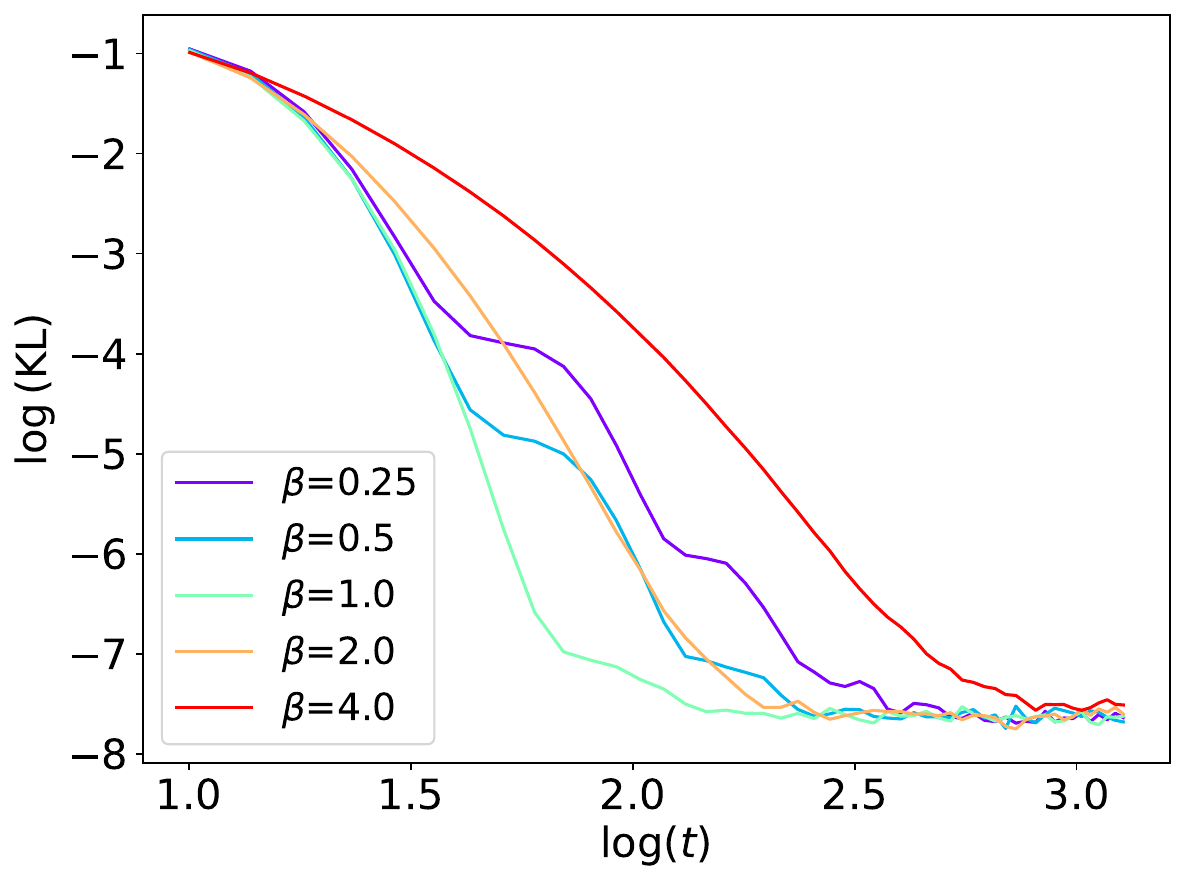}
         \caption{$V(x)=\frac{(x+1)^2}{2}-\frac{\cos(x)}{2}$}
         \label{fig:UL_case1_x_square_cos}
     \end{subfigure}

        \caption{Convergence rate of two strongly convex functions in one-dimension for \eqref{SDE} with $r(t) = \beta + 1/\log(t)$, where we measure the KL divergence in both $x$ and $v$.  }
        \label{fig:UL_case1_strongly_convex}
\end{figure} 
\begin{figure}[H]
     \centering
     \begin{subfigure}{0.32\textwidth}
         \centering
         \includegraphics[width=\textwidth]{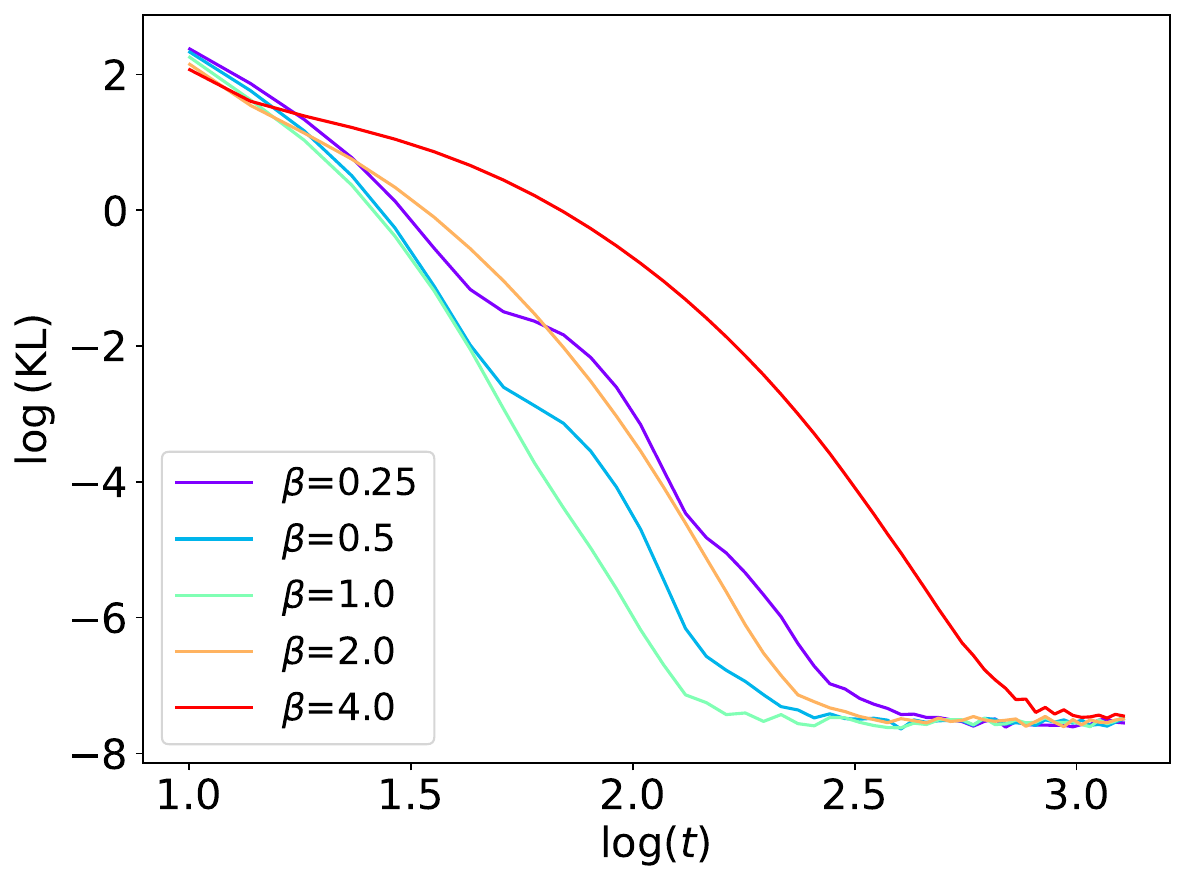}
         \caption{$V(x)=\frac{(x+2)^4}{4}-\frac{x^2}{2}+\frac{x}{8}$}
         \label{fig:UL_case1_x_421}
     \end{subfigure}
     \begin{subfigure}{0.32\textwidth}
         \centering
         \includegraphics[width=\textwidth]{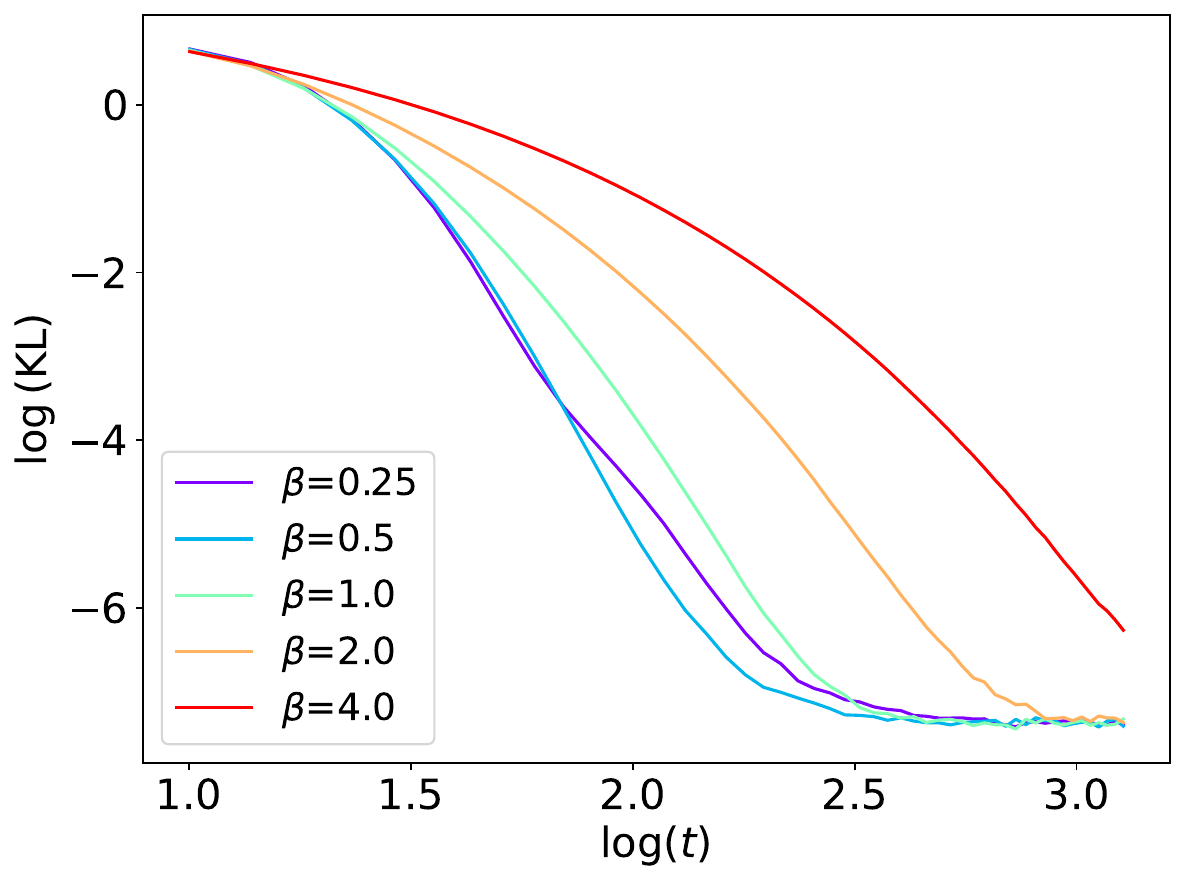}
         \caption{$V(x)=\frac{(x-2)^2}{2}-\frac{\sin(5x)}{2}$}
         \label{fig:UL_case1_x_square_sin}
     \end{subfigure}

        \caption{Convergence rate of two non-convex functions in one-dimension for \eqref{SDE} with $r(t) = \beta + 1/\log(t)$, where we measure the KL divergence in both $x$ and $v$ variables. }
        \label{fig:UL_case1_nonconvex}
\end{figure} 

\begin{figure}[H]
     \centering
     \begin{subfigure}{0.32\textwidth}
         \centering
         \includegraphics[width=\textwidth]{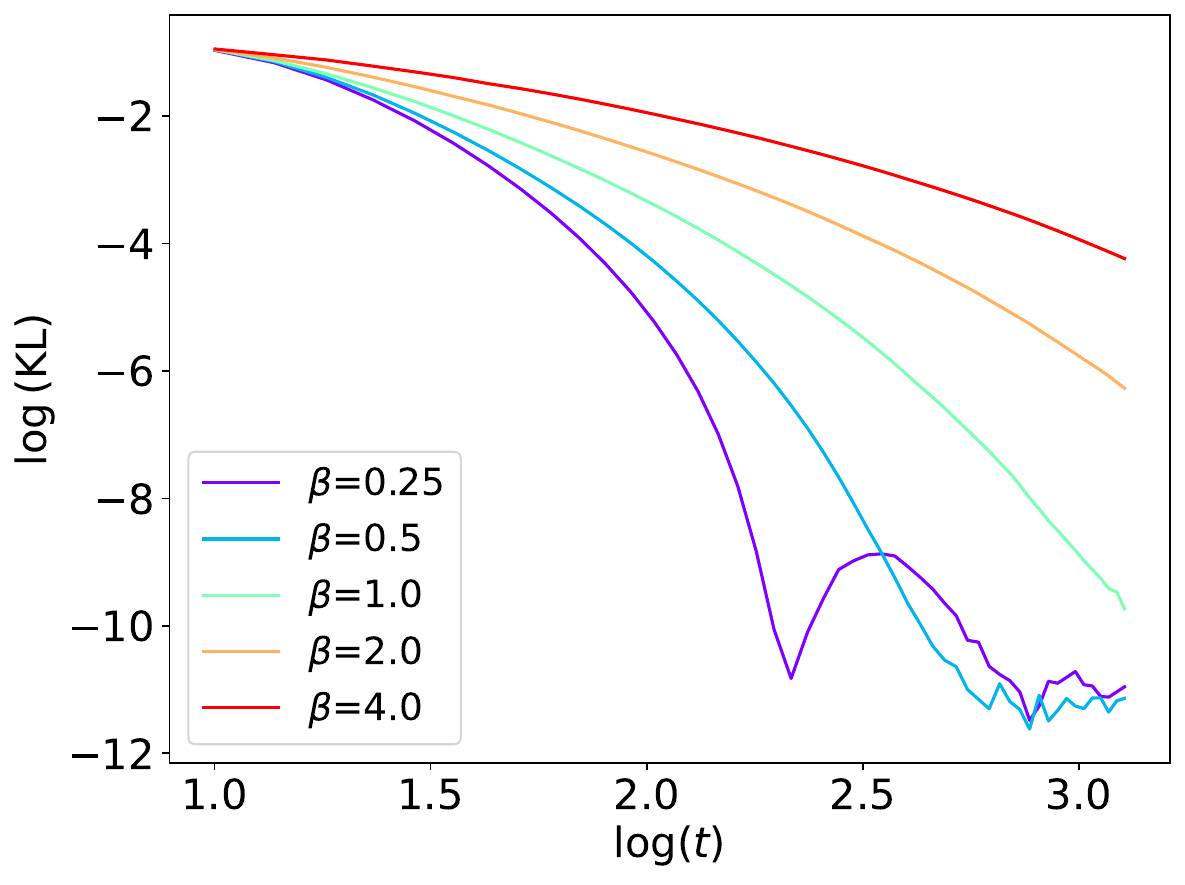}
         \caption{$V(x)=\frac{(x-1)^2}{8}$}
         \label{fig:UL_case2_x_square_x}
     \end{subfigure}
     \begin{subfigure}{0.32\textwidth}
         \centering
         \includegraphics[width=\textwidth]{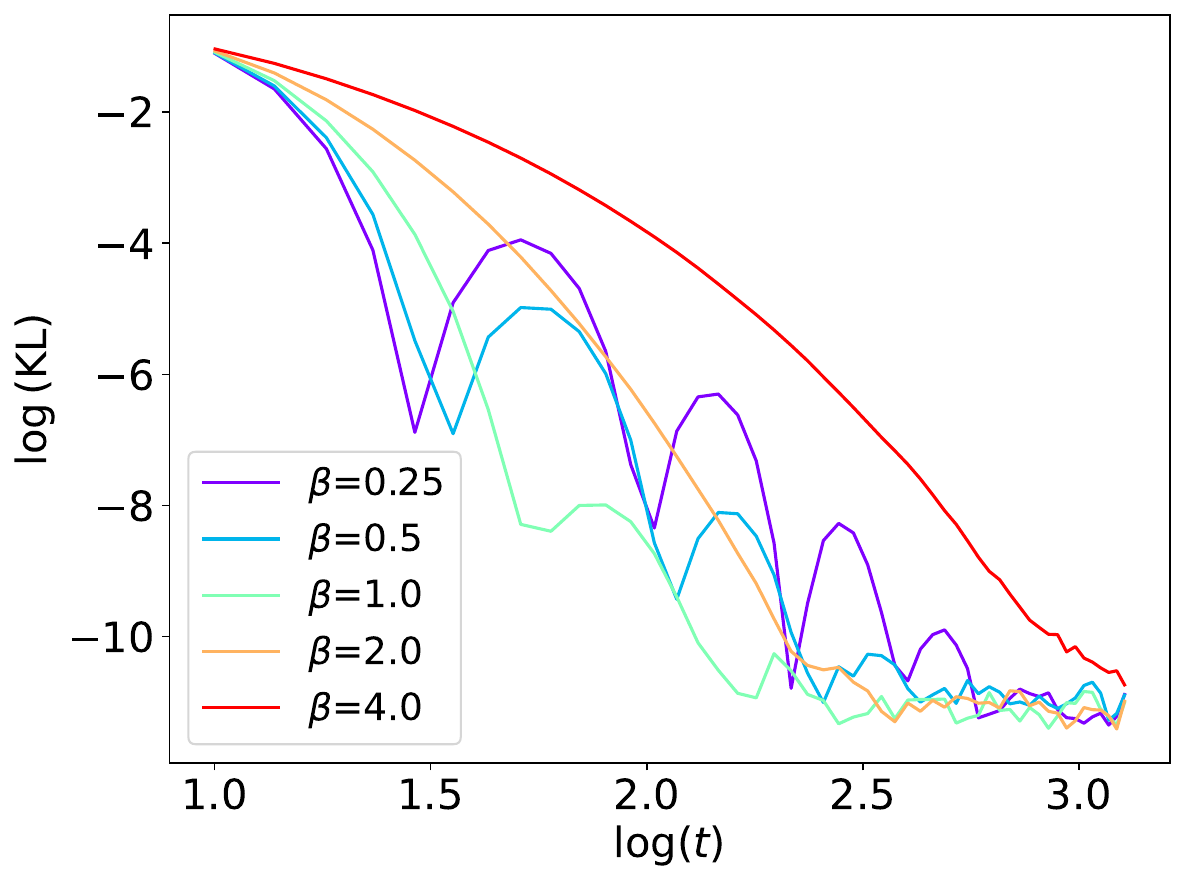}
         \caption{$V(x)=\frac{(x+1)^2}{2}-\frac{\cos(x)}{2}$}
         \label{fig:UL_case1_x_square_cos_x}
     \end{subfigure}

        \caption{Convergence rate of two strongly convex functions in one-dimension for \eqref{SDE}, where we only measure the KL divergence in the $x$ variable.  }
        \label{fig:UL_case1_strongly_convex_x}
\end{figure} 
\begin{figure}[H]
     \centering
     \begin{subfigure}{0.32\textwidth}
         \centering
         \includegraphics[width=\textwidth]{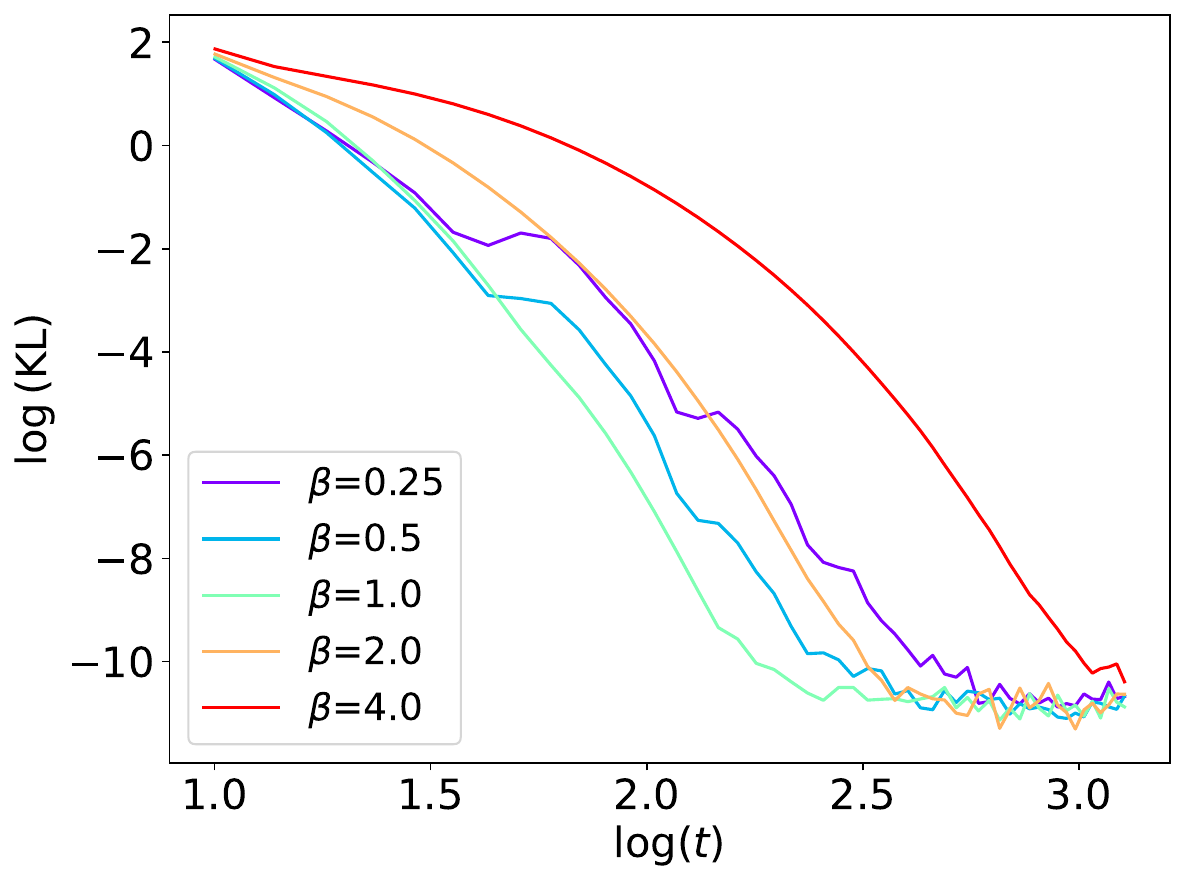}
         \caption{$V(x)=\frac{(x+2)^4}{4}-\frac{x^2}{2}+\frac{x}{8}$}
         \label{fig:UL_case1_x_421_x}
     \end{subfigure}
     \begin{subfigure}{0.32\textwidth}
         \centering
         \includegraphics[width=\textwidth]{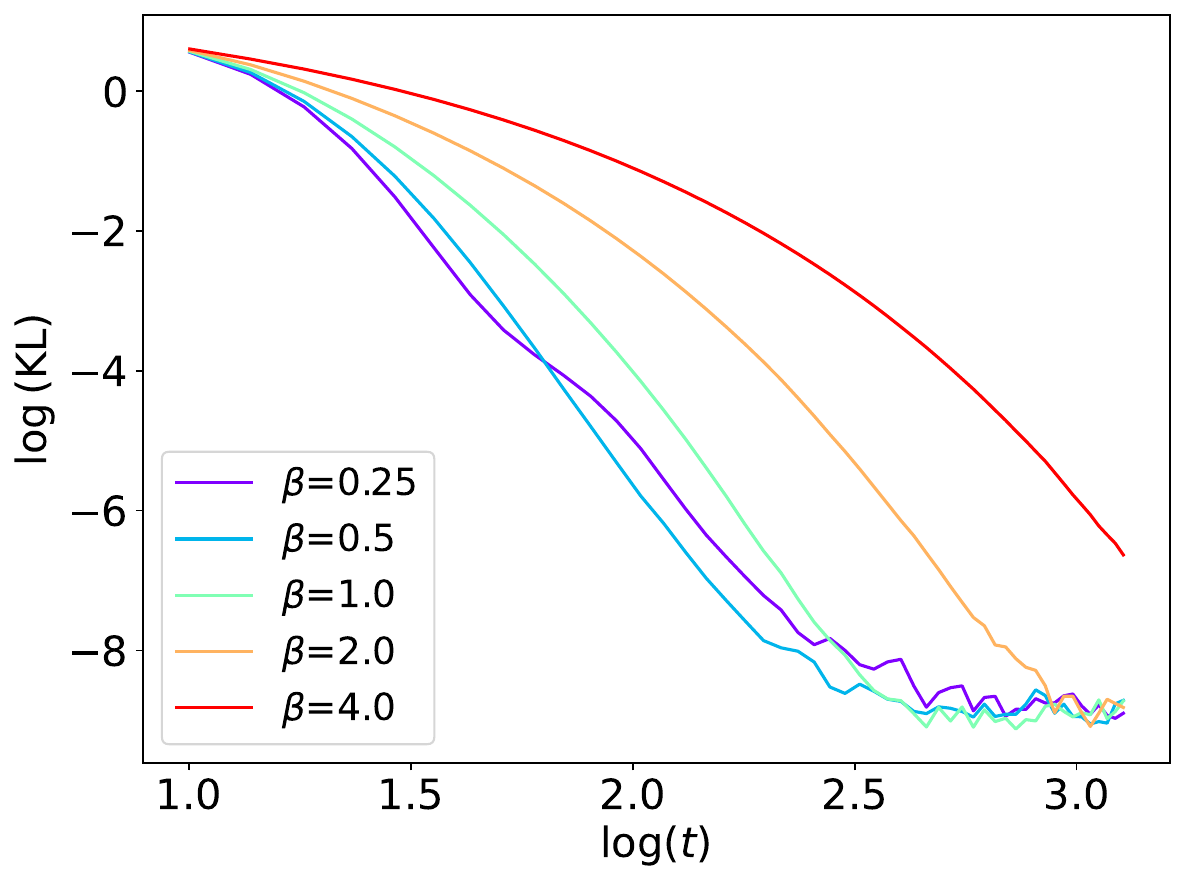}
         \caption{$V(x)=\frac{(x-2)^2}{2}-\frac{\sin(5x)}{2}$}
         \label{fig:UL_case1_x_square_sin_x}
     \end{subfigure}

        \caption{Convergence rate of two non-convex functions in one-dimension for \eqref{SDE}, where we only measure the KL divergence in the $x$ variable.}
        \label{fig:UL_case1_nonconvex_x}
\end{figure} 

\section{Discussion}
This paper studies the convergence analysis of time-dependent stochastic dynamics. We obtain a time-dependent Hessian matrix condition, which characterizes the convergence behavior of stochastic dynamics in terms of generalized Fisher information functionals.
Examples of convergence speeds are shown, including over-damped, irreversible drift and degenerate diffusion, and underdamped Langevin dynamics. We also present several numerical experiments to verify the current convergence analysis of general stochastic dynamics. 

In future work, we shall investigate the ``optimal'' choice of time-dependent matrix function $a$ and vector field $\gamma$ to find the global minimizer of a non-convex function $V$. Here, the ``optimal'' is in the sense of fast convergence speed towards the global minimizer. 
However, as we see in this paper, the convergence analysis for stochastic algorithms is more delicate than their deterministic counterparts. This requires us to estimate the general Hessian matrix, a.k.a. Ricci curvature lower bound, from both diffusion matrices $a$ and non-gradient vector from $\gamma$. They depend on the second derivatives of coefficients in stochastic dynamics. The other practical issue is the estimation of step sizes in the Euler-Maruyama scheme \eqref{EM}. The related discrete-time convergence analysis of stochastic algorithms is left in future studies. 


\section*{Appendix}\label{appendix}
The time-independent version of the Hessian matrix is first introduced in  \cite{FengLi2021}[Definition 1]. For completeness of this paper, we introduce the time-dependent version of it for matrices $a(t,x)$ and $z(t,x)$, and we take the interpolation parameter $\beta=0$ for \cite[Definition 1]{FengLi2021}, since we do not always have $\nabla\cdot(\pi(t,x)\gamma(t,x))=0$, which can be seen in Lemma \ref{lemma a: 2} and Lemma \eqref{lemma z: 2}. This is a major difference compared to \cite[Proposition 9]{FengLi2021}.
\begin{definition}[Hessian matrix]\label{def: curvature sum}
Let matrices $a(t,x)$ and $z(t,x)$ satisfy the H\"ormander like condtion, and conditions \eqref{full rank condition}, \eqref{condition: bochner}. We define a bilinear form associated with SDE \eqref{SDE setting}, and matrices $a,z$ as below, for a smooth vector field $\mathsf{U}\in C^{\infty}(\hR^{n+m};\hR^{n+m})$, 
\bea\label{defn: curvature tensor}
\mathfrak{R}(\mathsf{U},\mathsf{U})
&=& (\mathfrak R_a+\mathfrak R_z+\mathfrak R_{\pi}+ \mathfrak R_{\gamma_a}+\mathfrak R_{\gamma_z})(\mathsf{U},\mathsf{U})-\Lambda_1^{\ts}\Lambda_1-\Lambda_2^{\ts}\Lambda_2+\mathsf D^{\ts}\mathsf D+\mathsf E^{\ts}\mathsf E.
\eea 
We define $\mathfrak{R}(t, x): \mathbb R_+\times \mathbb R^{n+m}\rightarrow \mathbb{R}^{(n+m)\times (n+m)}$ as the corresponding time dependent matrix function such that
\bea 
\mathsf U^\ts\mathfrak{R}(x)\mathsf U=\mathfrak{R}(\mathsf U, \mathsf U),
\eea 
for all vector fields  $\mathsf U$. The bilinear forms in \eqref{defn: curvature tensor}
 are defined as below.  
\bea\label{tensor a}\mathfrak R_{a}(\mathsf{U},\mathsf{U})&=& \sum_{i,k=1}^n a^{\ts}_i\nabla a^{\ts}_i\nabla  a^{\ts}_k \mathsf{U} (a^{\ts}_k\mathsf{U})+\sum_{i,k=1}^n a^{\ts}_i a^{\ts}_i\nabla^2a^{\ts}_k \mathsf{U} (a^{\ts}_k\mathsf{U})\nonumber  \\
	&& -\sum_{i,k=1}^n a^{\ts}_k\nabla a^{\ts}_i\nabla a^{\ts}_i\mathsf{U}(a_k^{\ts}\mathsf{U})-\sum_{i,k=1}^n a^{\ts}_k a^{\ts}_i \nabla^2a^{\ts}_i\mathsf{U}(a_k^{\ts}\mathsf{U})\nonumber \\
&&+\sum_{i=1}^n\sum_{\hat k=1}^{n+m}\Big[ (aa^{\ts} \nabla\log \pi)_{\hat k} \nabla_{\hat k} a^{\ts}_i\mathsf{U}-a_i^{\ts}\nabla (aa^{\ts}\nabla\log \pi)_{\hat k} \mathsf U_{\hat k}\Big]a^{\ts}_i\mathsf{U} \nonumber \\
	&&+\nabla a\circ \Big(\sum_{k=1}^n\Big[ a^{\ts}\nabla a^{\ts}_k\mathsf{U}-a^{\ts}_k \nabla a^{\ts}\mathsf{U})\Big] a^{\ts}_k\mathsf{U}\Big)-\la  \left(a^{\ts} \nabla^2 a\circ  (a^{\ts}\mathsf{U})\right),a^{\ts}\mathsf{U}\ra_{\hR^n},\nonumber\\
  \label{tenser z}\mathfrak R_{z}(\mathsf{U},\mathsf{U})&=& \sum_{i=1}^n\sum_{k=1}^m a^{\ts}_i\nabla a^{\ts}_i\nabla  z^{\ts}_k \mathsf{U} (z^{\ts}_k\mathsf{U})+\sum_{i,k=1}^n a^{\ts}_i a^{\ts}_i\nabla^2z^{\ts}_k \mathsf{U} (z^{\ts}_k\mathsf{U}) \nonumber \\
	&& -\sum_{i=1}^n\sum_{k=1}^m z^{\ts}_k\nabla a^{\ts}_i\nabla a^{\ts}_i\mathsf{U}(z_k^{\ts}\mathsf{U})-\sum_{i,k=1}^n z^{\ts}_k a^{\ts}_i \nabla^2a^{\ts}_i\mathsf{U}(z_k^{\ts}\mathsf{U})\nonumber \\
&&+\sum_{k=1}^m\sum_{\hat k=1}^{n+m}\Big[ (aa^{\ts} \nabla\log \pi)_{\hat k} \nabla_{\hat k} z^{\ts}_k\mathsf{U}-z_k^{\ts}\nabla (aa^{\ts}\nabla\log \pi)_{\hat k} \mathsf U_{\hat k}\Big]z^{\ts}_k\mathsf{U}\nonumber \\
	&&+\nabla a\circ \Big(\sum_{k=1}^m\Big[ a^{\ts}\nabla z^{\ts}_k\mathsf{U}-z^{\ts}_k \nabla a^{\ts}\mathsf{U})\Big] z^{\ts}_k\mathsf{U}\Big)-\la  \left(z^{\ts} \nabla^2 a\circ  (a^{\ts}\mathsf{U})\right),z^{\ts}\mathsf{U}\ra_{\hR^m},\nonumber\\ 
\label{tensor R Psi}\mathfrak{R}_{\pi}(\mathsf{U},\mathsf{U})
	&=&2\sum_{k=1}^m \sum_{i=1}^n\left[\nabla z^{\ts}_{k} z^{\ts}_{k} \nabla a^{\ts}_{i}\mathsf{U} a^{\ts}_{i}\mathsf{U}+z^{\ts}_{k}\nabla z^{\ts}_{k} \nabla a^{\ts}_{i}\mathsf{U} a^{\ts}_{i}\mathsf{U}+z^{\ts}_{k} z^{\ts}_{k} \nabla^2 a^{\ts}_{i} \mathsf{U}   a^{\ts}_{i}\mathsf{U}   \right]\nonumber \\
&&+2\sum_{k=1}^m \sum_{i=1}^n\Big[(z^{\ts}_{k}\nabla a^{\ts}_i\mathsf{U} )^2+ (z^{\ts}\nabla\log\pi)_k \left[ z^{\ts}_{ k}\nabla a^{\ts}_i\mathsf{U} a^{\ts}_i\mathsf{U} \right]\Big] \nonumber\\
&&-2\sum_{j=1}^m\sum_{l=1}^n\left[ \nabla a^{\ts}_{l} a^{\ts}_{l}\nabla  z^{\ts}_{j}\mathsf{U} z^{\ts}_{j}\mathsf{U}+ a^{\ts}_{l}\nabla a^{\ts}_{l}\nabla z^{\ts}_{j}\mathsf{U} z^{\ts}_{j}\mathsf{U}+a^{\ts}_{l}a^{\ts}_{l} \nabla^2z^{\ts}_{j}\mathsf{U} z^{\ts}_{j} \mathsf{U} \right] \nonumber\\
&&-2\sum_{j=1}^m\sum_{l=1}^n \left[ (a^{\ts}_{l}\nabla z^{\ts}_j \mathsf{U})^2  +(a^{\ts}\nabla\log\pi)_l \Big[ a^{\ts}_{l}\nabla z^{\ts}_{j}\mathsf{U} z^{\ts}_{j}\mathsf{U}\Big]  \right], \nonumber\\ 
\mathfrak R_{\gamma_a}(\mathsf{U},\mathsf{U})&=& \frac{1}{2}\sum_{\hat k=1}^{n+m}\gamma_{\hat k} \la \mathsf U,\nabla_{\hat k}(aa^{\ts})\mathsf U\ra -  \la \nabla \gamma \mathsf{U},aa^{\ts}\mathsf{U}\ra_{\mathbb R^{n+m}},\nonumber\\ 
\mathfrak R_{\gamma_z}(\mathsf{U},\mathsf{U})&=& \frac{1}{2}\sum_{\hat k=1}^{n+m}\gamma_{\hat k} \la \mathsf U,\nabla_{\hat k}(zz^{\ts})\mathsf U\ra -  \la \nabla \gamma \mathsf{U},zz^{\ts}\mathsf{U}\ra_{\mathbb R^{n+m}}.\nonumber
\eea
 We define vector functions $\mathsf D:\hR^{n+m}\rightarrow \mathbb R^{n^2\times 1}$, and $\mathsf E:\hR^{n+m}\rightarrow \mathbb R^{(n\times m)\times 1}$ as below,
\bea\label{vector D E}
\mathsf D_{ik}= \sum_{\hat i, \hat k=1}^{n+m}a^{\ts}_{i\hat i}\pa_{x_{\hat i}} a^{\ts}_{k\hat k}\mathsf U_{\hat k}, \quad \mathsf E_{ik}=\sum_{\hat i,\hat k=1}^{n+m}a^{\ts}_{i\hat i}\pa_{x_{\hat i}} z^{\ts}_{k\hat k} \mathsf U_{\hat k}.
\eea 
For $\beta\in\mathbb R$, the vector functions $\Lambda_1: \hR^{n+m}\rightarrow \mathbb R^{n^2\times 1}$ and $\Lambda_2:\hR^{n+m}\rightarrow \mathbb R^{(n\times m)\times 1}$ are defined as, for $i,l\in\{1,\cdots,n\}$, 
\beaa 
(\Lambda_1)_{il}&=&\sum_{k=1}^n[\sum_{i'=1}^{n+m} a^{\ts}_{ii'}\lambda^{i'k}_l-\sum_{k'=1}^{n+m}a^{\ts}_{kk'} \lambda^{k'i}_l ]a^{\ts}_k\mathsf U+ \sum_{k=1}^m\Big(\sum_{i'=1}^{n+m}  a^{\ts}_{ii'}\omega^{i'k}_l -\sum_{k'=1}^{n+m} z^{\ts}_{kk'}\lambda^{k'i}_l\Big)z^{\ts}_k\mathsf U \\
&&- \sum_{k=1}^m \sum_{i'=1}^{n+m} a^{\ts}_{ii'} \omega^{i'k}_l  z^{\ts}_k\mathsf U -\frac{\beta}{2} \alpha_l (a^{\ts}_{ i}\mathsf U)+\frac{\beta}{2}   \la \mathsf U,\gamma \ra \mathbf{1}_{\{i=l\}}+\mathsf D_{il},
\eeaa 
and for $i\in \{1,\cdots,n\}$, $l\in\{1,\cdots,m\}$, 
\beaa (\Lambda_2)_{il}
&=&\sum_{k=1}^n[ \sum_{i'=1}^{n+m} a^{\ts}_{ii'} \lambda ^{i'k}_{l+n}-\sum_{k'=1}^{n+m}a^{\ts}_{kk'}\lambda^{k'i}_{l+n}]a^{\ts}_k\mathsf U+\sum_{k=1}^m\Big(\sum_{i'=1}^{n+m}  a^{\ts}_{ii'}\omega^{i'k}_{l+n} -\sum_{k'=1}^{n+m} z^{\ts}_{kk'} \lambda_{l+n}^{k'i}\Big)z^{\ts}_k\mathsf U \\
&&+\sum_{k=1}^n\sum_{k'=1}^{n+m}  z^{\ts}_{lk'} \lambda^{k'k}_i a^{\ts}_{k} \mathsf U+ z^{\ts}_{l}\nabla a^{\ts}_i \mathsf U-\sum_{k=1}^m\sum_{i'=1}^{n+m}a^{\ts}_{ii'} \omega^{i'k}_{l+n}  z^{\ts}_k\mathsf U -a^{\ts}_{i}\nabla z^{\ts}_{l} \mathsf U-\frac{\beta}{2} \alpha_{l+n}(a^{\ts}_{ i}\mathsf U)+\mathsf E_{il}.
\eeaa 
For each indices $i, k, \hat k$, assume that there exist smooth functions $\lambda^{i'k}_{l}$, $\omega^{i'k}_l$ and $\alpha_l$ for $l=1,\cdots,n+m$, 
\beaa 
\nabla_{i'}a^{\ts}_{k\hat k}&=&\sum_{l=1}^n\lambda^{i'k}_la^{\ts}_{l\hat k}+\sum_{l=1}^{m}\lambda^{i'k}_{l+n} z^{\ts}_{l\hat k},\quad 
 \nabla_{i'}z^{\ts}_{k\hat k}=\sum_{l=1}^n\omega^{i'k}_la^{\ts}_{l\hat k}+\sum_{l=1}^{m}\omega^{i'k}_{l+n} z^{\ts}_{l\hat k},
 \eeaa 
 and $\gamma_{\hat k}=\sum_{l=1}^n\alpha_la^{\ts}_{l\hat k}+\sum_{l=1}^m\alpha_{l+n}z^{\ts}_{l\hat k}$. For a vector function $\gamma\in\mathbb R^{n+m}$,  we define $\nabla\gamma\in \mathbb R^{(n+m)\times (n+m)}$ with $(\nabla\gamma)_{ij}=\nabla_i\gamma_j$.
\end{definition}

\end{document}